\documentclass[11pt,a4paper]{article}

\usepackage[T1]{fontenc}
\usepackage{lmodern}
\usepackage[margin=3cm]{geometry}
\usepackage{amsmath,amssymb,amsthm,mathtools}
\usepackage{mathrsfs}
\usepackage{microtype}
\allowdisplaybreaks
\usepackage{enumitem}
\usepackage{needspace}
\usepackage[hidelinks,pdfusetitle]{hyperref}

\newtheorem{theorem}{Theorem}[section]
\newtheorem{proposition}[theorem]{Proposition}
\newtheorem{lemma}[theorem]{Lemma}
\newtheorem{corollary}[theorem]{Corollary}
\theoremstyle{definition}
\newtheorem{definition}[theorem]{Definition}
\theoremstyle{remark}
\newtheorem{remark}[theorem]{Remark}

\newcommand{\Q}{\mathbb Q}
\newcommand{\Z}{\mathbb Z}
\newcommand{\R}{\mathbb R}
\newcommand{\C}{\mathbb C}
\newcommand{\A}{\mathbb A}
\newcommand{\Pj}{\mathbb P}
\newcommand{\cO}{\mathcal O}
\newcommand{\Kc}{\mathcal K}
\newcommand{\cG}{\mathcal G}
\newcommand{\cX}{\mathscr X}
\newcommand{\cY}{\mathcal Y}
\newcommand{\GL}{\operatorname{GL}}
\newcommand{\SL}{\operatorname{SL}}
\newcommand{\Gr}{\operatorname{Gr}}
\newcommand{\Pic}{\operatorname{Pic}}
\newcommand{\vol}{\operatorname{vol}}
\newcommand{\covol}{\operatorname{covol}}

\newcommand{\Lip}{\operatorname{Lip}}
\newcommand{\BL}{\mathrm{BL}}
\newcommand{\op}{\mathrm{op}}
\newcommand{\supp}{\operatorname{supp}}

\newcommand{\Addresses}{{%
  \bigskip
  \footnotesize
  Ruida Di, \textsc{Morningside Center of Mathematics,
  Academy of Mathematics and Systems Science, Chinese Academy of Sciences,
  No. 55, Zhongguancun East Road, Beijing 100190, China.}\par\nopagebreak
  \textit{Email address}, \texttt{drdmath@amss.ac.cn}
}}

\title{Effective Joint Equidistribution of Subspaces of Bounded Height in Rigid Adelic Spaces}
\author{Ruida Di}
\date{}

\begin{document}
\maketitle
\begingroup
\renewcommand{\thefootnote}{}
\begin{NoHyper}
\end{NoHyper}
\endgroup

\begin{abstract}
We prove effective joint equidistribution for subspaces of fixed
dimension in a fixed rigid adelic space over a number field.
The joint distribution records the rescaled height, the archimedean
position, and the shapes of the subspace and its quotient.
More generally, we obtain counting estimates for regions in which the
height bound may depend on the position and both shapes. Under a
H\"older-type bound for the measure of boundary layers, the estimates
have explicit power-saving error terms.
We apply these results to counting problems involving slopes, freeness,
and semistability, and to Diophantine approximation with prescribed shapes.
\end{abstract}

\tableofcontents

\section{Introduction}
The classical problem of counting rational subspaces by height is the starting point for our study.

\subsection{Motivation and related work}\label{subsec:intro-background}

We first consider the standard Euclidean setting over $\Q$. Fix integers $n\ge2$ and $1\le d<n$.
For a $d$-dimensional subspace $W\subset\Q^n$, the
lattice of its integer vectors is
\[
 \Lambda_W=W\cap\Z^n.
\]
The lattice $\Lambda_W$ has rank $d$ and spans the real
$d$-dimensional subspace $W\otimes_\Q\R$. It is primitive;
equivalently, $\Z^n/\Lambda_W$ is torsion-free. With the standard adelic norms on $\Q^n$, the height of $W$ is
\[
 H(W)=\covol(\Lambda_W).
\]
The problem of counting by height is to estimate, as $B\to\infty$,
\[
 N_{n,d}(B):=\#\bigl\{W\subset\Q^n:\dim_\Q W=d,\ H(W)\le B\bigr\}.
\]
Equivalently, $N_{n,d}(B)$ counts primitive rank-$d$ sublattices of $\Z^n$ of covolume at most $B$.
Schmidt \cite[Theorem~1]{Schmidt1968} proved that
\[
 N_{n,d}(B)=c_{n,d}B^n+O_{n,d}\bigl(B^{n-b(n,d)}\bigr),
 \qquad b(n,d)=\max\left\{\frac1d,\frac1{n-d}\right\},
\]
where
\[
 c_{n,d}=\frac{1}{2n}\binom nd
 \frac{\prod_{j=n-d+1}^{n} V_j}{\prod_{j=2}^{d} V_j}
 \frac{\prod_{j=2}^{d}\zeta(j)}{\prod_{j=n-d+1}^{n}\zeta(j)}.
\]
Here $V_j$ denotes the volume of the Euclidean unit ball in $\R^j$.
Kim \cite[Theorem~1.3]{KimGrassmannian} extended this asymptotic formula to primitive rank-$d$ sublattices of an arbitrary full-rank Euclidean lattice $L\subset\R^n$. For fixed $L$, the leading error exponent remains $n-b(n,d)$; the error terms are controlled by the successive minima of $L$. 

Over a number field $K$ of degree $\kappa=[K:\Q]$, the analogous problem is to count $d$-dimensional $K$-subspaces of $K^n$ of bounded height.
We use the absolute-height normalization of Section~\ref{subsec:parameter-spaces} throughout. For $d=1$, Schanuel \cite[Theorem~3 and its corollary]{Schanuel1979}
obtained an asymptotic formula with an error term for the number of
points of bounded Weil height in $\mathbb P^{n-1}(K)$, using maximum
norms at the archimedean places. For general $d$, Thunder \cite{Thunder1992} proved the following
asymptotic formula for the standard Euclidean and Hermitian norms
at the archimedean places:
\[
\begin{aligned}
N_{K,n,d}(B)&:=\#\{W\subset K^n:\dim_K W=d,\ H(W)\le B\}\\
&=\mathfrak a_{K,n,d}B^{\kappa n}+O_{K,n,d}(B^{\kappa n-\eta}),
\end{aligned}
\]
for some $\eta=\eta(K,n,d)>0$, where $\mathfrak a_{K,n,d}$ is the constant given in \eqref{eq:main-term-constant}. Thunder later introduced heights allowing local changes of
coordinates, as in the rigid adelic setting considered here. He also
obtained counting results for Grassmannians and flag varieties
\cite[Part~I, \S~1--2; Part~II, Theorem~3; Part~III, Theorem~5]{Thunder1993}. Widmer \cite[Theorem~3.1]{Widmer2010} obtained an extension of Schanuel's theorem to adelic Lipschitz heights on projective space. Since heights defined by rigid adelic norms are adelic Lipschitz
heights, Widmer's result applies to one-dimensional subspaces
of a fixed rigid adelic space.

Counting subspaces by height is also a case of counting rational points
on flag varieties. Franke, Manin and Tschinkel
\cite{FrankeManinTschinkel} established height asymptotics for generalized
flag varieties in their work on Manin's conjecture. Peyre
\cite{Peyre1995} formulated the leading constant in terms of Tamagawa
measures. For generalized flag varieties over number fields, Watanabe
\cite{Watanabe2003} proved asymptotic formulas for rational points
of bounded height subject to prescribed finite adelic congruence conditions.

To describe the joint distribution, we return to the standard Euclidean setting over $\Q$.
Besides its height, a subspace $W$ has a Grassmannian position and the shapes of its primitive lattice and factor lattice \cite[Definition~2.1]{HoreshKarasik}. The position of $W$ is the point $W\otimes_\Q\R$ in the real Grassmannian. Let $W^\perp$ denote the orthogonal complement of $W\otimes_\Q\R$, and let $\pi_{W^\perp}:\R^n\to W^\perp$ be the orthogonal projection. The lattice
\[
\Lambda_W^\pi:=\pi_{W^\perp}(\Z^n)
\]
is isometric to $\Z^n/\Lambda_W$ equipped with the norm induced by the Euclidean quotient $\R^n/(W\otimes_\Q\R)$. It has rank $n-d$ and covolume $H(W)^{-1}$. Consequently, the lattices
\[
H(W)^{-1/d}\Lambda_W
\qquad\text{and}\qquad
H(W)^{1/(n-d)}\Lambda_W^\pi
\]
both have covolume one. Their isometry classes are called the shapes of the primitive lattice $\Lambda_W$ and its factor lattice $\Lambda_W^\pi$.

Choose $W$ uniformly among the $d$-dimensional subspaces
of $\Q^n$ with $H(W)\le B$. We seek, as $B\to\infty$, the joint
limiting distribution of its Grassmannian position and these two
shapes, together with a rate of convergence.

Schmidt's study of lattice shapes is a precursor to this joint problem.
Write $\cX_{\Q,m}$ for the space of shapes of unimodular
$m$-dimensional Euclidean lattices, $\mu_{\Q,m}$ for its invariant
probability measure, and $\operatorname{sh}(\Lambda)$ for the shape
of a lattice $\Lambda$.
For $1<d<n$, Schmidt \cite[Theorem~2]{Schmidt1998}
obtained asymptotic counts for primitive rank-$d$ sublattices of $\Z^n$ of
bounded covolume whose similarity classes are represented by lean, Jordan measurable subsets of a suitable region $\mathcal H(a,b)$
in the generalized upper half space.
Horesh and Karasik \cite[Theorem~B]{HoreshKarasikVectors}
obtained effective joint equidistribution of the directions of
primitive vectors and the shapes of their orthogonal lattices.

For $n\ge3$, Horesh and Karasik \cite[Theorem~1.1]{HoreshKarasik}
proved effective joint equidistribution of the Grassmannian position
and the shapes of $\Lambda_W$ and its orthogonal lattice
$\Lambda_W^\perp:=\Z^n\cap W^\perp$.
The same authors also treat the factor lattice
$\Lambda_W^\pi=(\Lambda_W^\perp)^*$ in
\cite[Theorem~3.2(1)]{HoreshKarasik}, where $*$ denotes
the Euclidean dual. Let $\sigma_{\Q^n,d}$ be the
rotation-invariant probability measure on $\Gr_d(\R^n)$. For a set $\Omega\subset\Gr_d(\R^n)$ and bounded sets
$\mathcal A\subset\cX_{\Q,d}$ and
$\mathcal B\subset\cX_{\Q,n-d}$, all boundary controllable in the sense of \cite[Definition~3.1]{HoreshKarasik},
they proved, for every $\varepsilon>0$,
\[
\begin{aligned}
&\#\left\{W\subset\Q^n:\begin{array}{l}\dim_\Q W=d,\ H(W)\le B,\
     W\otimes_\Q\R\in\Omega, \\ \operatorname{sh}(\Lambda_W)\in\mathcal A,\
  \operatorname{sh}(\Lambda_W^\pi)\in\mathcal B\end{array}\right\}\\
&\quad= c_{n,d}\,\sigma_{\Q^n,d}(\Omega)\,
\mu_{\Q,d}(\mathcal A)\,\mu_{\Q,n-d}(\mathcal B)\,B^n +O_{\varepsilon,\Omega,\mathcal A,\mathcal B}
\bigl(B^{n-\frac{1}{2n(n+1)}+\varepsilon}\bigr).
\end{aligned}
\]
If either shape set is unbounded, the error estimate holds with
$1/(4n(n^2-1))$ in place of $1/(2n(n+1))$.

For flags, Horesh and Karasik
\cite[Theorems~1.1 and~1.2]{HoreshKarasikFlags} obtained asymptotic
counts with prescribed flag positions and shapes of the successive
quotient lattices. Their results give power-saving errors both for a
maximum height and for the anticanonical height.

Over global function fields, Horesh and Paulin
\cite[Theorem~1.2 and Corollary~4.8]{HoreshPaulin} established effective joint equidistribution of
the spans and shapes of primitive partial lattices and the shapes of
their orthogonal lattices, as their prescribed covolumes tend to infinity
through admissible values.

The subspace and quotient lattices determine the tangent lattice
on the Grassmannian \cite[\S~3]{BrowningHoreshWilsch}.
This tangent lattice enters Peyre's freeness invariant
\cite[Definition~6.7]{Peyre2021}.
Browning, Horesh and Wilsch \cite[Theorem~1.3]{BrowningHoreshWilsch}
proved that, on rational Grassmannians, imposing any fixed freeness
threshold $0\le\varepsilon<1$ preserves the leading term in the
count with respect to the anticanonical height, with a power-saving error.

Subject to the respective arithmetic hypotheses, equidistribution
results at fixed discriminant are known for primitive vectors
\cite[Theorem~1.2]{AkaEinsiedlerShapira}
\cite[Theorem~1.2]{EinsiedlerRuehrWirth}, for planes in four dimensions
\cite[Theorem~1.3]{AkaEinsiedlerWieser}, and for further ranks
\cite[Theorem~1.4]{AkaMussoWieser}.
Since $\operatorname{disc}(\Lambda_W)=H(W)^2$, these results concern individual height levels.
Chinta and Pereira J\'unior \cite{ChintaPereira} related the Dirichlet
series for counting two-dimensional sublattices of $\Z^4$ by covolume
to Eisenstein series on $\GL_4$ and its metaplectic double cover.
For rational planes in a quadratic space of signature $(1,3)$, Aka,
Andritsch and Wieser \cite[Theorem~1.2]{AkaAndritschWieser} proved joint
equidistribution of associated CM points and periodic geodesics,
subject to a splitting condition of Linnik type, as odd squarefree
discriminants tend to infinity.

\paragraph{Our work.}
Theorem~\ref{thm:general-equidistribution} extends the effective joint
equidistribution results of Horesh--Karasik over $\Q$ to subspaces of
a fixed rigid adelic space over an arbitrary number field.
The more general Theorem~\ref{thm:boundary-modulus-count} gives counts
for joint regions in height, archimedean position, and both shapes,
with explicit power-saving errors under H\"older-type bounds for the
measure of boundary layers.
Theorem~\ref{thm:cusp-decaying-height} allows the height bound to depend
on position and both shapes and to grow at a controlled rate in the cusp. In Section~\ref{sec:applications}, we obtain counts with bounds
on generators and defining equations (Corollaries~\ref{cor:generated-subspaces}
and~\ref{cor:generators-equations-count}). We also obtain a
power-saving count by maximal tangent slope
(Corollary~\ref{cor:maximal-slope-count}) and a limiting distribution
for the rescaled freeness deficit
(Corollary~\ref{cor:freeness-defect-distribution}).

Our approach also suggests analogous joint equidistribution
results for partial flags in a rigid adelic space, with the shapes of the
successive quotients as variables.

\subsection{Main theorems}\label{subsec:intro-results}

Let $K$ be a number field of degree $\kappa=[K:\Q]$, and let $E$ be
an $n$-dimensional rigid adelic space over $K$ in the sense of
\cite[Definition~2]{Gaudron} (see also
Section~\ref{subsec:parameter-spaces}).
Fix $1\le d<n$ and $e=n-d$.

At each place $v$, the rigid adelic structure on $E$ equips
$E_v=E\otimes_KK_v$ with a norm.  At the archimedean places these
norms are Euclidean or Hermitian, and at the finite places they come
from a projective $\cO_K$-lattice.

Write $H(E)=H(\det E)$ for the height of $E$. Define $\widehat E$ by
\[
 \|x\|_{\widehat E,v}=
 \begin{cases}
 H(E)^{-1/n}\|x\|_{E,v},&v\mid\infty,\\
 \|x\|_{E,v},&v\nmid\infty.
 \end{cases}
\]
Then $H(\widehat E)=1$. In this paper, we call the isometry class $[\widehat E]$
the shape of $E$. Subspaces carry the induced norms, and quotients carry the quotient
norms. Their heights are normalized by the same convention, using
their respective dimensions. Write $\cX_{K,m}$ for the space of $K$-linear isometry classes of
$m$-dimensional rigid adelic spaces of height one. This space carries a probability measure
$\mu_{K,m}$; see Section~\ref{subsec:parameter-spaces}.

For each $d$-dimensional subspace $W$ of $E$, its archimedean
position is the point
\[
 W_\infty=(W\otimes_KK_v)_{v\mid\infty}
 \in\Gr_{d,\infty}(E):=
 \prod_{v\mid\infty}\Gr_d(E\otimes_KK_v).
\]
Let $\sigma_{E,d}$ denote the product of the invariant probability
measures on these Grassmannians.

The canonical determinant isometry
\begin{equation}\label{eq:determinant-line-isometry}
 \det E\simeq\det W\otimes\det(E/W)
\end{equation}
implies $H(W)H(E/W)=H(E)$ \cite[Proposition~5]{Gaudron}.
After normalization, it gives
\[
 \det\widehat W\otimes\det\widehat{E/W}
 \simeq\det\widehat E.
\]
Define
\[
\cY_{E,d}=\left\{([W_1],[W_2]):
 \begin{array}{l}
 \dim_KW_1=d,\quad \dim_KW_2=e,\quad H(W_1)=H(W_2)=1,\\
 \det W_1\otimes\det W_2\simeq\det\widehat E
 \end{array}\right\},
\]
where $W_1,W_2$ are rigid adelic spaces over $K$ and $\simeq$ denotes
a $K$-linear isometry at every place. This is a closed subset of
$\cX_{K,d}\times\cX_{K,e}$. Let $\nu_{E,d}$ be the probability measure
on $\cY_{E,d}$ defined in \eqref{eq:joint-moduli-measure}.
For general $K$, the determinant condition couples the two shape
variables, and $\nu_{E,d}$ need not be the product of its two marginals.
For $K=\Q$, $\nu_{E,d}$ is a product measure.

Let
\[
 \mathcal X_E=\Gr_{d,\infty}(E)\times\cY_{E,d},\qquad
 \mu_E=\sigma_{E,d}\otimes\nu_{E,d},
\]
\[
 x_E(W)=\bigl(W_\infty,([\widehat W],[\widehat{E/W}])\bigr).
\]

Let $\mathcal Z_E=(0,\infty)\times\mathcal X_E$ and define the measure
\begin{equation}\label{eq:joint-radial-measure}
 d\lambda_E(r,x)=\kappa n r^{\kappa n-1}\,dr\,d\mu_E(x).
\end{equation}

For functions on $\mathcal X_E$, the bounded Lipschitz norm
$\|f\|_{\BL}=\|f\|_\infty+\Lip(f)$ is computed using the metric
$d_E$ from Section~\ref{subsec:metrics-exponents}. For functions on
$\mathcal Z_E$, we use the product metric $|r-r'|+d_E(x,x')$.

To define boundary layers, we instead use the logarithmic metric
\begin{equation}\label{eq:logarithmic-product-metric}
 d_{\log,E}((r,x),(r',x'))
 =|\log r-\log r'|+d_E(x,x').
\end{equation}
For a Borel set $\Omega\subset\mathcal Z_E$ and $h>0$, define
\begin{equation}\label{eq:joint-boundary-layer}
 \partial_h\Omega=
 \{z\in\mathcal Z_E:d_{\log,E}(z,\Omega)\le h,
   \ d_{\log,E}(z,\mathcal Z_E\setminus\Omega)\le h\},
\end{equation}
where distance to the empty set is $+\infty$.

For $1\le i\le n$, let $\lambda_i(E)$ denote the $i$-th
Roy--Thunder minimum: the least height bound for $i$ $K$-linearly
independent vectors in $E$ (see \eqref{eq:RT-minima}).
For a height-one isometry class $L=[F]\in\cX_{K,m}$, put
$\lambda_i(L)=\lambda_i(F)$.

Let
\begin{equation}\label{eq:cusp-size}
 \tau_E(w,(L_1,L_2))
 =\max\{2,\lambda_d(L_1),\lambda_e(L_2)\}.
\end{equation}

Let $\mathfrak a_{K,n,d}$ be the positive constant defined in
\eqref{eq:main-term-constant}.
For a Borel set $\Omega\subset\mathcal Z_E$ and a bounded Lipschitz
function $\varphi:\mathcal Z_E\to\C$, set
\begin{equation}\label{eq:joint-region-discrepancy}
 \mathcal E_B(\Omega,\varphi)
 =\sum_{\substack{W\subset E,\ \dim_KW=d\\
                   (H(W)/B,x_E(W))\in\Omega}}
       \varphi(H(W)/B,x_E(W))
 -\frac{\mathfrak a_{K,n,d}}{H(E)^{\kappa d}}B^{\kappa n}
                         \int_\Omega \varphi\,d\lambda_E.
\end{equation}
We use this notation whenever the sum is finite and the integral is
absolutely convergent.

When $n=2$, we have $d=e=1$ and $\lambda_1(L)=1$ for every
height-one adelic line; hence $\tau_E\equiv2$.

For $n\ge3$, let $b_{d,e}>0$ be defined by \eqref{eq:bde}.
Let $\delta_{\alpha,\gamma}>0$ be the exponent defined in
Section~\ref{subsec:numerical-parameters}.

\begin{theorem}\label{thm:boundary-modulus-count}
Let $M>0$. If $n\ge3$, fix $0\le\gamma<b_{d,e}/(\kappa n)$;
if $n=2$, set $\gamma=0$. Let $\Omega$ be a Borel subset of
\begin{equation}\label{eq:joint-region-growth}
 \{(r,x)\in\mathcal Z_E:0<r\le M\tau_E(x)^\gamma\},
\end{equation}
and let $\varphi:\mathcal Z_E\to\C$ be bounded Lipschitz.
Then:
\begin{enumerate}[label=\textup{(\roman*)}]
\item Suppose that, for some $h_\partial>0$ and a nondecreasing
function $\omega:(0,h_\partial)\to[0,\infty)$,
\begin{equation}\label{eq:general-boundary-modulus}
 \lambda_E(\partial_h\Omega)\le\omega(h)\quad(0<h<h_\partial),
 \qquad\lim_{h\to0^+}\omega(h)=0.
\end{equation}
Then there exists $B_0=B_0(E,\gamma,h_\partial)\ge2$ such that, for every
$\varepsilon>0$ and every $B\ge B_0$,
\begin{equation}\label{eq:general-modulus-convergence}
 |\mathcal E_B(\Omega,\varphi)|
 \ll_{E,M,\gamma,\varepsilon}\|\varphi\|_{\BL}B^{\kappa n}
       \left(\omega(B^{-\delta_{1,\gamma}})
                       +B^{-\delta_{1,\gamma}+\varepsilon}\right).
\end{equation}

\item Suppose that, for some $C_\partial>0$, $0<h_\partial\le1$,
and $0<\alpha\le1$,
\begin{equation}\label{eq:joint-boundary-condition}
 \lambda_E(\partial_h\Omega)\le C_\partial h^\alpha
 \qquad(0<h<h_\partial).
\end{equation}
Then, for every $0<\delta<\delta_{\alpha,\gamma}$ and every $B\ge2$,
\begin{equation}\label{eq:joint-region-count}
 |\mathcal E_B(\Omega,\varphi)|
 \ll_{E,M,C_\partial,h_\partial,\alpha,\gamma,\delta}
       \|\varphi\|_{\BL}B^{\kappa n-\delta}.
\end{equation}
\end{enumerate}
\end{theorem}

We next consider height conditions of the form
$H(W)Q(x_E(W))\le B$. Set
$\Omega_Q=\{(r,x)\in\mathcal Z_E:rQ(x)\le1\}$.
Then the height condition is equivalent to
$(H(W)/B,x_E(W))\in\Omega_Q$, so Theorem~\ref{thm:boundary-modulus-count}
applies once its boundary hypothesis is verified.

For $Q:\mathcal X_E\to(0,\infty)$, we say that $\log Q$ is
uniformly locally Lipschitz if there exist $L,\epsilon_0>0$ such that
\begin{equation}\label{eq:shape-weight-regularity}
 |\log Q(x)-\log Q(x')|\le Ld_E(x,x')
 \quad\text{whenever }d_E(x,x')<\epsilon_0.
\end{equation}

\begin{theorem}\label{thm:cusp-decaying-height}
If $n\ge3$, fix $0\le\gamma<b_{d,e}/(\kappa n)$;
if $n=2$, let $\gamma=0$. Let
$Q:\mathcal X_E\to(0,\infty)$ satisfy
\eqref{eq:shape-weight-regularity} and
\[
 Q(x)\ge c\,\tau_E(x)^{-\gamma}\qquad(x\in\mathcal X_E)
\]
for some $c>0$. Then, for every bounded Lipschitz function $\varphi:\mathcal Z_E\to\C$,
every $0<\delta<\delta_{1,\gamma}$, and every $B\ge2$,
\begin{equation}\label{eq:cusp-decaying-height-count}
\begin{aligned}
&\sum_{\substack{W\subset E,\ \dim_KW=d\\H(W)Q(x_E(W))\le B}}
       \varphi(H(W)/B,x_E(W))\\
&\quad=\frac{\mathfrak a_{K,n,d}}{H(E)^{\kappa d}}B^{\kappa n}
       \int_{\mathcal X_E}\int_0^{Q(x)^{-1}}
       \varphi(r,x)\kappa n r^{\kappa n-1}\,dr\,d\mu_E(x)\\
&\qquad+O_{E,c,\gamma,L,\epsilon_0,\delta}
       \bigl(\|\varphi\|_{\BL}B^{\kappa n-\delta}\bigr).
\end{aligned}
\end{equation}
\end{theorem}

By \eqref{eq:decaying-weight-uniform-bounds}, the lower bound
on $Q$ also ensures that the main-term integral is finite.
Taking $Q\equiv1$ in Theorem~\ref{thm:cusp-decaying-height}
gives effective joint equidistribution of the rescaled height
$H(W)/B$, the archimedean position, and the two shapes.

\begin{theorem}\label{thm:general-equidistribution}
\label{thm:bounded-height}
Let $0<\delta<\delta_{1,0}$.
\begin{enumerate}[label=\textup{(\roman*)}]
\item Let $0<a<b<\infty$. Then, for every
bounded Lipschitz function $\Psi$ on
$(0,\infty)\times\mathcal X_E$ satisfying
$\supp\Psi\subset[a,b]\times\mathcal X_E$, and every $B\ge2$,
\begin{align}
 B^{-\kappa n}\sum_{\substack{W\subset E\\\dim_KW=d}}
 \Psi\left(H(W)/B,x_E(W)\right)
 &=\frac{\mathfrak a_{K,n,d}}{H(E)^{\kappa d}}
 \int_0^\infty\int_{\mathcal X_E}
 \Psi(r,x)\,\kappa n r^{\kappa n-1}\,d\mu_E(x)\,dr\notag\\*
 &\quad+O_{E,a,b,\delta}\left(\|\Psi\|_{\BL}B^{-\delta}\right).
 \label{eq:effective-radial}
\end{align}

\item For every bounded Lipschitz
function $f$ on $\mathcal X_E$ and every $B\ge2$,
\begin{equation}\label{eq:effective-bounded}
 \sum_{\substack{W\subset E,\ \dim_KW=d\\H(W)\le B}}f(x_E(W))
=\frac{\mathfrak a_{K,n,d}}{H(E)^{\kappa d}}
 B^{\kappa n}\int_{\mathcal X_E}f\,d\mu_E
 +O_{E,\delta}\left(\|f\|_{\BL}B^{\kappa n-\delta}\right).
\end{equation}
\end{enumerate}
\end{theorem}

Taking $f=1$ in Theorem~\ref{thm:general-equidistribution}(ii) gives
\begin{equation}\label{eq:total-subspace-count}
\begin{aligned}
N_{E,d}(B)&:=\#\{W\subset E:\dim_KW=d,\ H(W)\le B\}\\
&=\frac{\mathfrak a_{K,n,d}}{H(E)^{\kappa d}}B^{\kappa n}
 +O_{E,\delta}(B^{\kappa n-\delta}).
\end{aligned}
\end{equation}
For $E=K^n$ with the standard adelic norms, the leading term agrees
with Thunder's asymptotic formula \cite{Thunder1992}. Over $\Q$, the formula
recovers the leading term of Kim's count of primitive sublattices
of a fixed Euclidean lattice \cite[Theorem~1.3]{KimGrassmannian}.

\subsection{Idea of the proof}\label{subsec:intro-proof}

We begin by counting subspaces whose two shapes lie in compact
truncations. The contribution from the cusp is treated separately.

Fix a $K$-basis of $E$. The $d$-dimensional $K$-subspaces
are then parametrized by $\SL_n(K)/\mathbf P_d(K)$, where
$\mathbf P_d$ is the parabolic stabilizing the standard $d$-plane. Strong approximation for $\SL_n$, together with the finiteness
of the associated finite adelic double quotient, reduces the
subspace count to finitely many weighted lattice point counts. In each sector, Iwasawa coordinates record height,
position, and the two shapes. We reduce to nonnegative test functions
and encode the joint constraints in a nonnegative sector weight.
The height bound may then depend on the archimedean position and
both shapes. Integrating these weights
against Haar measure identifies the leading constant and the joint
limiting measure, including the determinant relation between the two
shapes (Proposition~\ref{prop:parabolic-mass}).

Uniform matrix coefficient estimates give operator-norm bounds
for averaging on the mean-zero subspace of $L^2$ of each arithmetic quotient. Proposition~\ref{prop:smoothing} bounds the discrepancy between
the weighted lattice sum and its Haar integral by two terms:
a spectral operator term and the integral of the difference between
the upper and lower envelopes of the weight. The spectral
term is estimated in Proposition~\ref{prop:spectral}.

We truncate the shape spaces by bounding their last successive
minima. On the resulting compact sets, quantitative reduction theory
controls the variation of the sector weights under small translations. The short pseudobasis theorem gives representatives
whose norms and inverse norms are bounded polynomially in the
truncation parameter
(Proposition~\ref{prop:controlled-representatives}). Small translations
therefore perturb height, position, and shapes by controlled amounts.
The region indicator can change only within a boundary neighborhood;
its contribution to the counting error is bounded by the measure of
that neighborhood.

The cusp estimate rests on a geometry-of-numbers observation: a
large normalized last minimum forces a subspace to contain a
hyperplane whose height is small relative to that of the subspace.
Counting these incidences and applying duality to the quotient
spaces gives quantitative no escape of mass
(Theorem~\ref{thm:no-escape}). Together with the spectral and
boundary estimates, this proves Theorem~\ref{thm:boundary-modulus-count}.
When the measures of the boundary layers satisfy the H\"older bound
in \eqref{eq:joint-boundary-condition}, balancing the smoothing and
truncation parameters gives the power-saving error in
Theorem~\ref{thm:boundary-modulus-count}(ii).

\subsection{Organization of the paper}\label{subsec:intro-organization}

Section~\ref{sec:adelic-preliminaries} introduces the adelic parameter
spaces and measures, together with the estimates for successive minima.
In Section~\ref{sec:equidistribution}, we reduce subspace counting to
lattice point counting and establish quantitative coordinate estimates.
The spectral and smoothing estimates for weighted lattice sums are proved
in Section~\ref{sec:spectral-smoothing}.
These estimates are combined in Section~\ref{sec:compact-truncations}
to obtain effective counts on compact shape truncations and
equidistribution for compactly supported tests.
Section~\ref{sec:counting-no-escape} establishes quantitative cusp
estimates for the arithmetic counts and the limiting measure. The main
theorems and their consequences for height counting are then proved in
Section~\ref{sec:main-theorem-proofs}, using the compact counting and
cusp estimates.
Section~\ref{sec:applications} treats counts with bounds on generators
and defining equations, slopes and freeness, semistability, unbounded
shape statistics, determinant classes, and local approximation with
prescribed shapes.
Finally, Appendix~\ref{app:gauge-form} computes the parabolic measure
and identifies its pushforward to the shape space.

\paragraph{Acknowledgements.}
The author gratefully acknowledges financial support from the Morningside
Center of Mathematics and the Chinese Academy of Sciences, and thanks
Professor Ye Tian for his support of this research. During his two years at
the Morningside Center of Mathematics, the author benefited from discussions
with Professor Pengyu Yang and his Ph.D. student Mengyang Wang on
homogeneous dynamics, including Ratner's theory. He also thanks Professor
Lifan Guan for helpful discussions on Manin's conjecture, Professors Ping Xi
and Wei He for valuable discussions, and Professor Zhizhong Huang for a
valuable discussion at the conference on Arakelov geometry in Hangzhou.
The author is grateful to Professor Fr\'ed\'eric Paulin for his comments
and encouragement and to Menny Aka for his interest in this work.

\paragraph{AI Declaration.}
The author used ChatGPT (OpenAI, Sol 6.1) for language editing,
reference checking, and manuscript organization.
It also assisted with revising mathematical arguments, including
finding a simpler proof of Proposition~\ref{prop:sector-properness}.
The author reviewed and revised the proposed changes and takes full
responsibility for the mathematical content, the accuracy of the
references, and the final form of the paper.

\section{Adelic parameter spaces, measures, and successive minima}\label{sec:adelic-preliminaries}

In estimates involving $O$, $\ll$, or $\asymp$, subscripts indicate
the permitted dependence of the implicit constants. We omit from these
subscripts the dependence on the fixed Haar measures, metrics, and
Lie algebra norms.

\subsection{Heights and isometry classes}\label{subsec:parameter-spaces}

Fix a number field $K$ of signature $(r_1,r_2)$.  Let $D_K$, $h_K$,
$R_K$, $w_K$, and $\zeta_K(s)$ denote the discriminant, class number,
regulator, number of roots of unity, and Dedekind zeta function of $K$,
respectively, with the convention $R_K=1$ when the unit rank is zero.  Let
\begin{equation}\label{eq:number-field-data}
 \kappa=[K:\Q]=r_1+2r_2,
 \qquad s_\infty=r_1+r_2.
\end{equation}
For every place $v$, let $n_v$ denote the local degree: $n_v=[K_v:\Q_p]$ if
$v\mid p$, $n_v=1$ at a real place, and $n_v=2$ at a complex place.
For $v\mid p$, let $|\cdot|_v$ be the unique extension to $K_v$ of the
absolute value on $\Q_p$ normalized by $|p|_p=p^{-1}$.
At an archimedean place, use the usual real or complex absolute value.
With these conventions, the product formula is
\begin{equation*}
 \prod_v|a|_v^{n_v}=1\qquad(a\in K^\times).
\end{equation*}
For a finite place $v$, write $\cO_v$ for the ring of integers of $K_v$.
For a $K$-vector space $V$ and a place $v$, write
$V_v=V\otimes_KK_v$.

For an idele $a=(a_v)_v$, write
\begin{equation*}
 |a|_{\A_K}=\prod_v|a_v|_v^{n_v},
 \qquad \A_K^1=\{a\in\A_K^\times:|a|_{\A_K}=1\}.
\end{equation*}
For every $m\ge1$ and every place $v$, set
\begin{equation*}
 \|(x_1,\ldots,x_m)\|_{m,v}
 =
 \begin{cases}
 \bigl(\sum_{i=1}^m |x_i|_v^2\bigr)^{1/2},
     & v\mid\infty,\\[1mm]
 \max_{1\le i\le m}|x_i|_v,
     & v\nmid\infty.
 \end{cases}.
\end{equation*}
\begin{definition}\label{def:rigid-adelic-space}
Let $F$ be an $r$-dimensional $K$-vector space.  A \emph{rigid
adelic structure} on $F$ is a family of norms
$\|\cdot\|_{F,v}$ on $F_v=F\otimes_K K_v$, one for each place $v$ of $K$,
such that there exist a $K$-basis of $F$ and an element
$g_F=(g_{F,v})_v\in\GL_r(\A_K)$ satisfying, in the coordinates of this basis,
\begin{equation}\label{eq:adelic-norm-representatives}
 \|x\|_{F,v}=\|g_{F,v}x\|_{r,v}\qquad(x\in F_v)
\end{equation}
for every place $v$.  The vector space $F$ equipped with these
norms is called a \emph{rigid adelic space} over $K$.
\end{definition}
Definition~\ref{def:rigid-adelic-space} is Gaudron's definition
\cite[Definition~2]{Gaudron} in the number-field setting. Let $F$
be such a space of dimension $r\ge1$.

For $0\ne x\in F$, define
\begin{equation*}
 H_F(x)=\prod_v\|x\|_{F,v}^{n_v/\kappa}.
\end{equation*}
The product formula shows that $H_F(\lambda x)=H_F(x)$ for
$\lambda\in K^\times$; thus $H_F$ defines a projective height on
$\Pj(F)(K)$. 

For a $K$-subspace $W\subset F$, we equip $W$ with the
induced norms and $F/W$ with the quotient norms. The dual space
$F^\vee$ is equipped with the operator norms. These induced, quotient,
and dual structures are rigid.

The determinant line $\det F=\bigwedge^r F$ carries the induced
adelic norms. For $0\ne z\in\det F$, set $H(F)=H_{\det F}(z)$. For $g_F$ as in \eqref{eq:adelic-norm-representatives},
\cite[\S2.2]{Gaudron} gives
\begin{equation*}
 H(F)^\kappa=\prod_v|\det g_{F,v}|_v^{n_v}=|\det g_F|_{\A_K}.
\end{equation*}

For $1\le i\le r$, define the Roy--Thunder minima
\cite{RoyThunder} (see also \cite[\S3.1]{Gaudron}) by
\begin{equation}\label{eq:RT-minima}
 \lambda_i(F)=
 \inf\left\{
 \max_{1\le j\le i}H_F(x_j):
 x_1,\ldots,x_i\in F\text{ are $K$-linearly independent}
 \right\}.
\end{equation}
Since $H_F$ is equivalent to a Weil height on $\Pj(F)(K)$, Northcott's
theorem \cite[Theorem~2.4.9]{BombieriGubler} shows that the infimum in
\eqref{eq:RT-minima} is attained.

For $m\ge1$, let
\begin{equation*}
 \Kc_m=
 \prod_v\{u\in\GL_m(K_v):\|ux\|_{m,v}=\|x\|_{m,v}
 \text{ for all }x\in K_v^m\},
\end{equation*}
and set
\begin{equation*}
\begin{aligned}
\GL_m(\A_K)^1&=\{g\in\GL_m(\A_K):|\det g|_{\A_K}=1\},
 \\
\cX_{K,m}&=\Kc_m\backslash \GL_m(\A_K)^1/\GL_m(K).
\end{aligned}
\end{equation*}
Equip $\cX_{K,m}$ with the quotient topology.  The group
$\GL_m(\A_K)^1$ is locally compact and second countable,
$\GL_m(K)$ is discrete and closed, and $\Kc_m$ is compact.  It follows that
$\cX_{K,m}$ is locally compact, Hausdorff, and second countable; in
particular, it is a standard Borel space.

The space $\cX_{K,m}$ parametrizes $K$-linear isometry classes of
$m$-dimensional rigid adelic spaces of height one.
By the height formula, $H(F)=1$ if and only if the matrix $g_F$
in \eqref{eq:adelic-norm-representatives} belongs to
$\GL_r(\A_K)^1$. In this case let
\begin{equation*}
 [F]:=\Kc_rg_F\GL_r(K)\in\cX_{K,r}.
\end{equation*}
For
$L=[F]\in\cX_{K,r}$ and $1\le i\le r$, write $\lambda_i(L)=\lambda_i(F)$.

For each rigid adelic space $F$, define $\widehat F$ by
\begin{equation}\label{eq:height-one-normalization}
 \|x\|_{\widehat F,v}=
 \begin{cases}
 H(F)^{-1/r}\|x\|_{F,v},&v\mid\infty,\\
 \|x\|_{F,v},&v\nmid\infty.
 \end{cases}
\end{equation}
Then
\begin{equation*}
 H(\widehat F)=1,
 \qquad
 \lambda_i(\widehat F)=H(F)^{-1/r}\lambda_i(F),
\end{equation*}
so $[\widehat F]\in\cX_{K,r}$. We call $[\widehat F]$ the shape of $F$.

Applying Borel's finiteness theorem \cite[Theorem~5.8]{BorelAdeles} to $\GL_m$ shows that
$\GL_m(\A_K)^1$ is unimodular and that
$\GL_m(\A_K)^1/\GL_m(K)$ carries a finite invariant measure.  Normalize
this measure to have total mass one, and let $\mu_{K,m}$ be its pushforward
to $\cX_{K,m}$. Let
\begin{equation*}
 \mathcal U_K=
 \prod_{v\nmid\infty}\cO_v^\times
 \times\prod_{v\text{ real}}\{\pm1\}
 \times\prod_{v\text{ complex}}S^1,
\end{equation*}
and
\begin{equation*}
 \Pic^0(K)=\mathcal U_K\backslash\A_K^1/K^\times
\end{equation*}
for the degree-zero Arakelov class group. There is an exact sequence
\begin{equation}\label{eq:arakelov-class-exact-sequence}
 0\longrightarrow
 \frac{\{(u_v)_{v\mid\infty}:\sum_{v\mid\infty}n_vu_v=0\}}
 {\{(\log|\varepsilon|_v)_{v\mid\infty}:\varepsilon\in\cO_K^\times\}}
 \longrightarrow\Pic^0(K)
 \longrightarrow\operatorname{Cl}(K)\longrightarrow0,
\end{equation}
By Dirichlet's unit theorem and the finiteness of the class
group, $\Pic^0(K)$ is compact. Let $m_{\Pic^0(K)}$ denote its Haar probability measure.  

The determinant defines the continuous map
\begin{equation*}
 \det:\cX_{K,m}\longrightarrow\Pic^0(K),
 \qquad \det(\Kc_mg\GL_m(K))=\mathcal U_K(\det g)K^\times.
\end{equation*}
The determinant map is surjective: for each $a\in\A_K^1$,
$\operatorname{diag}(a,1,\ldots,1)$ lies in $\GL_m(\A_K)^1$
and has determinant $a$. We have
\[
 (\det)_*\mu_{K,m}=m_{\Pic^0(K)}.
\]
Since $\cX_{K,m}$ and $\Pic^0(K)$ are standard Borel spaces,
the disintegration theorem \cite[Theorem~3.4]{KallenbergProbability}
gives a measurable probability kernel $c\mapsto\mu_{K,m,c}$ such that
\begin{equation*}
 \mu_{K,m}=\int_{\Pic^0(K)}\mu_{K,m,c}\,dm_{\Pic^0(K)}(c),
\end{equation*}
The kernel is unique up to an $m_{\Pic^0(K)}$-null set,
and $\mu_{K,m,c}$ is supported on $\det^{-1}(c)$
for $m_{\Pic^0(K)}$-almost every $c$.

Fix an $n$-dimensional rigid adelic space $E$ with $n\ge2$.
Let $1\le d<n$, $e=n-d$, and $c_E=\det([\widehat E])$.
For a $d$-dimensional $K$-subspace $W\subset E$, the induced and quotient
adelic structures make the canonical isomorphism
\[
 \det E\simeq\det W\otimes\det(E/W)
\]
an isometry at every place.  In particular, the formula \cite[Proposition~5]{Gaudron} gives
\begin{equation}\label{eq:subspace-quotient-height}
 H(W)\,H(E/W)=H(E).
\end{equation}
By \eqref{eq:height-one-normalization}, normalization multiplies the
archimedean norms on $W$ and $E/W$ by $H(W)^{-1/d}$ and
$H(E/W)^{-1/e}$, respectively.
Taking determinant classes gives
\begin{equation*}
 \det([\widehat W])\,
 \det([\widehat{E/W}])=c_E.
\end{equation*}

Set
\begin{equation}\label{eq:joint-moduli-space}
 \cY_{E,d}=
 \{(L_1,L_2)\in\cX_{K,d}\times\cX_{K,e}:
 \det(L_1)\det(L_2)=c_E\}.
\end{equation}

The determinant maps are continuous, and $\cY_{E,d}$ is the
fiber over $c_E$ of the map
\begin{equation*}
 (L_1,L_2)\longmapsto
 \det(L_1)\det(L_2).
\end{equation*}
Thus $\cY_{E,d}$ is a closed subspace of
$\cX_{K,d}\times\cX_{K,e}$; in particular, it is locally compact, second
countable, and standard Borel.  We equip it with the subspace topology and define
\begin{equation}\label{eq:joint-moduli-measure}
 \nu_{E,d}=
 \int_{\Pic^0(K)}
 \mu_{K,d,c}\otimes\mu_{K,e,c_Ec^{-1}}\,dm_{\Pic^0(K)}(c).
\end{equation}

For $m_{\Pic^0(K)}$-almost every $c$, the product measure
$\mu_{K,d,c}\otimes\mu_{K,e,c_Ec^{-1}}$ is supported on
$\cY_{E,d}$. Hence $\nu_{E,d}$ is a probability measure on
$\cY_{E,d}$. Disintegration and the change
of variables $c\mapsto c_Ec^{-1}$, which preserves Haar measure, give the marginals
\begin{equation}\label{eq:marginals}
 \bigl((L_1,L_2)\mapsto L_1\bigr)_*\nu_{E,d}=\mu_{K,d},
 \qquad
 \bigl((L_1,L_2)\mapsto L_2\bigr)_*\nu_{E,d}=\mu_{K,e}.
\end{equation}

\begin{remark}\label{rem:determinant-class-constraint}
The condition in \eqref{eq:joint-moduli-space} records the degree zero Arakelov determinant class, not only the ideal class.  By
\eqref{eq:arakelov-class-exact-sequence}, the two determinant
classes can remain coupled even when the ideal class group is trivial.  For
$K=\Q$, $\Pic^0(K)$ is trivial and \eqref{eq:joint-moduli-measure} is the product measure.
\end{remark}

Set
\begin{equation*}
 \Gr_{d,\infty}(E)=\prod_{v\mid\infty}\Gr_d(E_v).
\end{equation*}
For a $d$-dimensional $K$-subspace $W\subset E$, let
\begin{equation*}
 W_\infty=(W\otimes_KK_v)_{v\mid\infty}\in\Gr_{d,\infty}(E).
\end{equation*}
For every place $v$, let
\begin{equation}\label{eq:det-one-isometry-group}
 \Kc_{E,v}=
 \{u\in\GL(E_v):\|ux\|_{E,v}=\|x\|_{E,v}\text{ for all }x\}
 \cap\SL(E_v).
\end{equation}
For $v\mid\infty$, after choosing an orthonormal basis this group is
$\operatorname{SO}(n)$ at a real place and $\operatorname{SU}(n)$ at a complex place.  Since
$0<d<n$, it acts transitively on $\Gr_d(E_v)$.  Hence there is a unique
$\Kc_{E,v}$-invariant probability measure on $\Gr_d(E_v)$; let
$\sigma_{E,d}$ be the product of these probability measures.

Recall the joint position and shape space and its probability measure:
\[
 \mathcal X_E=\Gr_{d,\infty}(E)\times\cY_{E,d},\qquad
 \mu_E=\sigma_{E,d}\otimes\nu_{E,d}.
\]
For a $d$-dimensional $K$-subspace $W\subset E$, its pair of
shapes and its joint position and shape are
\[
 y_W=([\widehat W],[\widehat{E/W}])\in\cY_{E,d},\qquad
 x_E(W)=(W_\infty,y_W)\in\mathcal X_E.
\]

\subsection{Topology and metrics on the parameter spaces}\label{subsec:metrics-exponents}

For a rigid adelic space $F$ of rank $r$, define the $\cO_K$-lattice
\begin{equation}\label{eq:integral-module}
 \mathcal M_F=\{x\in F:\|x\|_{F,v}\le1\text{ for all }v\nmid\infty\}.
\end{equation}
This is a projective $\cO_K$-module of rank $r$, with
$\mathcal M_F\otimes_{\cO_K}\cO_v=\{x\in F_v:\|x\|_{F,v}\le1\}$
for $v\nmid\infty$.
Let
\[
\begin{aligned}
\Kc_{m,f}&=\prod_{v\nmid\infty}\GL_m(\cO_v)
 \subset\GL_m(\A_{K,f}),\\
\mathcal M(g_f)&=\{x\in K^m:g_vx\in\cO_v^m\ \ (v\nmid\infty)\},
\end{aligned}
\]
where $g_f=(g_v)_v\in\GL_m(\A_{K,f})$. The finite-adelic
double cosets classify full $\cO_K$-lattices in $K^m$ up to
$\cO_K$-module isomorphism \cite[\S\S2.1--2.2]{BorelAdeles}:
\[
 \Kc_{m,f}\backslash\GL_m(\A_{K,f})/\GL_m(K)
 \xrightarrow[\sim]{\ [g_f]\mapsto[\mathcal M(g_f)]\ }
 \{\text{rank $m$ projective $\cO_K$-modules}\}/\simeq.
\]
Let $\det_{\cO_K}M=\bigwedge_{\cO_K}^mM$, a rank one projective
$\cO_K$-module. By the Steinitz classification \cite[\S2.1]{BorelAdeles},
\[
 \{\text{rank $m$ projective $\cO_K$-modules}\}/\simeq
 \xrightarrow[\sim]{\ [M]\mapsto[\det_{\cO_K}M]\ }
 \Pic(\cO_K)\simeq\operatorname{Cl}(K),
\]
with inverse $[\mathfrak a]\mapsto[\cO_K^{m-1}\oplus\mathfrak a]$
for $[\mathfrak a]\in\operatorname{Cl}(K)$.

Fix a rank $m$ projective $\cO_K$-module $M$ and let
$V=M\otimes_{\cO_K}K$. For each finite place $v$, let
\[
 M_v=M\otimes_{\cO_K}\cO_v\subset V_v.
\]
For each archimedean place $v$, let $\mathcal P(V_v)$ denote the space
of positive definite symmetric bilinear
forms on $V_v$ when $v$ is real, and positive definite Hermitian forms on
$V_v$ when $v$ is complex.
After choosing a basis of $V_v$, the matrix of a form identifies
$\mathcal P(V_v)$ with
\[
 \begin{cases}
 \{Q\in M_m(\R):Q^T=Q,\ x^TQx>0\text{ for every }0\ne x\in\R^m\},
     &v\text{ real},\\[1mm]
 \{Q\in M_m(\C):Q^*=Q,\ x^*Qx>0\text{ for every }0\ne x\in\C^m\},
     &v\text{ complex},
 \end{cases}
\]
These sets are open in the real vector spaces of symmetric and Hermitian
matrices, respectively. We equip $\mathcal P(V_v)$ with the induced
topology and smooth structure, and
$\prod_{v\mid\infty}\mathcal P(V_v)$ with the product structure.

For each $v\nmid\infty$, choose an $\cO_v$-basis 
$e_{1,v},\ldots,e_{m,v}$ of $M_v$.  Given
$q=(q_v)_{v\mid\infty}\in\prod_{v\mid\infty}\mathcal P(V_v)$,
define the rigid adelic space $F_{M,q}$ on $V$ by
\[
 \|x\|_{F_{M,q},v}=
 \begin{cases}
  \displaystyle\max_{1\le i\le m}|x_i|_v,
   &v\nmid\infty,\quad x=\sum_{i=1}^m x_i e_{i,v},\\[2mm]
  q_v(x,x)^{1/2},&v\mid\infty.
 \end{cases}
\]

Let
\[
\begin{aligned}
\mathcal S_M&=
 \bigl\{q\in\prod_{v\mid\infty}\mathcal P(V_v):H(F_{M,q})=1\bigr\},
 \\
\cX_{K,m}[M]&=\{[F]\in\cX_{K,m}:\mathcal M_F\simeq M\}.
\end{aligned}
\]
The function $q\mapsto H(F_{M,q})$ is continuous and positive, since in
fixed bases it is a constant multiple of a product of positive powers
of Gram determinants.  Multiplying all the forms $q_v$ by $t>0$
multiplies $H(F_{M,q})$ by $t^{m/2}$.  We equip $\mathcal S_M$ with the
subspace topology from $\prod_{v\mid\infty}\mathcal P(V_v)$.

The action of
\[
 \operatorname{Aut}_{\cO_K}(M)=\{\gamma\in\GL_K(V):\gamma M=M\}
\]
on $\mathcal S_M$ is given by
$(\gamma\cdot q)_v(x,y)=q_v(\gamma_v^{-1}x,\gamma_v^{-1}y)$,
where $\gamma_v=\gamma\otimes_KK_v$.

\begin{proposition}\label{prop:fixed-module-components}
Let $m\ge1$. Then:
\begin{enumerate}[label=\textup{(\roman*)}]
\item For every rank $m$ projective $\cO_K$-module $M$, the map
\[
 \Phi_M:\mathcal S_M\longrightarrow\cX_{K,m}[M],
 \qquad q\longmapsto[F_{M,q}],
\]
induces a homeomorphism
\begin{equation}\label{eq:module-component-quotient}
 \operatorname{Aut}_{\cO_K}(M)\backslash\mathcal S_M
 \xrightarrow{\ \sim\ }\cX_{K,m}[M].
\end{equation}
\item As $M$ ranges over the isomorphism classes of rank $m$ projective $\cO_K$-modules, the spaces $\cX_{K,m}[M]$ are the
connected components of $\cX_{K,m}$, and
\[
 \cX_{K,m}=\coprod_{[\mathfrak a]\in\operatorname{Cl}(K)}
 \cX_{K,m}[\cO_K^{m-1}\oplus\mathfrak a],
 \qquad \pi_0(\cX_{K,m})\simeq\operatorname{Cl}(K).
\]
\end{enumerate}
\end{proposition}
\begin{proof}
For (i), let $[F]\in\cX_{K,m}[M]$. Extend an $\cO_K$-isomorphism $M\simeq\mathcal M_F$ to a $K$-linear isomorphism $V\simeq F$, and pull back the local norms of $F$ to $V$.
The archimedean norms determine forms $q=(q_v)_{v\mid\infty}$,
and the finite unit balls are $M_v$. Thus the resulting adelic space
is $F_{M,q}$. Since $H(F)=1$, we have $q\in\mathcal S_M$, proving
that $\Phi_M$ is surjective.

If $\Phi_M(q)=\Phi_M(q')$, a $K$-linear isometry
$\gamma:F_{M,q}\to F_{M,q'}$ satisfies
$\gamma(M_v)=M_v$ at every finite place.  Since
\[
 M=\{x\in V:x\in M_v\text{ for all }v\nmid\infty\},
\]
we have $\gamma\in\operatorname{Aut}_{\cO_K}(M)$.
At each archimedean place $v$, the isometry condition gives
\[
 q'_v(\gamma_v x,\gamma_v y)=q_v(x,y)
 \qquad(x,y\in V_v).
\]
Hence $q'=\gamma\cdot q$.
Conversely, if $q'=\gamma\cdot q$ with
$\gamma\in\operatorname{Aut}_{\cO_K}(M)$, then
$\gamma:F_{M,q}\to F_{M,q'}$ is an isometry at every place.
The fibers of $\Phi_M$ are therefore the automorphism
orbits.

For the topology, fix an identification $V\simeq K^m$ and choose
$g_f\in\GL_m(\A_{K,f})$ such that $\mathcal M(g_f)=M$.
Let
\[
 T_{g_f}=\bigl\{g_\infty\in\prod_{v\mid\infty}\GL_m(K_v):
           (g_\infty,g_f)\in\GL_m(\A_K)^1\bigr\}.
\]
Define the continuous map
\[
\theta:T_{g_f}\longrightarrow\mathcal S_M,\qquad
(g_v)_{v\mid\infty}\longmapsto(q_v)_{v\mid\infty},\qquad
 q_v(x,y)=x^*g_v^*g_vy.
\]
Let $\pi:\GL_m(\A_K)^1\to\cX_{K,m}$ be the quotient map.
For $q\in\mathcal S_M$, let $Q_v$ be the matrix of $q_v$.
The height condition gives
$((Q_v^{1/2})_{v\mid\infty},g_f)\in\GL_m(\A_K)^1$, and
\[
 \Phi_M(q)=\pi\bigl((Q_v^{1/2})_{v\mid\infty},g_f\bigr).
\]
The positive square root depends continuously on $Q_v$, so
$\Phi_M$ is continuous.

Let $U\subset\mathcal S_M$ be open and set
$\widetilde U=\theta^{-1}(U)$, which is open in $T_{g_f}$.
Since $\Kc_{m,f}g_f$ is open in $\GL_m(\A_{K,f})$
and has constant finite determinant norm,
$\widetilde U\times\Kc_{m,f}g_f$ is open in
$\GL_m(\A_K)^1$.  Moreover,
\[
 \Phi_M(U)=\pi(\widetilde U\times\Kc_{m,f}g_f).
\]
The map $\pi$ is open because it is a quotient by left and right
translations. Thus $\Phi_M$ is open and induces the homeomorphism
\eqref{eq:module-component-quotient}.

For (ii), taking $U=\mathcal S_M$ above shows that
$\cX_{K,m}[M]=\Phi_M(\mathcal S_M)$ is open.
Its complement is the union of the open subsets corresponding to the
other module types, so it is also closed.

For $q_0,q_1\in\mathcal S_M$, let $q_t=(1-t)q_0+tq_1$,
$0\le t\le1$. The forms $q_t$ are positive definite, and
$H(F_{M,q_t})^{-2/m}q_t$ is a path in $\mathcal S_M$ with the
same endpoints. Hence $\cX_{K,m}[M]$ is connected and, being also
open and closed, is a connected component. The description of the
components follows from the
Steinitz classification of projective $\cO_K$-modules.
\end{proof}
Write $K_\infty=\prod_{v\mid\infty}K_v$. Fix a
$K$-basis of $V$ and let
\[
 \GL_m(K_\infty)^1=
 \left\{(g_v)_{v\mid\infty}\in\prod_{v\mid\infty}\GL_m(K_v):
 \prod_{v\mid\infty}|\det g_v|_v^{n_v}=1\right\}.
\]
The group $\GL_m(K_\infty)^1$ acts on $\mathcal S_M$ by
\[
 (g\cdot q)_v(x,y)=q_v(g_v^{-1}x,g_v^{-1}y).
\]
The determinant condition gives $H(F_{M,g\cdot q})=H(F_{M,q})=1$.
The action is transitive. The stabilizer of $q$ is the product of
the orthogonal or unitary groups of its forms $q_v$. Hence
\[
 \mathcal S_M\simeq\GL_m(K_\infty)^1
 \big/\left(\prod_{v\text{ real}}\mathrm O(m)
 \times\prod_{v\text{ complex}}\mathrm U(m)\right).
\]
We equip $\mathcal S_M$ with the smooth structure of this homogeneous space.
For each module class $[M]$, fix a $\GL_m(K_\infty)^1$-invariant
Riemannian distance $D_M$ on $\mathcal S_M$. On $\cX_{K,m}$, let
$d_m(L,L')=1$ when $L,L'$ lie in distinct components. If
$L,L'\in\cX_{K,m}[M]$ are represented by $q,q'\in\mathcal S_M$, define
\begin{equation}\label{eq:moduli-distance}
 d_m(L,L')=
 \min\Bigl\{1,\inf_{\gamma\in\operatorname{Aut}_{\cO_K}(M)}
 D_M(q,\gamma\cdot q')\Bigr\}.
\end{equation}
Choose a $\Z$-basis of $M$. In this basis, every element of
$\operatorname{Aut}_{\cO_K}(M)$ is represented by an integer matrix, so $\operatorname{Aut}_{\cO_K}(M)$ is discrete.
 Since its determinants over $K$ belong to $\cO_K^\times$,
the group lies in $\GL_m(K_\infty)^1$.
The action of $\GL_m(K_\infty)^1$ on $\mathcal S_M$ is proper,
since $\mathcal S_M$ is a homogeneous space with compact stabilizer.
Thus $\operatorname{Aut}_{\cO_K}(M)$ acts properly on $\mathcal S_M$,
and \eqref{eq:moduli-distance} defines a metric inducing the
quotient topology.

Fix a Riemannian distance $D_\infty$ on $\Gr_{d,\infty}(E)$ and
write $d_\infty=\min\{1,D_\infty\}$. For $x=(w,(L_1,L_2))$ and
$x'=(w',(L_1',L_2'))$ in $\mathcal X_E$, define
\[
 d_E(x,x')=d_\infty(w,w')+d_d(L_1,L_1')+d_e(L_2,L_2').
\]
Write $\operatorname{dist}=d_d+d_e$ for the metric on $\cY_{E,d}$.
On $\mathcal Z_E=(0,\infty)\times\mathcal X_E$ use the product
metric $|r-r'|+d_E(x,x')$ for bounded Lipschitz norms.
For a bounded function $f$ on any of these metric spaces $(X,D)$, let
\[
 \Lip(f)=\sup_{x\ne y}\frac{|f(x)-f(y)|}{D(x,y)},\qquad
 \|f\|_{\BL}=\|f\|_\infty+\Lip(f).
\]
We use the radial measure from \eqref{eq:joint-radial-measure}:
\[
 d\lambda_E(r,x)=\kappa n r^{\kappa n-1}\,dr\,d\mu_E(x).
\]
For boundary layers we use the logarithmic metric
\eqref{eq:logarithmic-product-metric}:
\[
 d_{\log,E}((r,x),(r',x'))
 =|\log r-\log r'|+d_E(x,x').
\]
Thus, for a Borel set $\Omega\subset\mathcal Z_E$ and $h>0$,
the boundary layer \eqref{eq:joint-boundary-layer} is
\[
 \partial_h\Omega=
 \{z\in\mathcal Z_E:d_{\log,E}(z,\Omega)\le h,
   \ d_{\log,E}(z,\mathcal Z_E\setminus\Omega)\le h\},
\]
where distance to the empty set is $+\infty$.

The normalization in \eqref{eq:height-one-normalization} multiplies
the archimedean norms of each subspace and quotient by $H(E)^{-1/n}$.
Normalizing these spaces to height one cancels this factor. Thus, for
every $d$-dimensional subspace $W\subset E$,
\[
 H_{\widehat E}(W)=H(E)^{-d/n}H_E(W),\qquad
 x_{\widehat E}(W)=x_E(W).
\]
Since $c_{\widehat E}=c_E$ and the local isometry groups are unchanged,
the shape and joint parameter spaces and their measures are preserved.
Using the same Grassmannian metric also preserves $d_E$ and
$d_{\log,E}$.

\subsection{Numerical parameters}\label{subsec:numerical-parameters}

Let $V_j$ be the volume of the Euclidean unit ball in $\R^j$, and set
\begin{equation}\label{eq:archimedean-volume-ratios}
 \mathcal V_{n,d}
 =
 \frac{\displaystyle\prod_{j=n-d+1}^{n}V_j}
      {\displaystyle\prod_{j=2}^{d}V_j},
 \qquad
 \mathcal V^{(2)}_{n,d}
 =
 \frac{\displaystyle\prod_{j=n-d+1}^{n}V_{2j}}
      {\displaystyle\prod_{j=2}^{d}V_{2j}},
\end{equation}
and
\begin{equation}\label{eq:zeta-ratio}
 \mathcal Z_{K,n,d}
 =
 \frac{\displaystyle\prod_{j=n-d+1}^{n}\zeta_K(j)}
      {\displaystyle\prod_{j=2}^{d}\zeta_K(j)},
\end{equation}
with the convention that an empty product is $1$.  Set
\begin{equation}\label{eq:main-term-constant}
    \mathfrak a_{K,n,d}
 =
 \frac{h_KR_K}{w_Kn}
 2^{r_2(de+1)}|D_K|^{-(de+1)/2}
 \binom nd^{s_\infty}
 \mathcal V_{n,d}^{r_1}
 \bigl(\mathcal V^{(2)}_{n,d}\bigr)^{r_2}
 \mathcal Z_{K,n,d}^{-1}.
\end{equation}
 
Set
\begin{equation}\label{eq:basic-parameters}
 \rho=\kappa(n^2-1),\qquad
 q_n=2\left\lceil\frac{n-1}{2}\right\rceil,
\end{equation}
and
\begin{equation}\label{eq:beta}
 \beta=\max\{d-1,e-1\}.
\end{equation}
Set also
\begin{equation*}
 \delta_0=\frac{\kappa n}{2q_n(\rho+1)}.
\end{equation*}
For $n\ge3$, let
\begin{equation}\label{eq:bde}
 b_{d,e}=\begin{cases}
 \min\{\kappa d,\kappa e/(e-1)\},&d,e\ge2,\\
 \kappa e/(e-1),&d=1<e,\\
 \kappa d,&e=1<d.
 \end{cases}
\end{equation}
Let $0<\alpha\le1$. For $n\ge3$, let
$0\le\gamma<b_{d,e}/(\kappa n)$ and
$b_\gamma=b_{d,e}-\gamma\kappa n>0$; for $n=2$, let $\gamma=0$.
Define
\begin{equation}\label{eq:delta-holder-cusp}
 \delta_{\alpha,\gamma}=
 \begin{cases}
 \displaystyle
 \frac{\kappa n\,b_\gamma}
 {2q_n\bigl(b_\gamma+\rho\beta+
                    \rho\max\{b_\gamma/\alpha,\gamma\}\bigr)},&n\ge3,\\[4mm]
 \displaystyle\frac{\kappa n\,\alpha}{2q_n(\rho+\alpha)},&n=2.
 \end{cases}
\end{equation}

\subsection{Successive minima and compactness}\label{subsec:adelic-minima}

For a rigid adelic space $F$, let
\[
 F_\infty=\prod_{v\mid\infty}F_v,
 \qquad
 |x|_F=\max_{v\mid\infty}\|x_v\|_{F,v},
 \qquad
 \mathcal B_F=\{x\in F_\infty:|x|_F\le1\}.
\]
The $\cO_K$-lattice $\mathcal M_F$ in \eqref{eq:integral-module}
embeds diagonally as a full lattice in $F_\infty$.
For a $K$-subspace $S\subset F$ with the induced norms, the localization
property recalled in Section~\ref{subsec:metrics-exponents} gives,
for every finite place $v$,
\begin{equation}\label{eq:subspace-localization}
 \mathcal M_S=\mathcal M_F\cap S,
 \qquad
 \mathcal M_S\otimes_{\cO_K}\cO_v
 =\{x\in S_v:\|x\|_{F,v}\le1\}.
\end{equation}

\begin{lemma}\label{lem:covol-height}
Let $S\subset F$ have dimension $s\ge1$, and let
$\mathcal B_S=\mathcal B_F\cap S_\infty$.  One has
\begin{equation}\label{eq:covolume-height-comparison}
 \frac{\covol(\mathcal M_S)}{\vol(\mathcal B_S)}
 \asymp_{K,s}H(S)^\kappa.
\end{equation}
Here covolume and volume are taken with respect to Lebesgue measure in the
Minkowski coordinates on $S_\infty$.
\end{lemma}

\begin{proof}
Choose a $K$-basis $s_1,\ldots,s_s$ of $S$.  The $\R$-linear isomorphism
\begin{equation}\label{eq:minkowski-coordinate-map}
 \prod_{v\mid\infty}K_v^s\longrightarrow S_\infty,
 \qquad (z_1,\ldots,z_s)\longmapsto\sum_{i=1}^s z_i s_i,
\end{equation}
multiplies $\covol(\mathcal M_S)$ and $\vol(\mathcal B_S)$ by the same
Jacobian, so their ratio may be computed in these coordinates.
For each finite place $v$, choose $A_v\in\GL_s(K_v)$ such that
\begin{equation*}
 \left\{z\in K_v^s:\left\|\sum_{i=1}^s z_i s_i\right\|_{F,v}\le1\right\}
 =A_v\cO_v^s.
\end{equation*}
Then
\begin{equation*}
 \|s_1\wedge\cdots\wedge s_s\|_{\det,v}=|\det A_v|_v^{-1}.
\end{equation*}
By \eqref{eq:subspace-localization}, the preimage of $\mathcal M_S$ in
$K^s$ under \eqref{eq:minkowski-coordinate-map} localizes to $A_v\cO_v^s$ at every
finite place.  Its top exterior power is therefore a fractional ideal
whose localization at $v$ is $(\det A_v)\cO_v$.  The covolume formula for
fractional ideals gives, in these coordinates,
\begin{equation*}
 \covol(\mathcal M_S)
 \asymp_{K,s}\prod_{v\nmid\infty}|\det A_v|_v^{-n_v}.
\end{equation*}

At an archimedean place, a $K_v$-linear change of coordinates with
determinant $a_v$ has real Jacobian $|a_v|_v^{n_v}$.  In the coordinates \eqref{eq:minkowski-coordinate-map}, the local unit ball has volume \( \asymp_s \|s_1\wedge\cdots\wedge s_s\|_{\det,v}^{-n_v} \).
Multiplying the finite and archimedean contributions gives
\begin{equation*}
 \frac{\covol(\mathcal M_S)}{\vol(\mathcal B_S)}
 \asymp_{K,s}
 \prod_{v\nmid\infty}|\det A_v|_v^{-n_v}
 \prod_{v\mid\infty}
 \|s_1\wedge\cdots\wedge s_s\|_{\det,v}^{n_v}
 \asymp_{K,s}H(S)^\kappa,
\end{equation*}
where the last comparison follows from the definition of $H(S)$.
\end{proof}

\begin{lemma}\label{lem:unit-balance}
There is $C_{K,r}\ge1$ such that for every $r$-dimensional rigid adelic
space $F$ and every $K$-line $L\subset F$, there is
$0\ne y_L\in\mathcal M_F\cap L$ such that
\begin{equation}\label{eq:balanced-line-generator}
 \max_{v\mid\infty}\|y_L\|_{F,v}\le C_{K,r}H(L).
\end{equation}
\end{lemma}

\begin{proof}
The rank one projective module $\mathcal M_F\cap L$ determines an ideal
class $\mathfrak c$ of $K$.  For each ideal class $\mathfrak c$, fix an integral ideal
$\mathfrak b_{\mathfrak c}$ representing $\mathfrak c$ and choose
$0\ne\alpha_{\mathfrak c}\in\mathfrak b_{\mathfrak c}$.  Choose
$x\in L\setminus\{0\}$ such that
\begin{equation*}
 \mathcal M_F\cap L=\mathfrak b_{\mathfrak c} x,
\end{equation*}
and let $y=\alpha_{\mathfrak c}x\in\mathcal M_F$.  By
\eqref{eq:subspace-localization}, for every finite place $v$,
\begin{equation*}
 \{z\in L\otimes_KK_v:\|z\|_{F,v}\le1\}
 =\mathfrak b_{\mathfrak c}\cO_v\,x.
\end{equation*}
Choose a uniformizer $\pi_v$ of $K_v$.  Then
$\mathfrak b_{\mathfrak c}\cO_v=
\pi_v^{\operatorname{ord}_v(\mathfrak b_{\mathfrak c})}\cO_v$, and homogeneity of the restricted norm gives
\begin{equation*}
 \|y\|_{F,v}=|\alpha_{\mathfrak c}|_v
 |\pi_v|_v^{-\operatorname{ord}_v(\mathfrak b_{\mathfrak c})}.
\end{equation*}
The product of these finite-place norm factors depends only on $\mathfrak c$.  Since the
class group is finite and the ideals $\mathfrak b_{\mathfrak c}$ and elements
$\alpha_{\mathfrak c}$ are fixed,
\begin{equation*}
 \prod_{v\nmid\infty}\|y\|_{F,v}^{n_v/\kappa}\asymp_K 1.
\end{equation*}
By the definition of $H(L)$ and the preceding estimate for the
nonarchimedean factors,
\begin{equation*}
 \prod_{v\mid\infty}\|y\|_{F,v}^{n_v/\kappa}
 \asymp_K H(L).
\end{equation*}
Moreover,
\begin{equation*}
 \sum_{v\mid\infty}n_v
 \left(
  \log\|y\|_{F,v}
  -\log\prod_{v'\mid\infty}\|y\|_{F,v'}^{n_{v'}/\kappa}
 \right)=0.
\end{equation*}
By Dirichlet's unit theorem,
\(\{(\log|u|_v)_{v\mid\infty}:u\in\cO_K^\times\}\)
is a lattice in this hyperplane.  Hence we may choose
$u\in\cO_K^\times$ so that
\[
 \left(
  \log\|uy\|_{F,v}
  -\log\prod_{v'\mid\infty}\|y\|_{F,v'}^{n_{v'}/\kappa}
 \right)_{v\mid\infty}
\]
lies in a fixed bounded fundamental domain for that lattice.  It follows that
\begin{equation*}
 \max_{v\mid\infty}\|uy\|_{F,v}
 \ll_K \prod_{v\mid\infty}\|y\|_{F,v}^{n_v/\kappa}
 \ll_K H(L).
\end{equation*}
Taking $y_L=uy$ proves \eqref{eq:balanced-line-generator}.
\end{proof}

\begin{lemma}\label{lem:block-minima}
Let $F$ be an $r$-dimensional rigid adelic space over $K$.  Define
recursively $F_0=0$ and, for $1\le i\le r$, choose
$x_i\in\mathcal M_F\setminus F_{i-1}$ such that
\begin{equation}\label{eq:block-minima-recursion}
 |x_i|_F
 =\min\{|x|_F:x\in\mathcal M_F\setminus F_{i-1}\},
 \qquad
 F_i=F_{i-1}+Kx_i.
\end{equation}
Let $\lambda_j(\mathcal M_F,\mathcal B_F)$, $1\le j\le\kappa r$, be the real successive minima
of $(\mathcal M_F,\mathcal B_F)$.  Then:
\begin{enumerate}[label=\textup{(\roman*)}]
\item for $1\le i\le r$ and $1\le j\le\kappa$,
\begin{equation}\label{eq:block-comparison}
 |x_i|_F\le\lambda_{(i-1)\kappa+j}(\mathcal M_F,\mathcal B_F)\ll_K|x_i|_F;
\end{equation}
\item for $1\le i\le r$,
\begin{equation}\label{eq:flag-height-mu}
 H(F_i)\asymp_{K,r}\prod_{j=1}^i|x_j|_F,
 \qquad
 |x_i|_F\asymp_{K,r}\lambda_i(F);
\end{equation}
\item
\begin{equation}\label{eq:minima-height-product}
 H(F)\asymp_{K,r}\prod_{i=1}^r\lambda_i(F).
\end{equation}
\end{enumerate}
\end{lemma}

\begin{proof}
The minimum in \eqref{eq:block-minima-recursion} is attained because
$\mathcal M_F$ is discrete in $F_\infty$, and
$|x_1|_F\le\cdots\le|x_r|_F$.  Since
$\dim_\R(F_{i-1})_\infty=\kappa(i-1)$, among any $\kappa(i-1)+1$ $\R$-linearly independent vectors of $\mathcal M_F$, at least one lies outside
$F_{i-1}$.  Hence
\begin{equation*}
 \lambda_{(i-1)\kappa+1}(\mathcal M_F,\mathcal B_F)\ge|x_i|_F.
\end{equation*}
Fix an integral basis $\omega_1,\ldots,\omega_\kappa$ of $\cO_K$.  The
vectors
\begin{equation*}
 \omega_a x_b\qquad(1\le a\le\kappa,\ 1\le b\le i)
\end{equation*}
are $\R$-linearly independent, belong to $\mathcal M_F$, and satisfy
$|\omega_ax_b|_F\ll_K|x_i|_F$.  Thus
$\lambda_{\kappa i}(\mathcal M_F,\mathcal B_F)\ll_K|x_i|_F$, proving (i).

For $j\le i$, restricting to $\mathcal M_F\cap F_i$ does not change the
recursive minimum defining $x_j$: every vector of
$(\mathcal M_F\cap F_i)\setminus F_{j-1}$ is admissible in
\eqref{eq:block-minima-recursion}, and $x_j\in F_i$ attains the minimum.
Minkowski's second theorem, in the form recorded in
\cite[Appendix~C.2]{BombieriGubler}, together with
Lemma~\ref{lem:covol-height}, gives
\begin{equation*}
 \left(\prod_{j=1}^i|x_j|_F\right)^\kappa
 \asymp_{K,r}
 \frac{\covol(\mathcal M_F\cap F_i)}
      {\vol\bigl(\mathcal B_F\cap(F_i)_\infty\bigr)}
 \asymp_{K,r}H(F_i)^\kappa.
\end{equation*}
Taking $\kappa$-th roots gives the first estimate in \textup{(ii)}.

For every $j\le i$, the finite norms of $x_j$ are at most $1$, and hence
\begin{equation*}
 H(Kx_j)\le |x_j|_F\le|x_i|_F.
\end{equation*}
It follows that $\lambda_i(F)\le|x_i|_F$.  Conversely, let $B>\lambda_i(F)$ and
choose $i$ $K$-linearly independent lines of height at most $B$.  By
Lemma~\ref{lem:unit-balance} they contain $i$ independent vectors of
$\mathcal M_F$ with archimedean norm at most $C_{K,r}B$.  One of these
vectors lies outside $F_{i-1}$, so $|x_i|_F\le C_{K,r}B$.  Letting
$B\to\lambda_i(F)^+$ proves the second estimate in (ii).

Taking $i=r$ in the first estimate of \textup{(ii)} and using the second
estimate for $1\le j\le r$ proves \textup{(iii)}.
\end{proof}

Let $0\ne S\subsetneq E$ and let $s=\dim_KS$.  Lemma~\ref{lem:block-minima}(iii) gives
\begin{equation*}
 H(S)\asymp_{K,s}\prod_{i=1}^s\lambda_i(S)
 \gg_{K,s}\prod_{i=1}^s\lambda_i(E).
\end{equation*}
Since $E$ is fixed and $1\le s<n$, there is $b_E>0$ such that
\begin{equation}\label{eq:beta-E}
 H(S)\ge b_E
 \qquad(0\ne S\subsetneq E).
\end{equation}

\begin{proposition}\label{prop:moduli-compactness}
For $m\ge1$, the function $L\mapsto\lambda_m(L)$ on $\cX_{K,m}$ is
continuous, and its sublevel sets
\begin{equation}\label{eq:compact-moduli-truncation}
 \cX_{K,m}(R)=\{L\in\cX_{K,m}:\lambda_m(L)\le R\}
\end{equation}
are compact for every $R\ge1$.
\end{proposition}

\begin{proof}
Fix $L=[F]\in\cX_{K,m}$ and a representative
$g=(g_\infty,g_f)\in\GL_m(\A_K)^1$.  Given $\delta>0$, choose
an open identity neighborhood $U\subset\GL_m(K_\infty)^1$ so small that
\[
 e^{-\delta}\|y\|_{m,v}\le\|u_vy\|_{m,v}
 \le e^\delta\|y\|_{m,v}
 \qquad(u\in U,\ v\mid\infty,\ y\in K_v^m).
\]
Since $\Kc_{m,f}$ is compact open and has determinant norm one,
$(Ug_\infty)\times(\Kc_{m,f}g_f)$ is open in $\GL_m(\A_K)^1$.
Its image in $\cX_{K,m}$ is a neighborhood of $L$.  Every class $L'$ in this neighborhood has a representative $(ug_\infty,k_fg_f)$ with $u\in U$ and $k_f\in\Kc_{m,f}$. Left multiplication by $(1,k_f^{-1})\in\Kc_m$ gives the representative $(ug_\infty,g_f)$.  Thus $L'$ is represented
by a space $F'$ with the same finite norms as $F$ and
\begin{equation*}
 e^{-\delta}\|x\|_{F,v}
 \le \|x\|_{F',v}
 \le e^{\delta}\|x\|_{F,v}
 \qquad(v\mid\infty).
\end{equation*}
Hence
\begin{equation*}
 e^{-\delta}H_F(x)\le H_{F'}(x)\le e^{\delta}H_F(x)
 \qquad(0\ne x\in K^m).
\end{equation*}
Taking the infimum over independent tuples gives
\[
 e^{-\delta}\lambda_i(F)\le\lambda_i(F')\le e^\delta\lambda_i(F)
 \qquad(1\le i\le m),
\]
which proves continuity.

Now let $L=[F]\in\cX_{K,m}(R)$.  Since $H(F)=1$ and
$\lambda_m(F)\le R$, Lemma~\ref{lem:block-minima}(iii) gives
$\lambda_1(F)\gg_{K,m}R^{-(m-1)}$.  A representative of $L$ lies in
$\GL_m(\A_K)^1$, so the adelic covolume is fixed.  If a nonzero rational
vector lies in the finite unit balls and in archimedean balls of radius
$\epsilon$, then $H_F(x)\le\epsilon$; the lower bound for $\lambda_1(F)$
therefore excludes a fixed neighborhood of the origin.  By McFeat's adelic Mahler criterion
\cite[Part~II, \S3.3, Theorem~3]{McFeat}, the corresponding adelic
lattices form a relatively compact set. Their images in $\cX_{K,m}$
are therefore relatively compact. Since $\lambda_m$ is continuous,
$\cX_{K,m}(R)$ is closed and hence compact.
\end{proof}

For $R\ge2$, let
\begin{equation}\label{eq:shape-truncation}
 \mathcal Y_R=
 \{(L_1,L_2)\in\cY_{E,d}:\lambda_d(L_1)\le R,\ \lambda_e(L_2)\le R\}.
\end{equation}
These sets are compact by Proposition~\ref{prop:moduli-compactness}.
Recall from \eqref{eq:cusp-size} that
$\tau_E(w,(L_1,L_2))=\max\{2,\lambda_d(L_1),\lambda_e(L_2)\}$.
Thus, for $R\ge2$,
\[
 \{x\in\mathcal X_E:\tau_E(x)\le R\}
 =\Gr_{d,\infty}(E)\times\mathcal Y_R.
\]

\subsection{Groups and Haar measures}\label{subsec:group-measures}

Let $\mathbf G=\SL_n$ and $G_v=\SL_n(K_v)$ for $v\mid\infty$, so that
\[
 G=\SL_n(K_\infty)=\prod_{v\mid\infty}G_v.
\]
Fix a $d$-dimensional $K$-subspace $W_0\subset E$.  Choose a $K$-basis
$e_1,\ldots,e_n$ of $E$ adapted to $W_0$, so that
\begin{equation}\label{eq:standard-plane}
 W_0=Ke_1\oplus\cdots\oplus Ke_d.
\end{equation}
Using this basis, identify $\SL(E)$ with $\SL_n$.
Let $\mathbf P_d\subset\SL_n$ be the stabilizer of $W_0$. Relative to
$E=W_0\oplus(Ke_{d+1}\oplus\cdots\oplus Ke_n)$,
\begin{equation}\label{eq:parabolic-block}
 \mathbf P_d
 =\left\{
   \begin{pmatrix}A&B\\0&D\end{pmatrix}:
   A\in\GL_d,\ D\in\GL_e,\ B\in\operatorname{Mat}_{d\times e},\
   \det(A)\det(D)=1
  \right\}.
\end{equation}
Its unipotent radical and Levi subgroup are
\begin{align*}
 \mathbf N_d&=
 \left\{\begin{pmatrix}I_d&X\\0&I_e\end{pmatrix}:
 X\in\operatorname{Mat}_{d\times e}\right\},\\
 \mathbf M_d&=
 \left\{\begin{pmatrix}A&0\\0&D\end{pmatrix}:
 A\in\GL_d,\ D\in\GL_e,\ \det(A)\det(D)=1\right\}.
\end{align*}
so that
\begin{equation}\label{eq:parabolic-Levi}
 \mathbf P_d=\mathbf N_d\rtimes\mathbf M_d.
\end{equation}
Let $\pi_M:\mathbf P_d\to\mathbf M_d$ be the Levi projection and define
\begin{equation}\label{eq:parabolic-character}
 \chi_d\!\left(\begin{pmatrix}A&B\\0&D\end{pmatrix}\right)=\det A.
\end{equation}
Then
\begin{equation*}
 \mathbf S_d=\ker(\chi_d|_{\mathbf M_d})\simeq\SL_d\times\SL_e.
\end{equation*}

At the archimedean places set
\begin{equation*}
 P_{d,\infty}=\prod_{v\mid\infty}\mathbf P_d(K_v),\qquad
 N_{d,\infty}=\prod_{v\mid\infty}\mathbf N_d(K_v),\qquad
 M_{d,\infty}=\prod_{v\mid\infty}\mathbf M_d(K_v).
\end{equation*}
Define
\begin{equation}\label{eq:parabolic-height-character-early}
 P_{d,\infty}^1=
 \left\{p=(p_v)\in P_{d,\infty}:
 \prod_{v\mid\infty}|\chi_d(p_v)|_v^{n_v/\kappa}=1\right\},\qquad M_{d,\infty}^1=M_{d,\infty}\cap P_{d,\infty}^1.
\end{equation}
Then $P_{d,\infty}^1=N_{d,\infty}\rtimes M_{d,\infty}^1$.

Using the local isometry groups in \eqref{eq:det-one-isometry-group}, let
\[
 \Kc_{E,\infty}=\prod_{v\mid\infty}\Kc_{E,v},\qquad
 \Kc_{E,f}=\prod_{v\nmid\infty}\Kc_{E,v}.
\]
At a finite place,
\[
 \Kc_{E,v}=g_{E,v}^{-1}\SL_n(\cO_v)g_{E,v},
\]
where $g_{E,v}$ is the matrix in
\eqref{eq:adelic-norm-representatives}, expressed in the chosen basis.
The stabilizer of $W_{0,\infty}$ in $\Kc_{E,\infty}$ is
\[
 H_{d,E}
 =\prod_{v\mid\infty}\bigl(\mathbf P_d(K_v)\cap\Kc_{E,v}\bigr)
 =P_{d,\infty}\cap\Kc_{E,\infty}.
\]

For $t\in\R$ set
\begin{equation}\label{eq:height-flow}
 a_{d,v}(t)=
 \begin{pmatrix}e^{t/d}I_d&0\\0&e^{-t/e}I_e\end{pmatrix},
 \qquad
 a_d(t)=(a_{d,v}(t))_{v\mid\infty}.
\end{equation}
Thus
$\prod_{v\mid\infty}|\chi_d(a_{d,v}(t))|_v^{n_v/\kappa}=e^t$.
Extending a basis of any $d$-dimensional $K$-subspace of $E$
to a determinant-one basis of $E$ shows that $\SL_n(K)$ acts
transitively on $\Gr_d(E)(K)$. Since $\mathbf P_d(K)$ is the
stabilizer of $W_0$, we obtain the bijection
\begin{equation}\label{eq:rational-Grassmannian}
 \SL_n(K)/\mathbf P_d(K)\longrightarrow\Gr_d(E)(K),
 \qquad \gamma\mathbf P_d(K)\longmapsto\gamma W_0.
\end{equation}

For additive coordinates, use Lebesgue measure at real places, two dimensional
Lebesgue measure at complex places, and the Haar measure satisfying
$\vol(\cO_v)=1$ at finite places. With these conventions, the diagonal copy
of $K$ in $\A_K$ has covolume $2^{-r_2}|D_K|^{1/2}$.
Let $dg_v$ be the local Haar measures defined by the gauge forms in
Appendix~\ref{app:gauge-form}. Factor the Tamagawa measure on $\mathbf G(\A_K)$ as
\begin{equation*}
 dg_{\A}=dg\,dg_f,
\end{equation*}
where
\begin{equation}\label{eq:tamagawa-splitting}
 dg=\prod_{v\mid\infty}dg_v,
 \qquad
 dg_f=2^{r_2(n^2-1)}|D_K|^{-(n^2-1)/2}\prod_{v\nmid\infty}dg_v.
\end{equation}
Since $\SL_n$ is split and simply connected, its Tamagawa number is one
\cite{Langlands}:
\begin{equation*}
 \vol\bigl(\mathbf G(\A_K)/\mathbf G(K)\bigr)=1.
\end{equation*}

We give discrete arithmetic subgroups counting measure and define
quotient measures by Weil's quotient formula
\cite[Theorem~2.51 and (2.52)]{FollandHarmonic}.  Write $\A_{\Q,f}$ for the finite adeles of $\Q$.  Subscripts $\infty$ and
$f$ denote archimedean and finite components.  For every $m$,
$\SL_m(\A_K)$ carries its Tamagawa measure and has Tamagawa number one
\cite{Langlands}.  We do not rescale the adelic or archimedean quotient measures to
have total mass one.  When no confusion
can arise, arithmetic subgroups are identified with their archimedean images.

\section{Reduction to lattice point counting}
\label{sec:equidistribution}

Throughout this section, we assume that $H(E)=1$.

\subsection{Arithmetic sectors and parabolic coordinates}\label{subsec:sector-decomposition}

The group $\Kc_{E,f}$ defined in Section~\ref{subsec:group-measures}
is compact open and commensurable with
$\prod_{v\nmid\infty}\SL_n(\cO_v)$.  Since $\mathbf G=\SL_n$ is simply
connected and $K$-isotropic, strong approximation
\cite[Theorem~7.12]{PlatonovRapinchuk} gives
\begin{equation*}
 \mathbf G(\A_{K,f})=\Kc_{E,f}\mathbf G(K).
\end{equation*}
Set $\Gamma_1=\mathbf G(K)\cap\Kc_{E,f}$. Strong approximation
identifies the following double-coset spaces:
\begin{equation}\label{eq:finite-parabolic-classes}
 \Kc_{E,f}\backslash \mathbf G(\A_{K,f})/\mathbf P_d(K)
 \simeq
 \Gamma_1\backslash\mathbf G(K)/\mathbf P_d(K).
\end{equation}
By \eqref{eq:rational-Grassmannian}, the right hand side is the set of
$\Gamma_1$-orbits on $\Gr_d(E)(K)=(\mathbf G/\mathbf P_d)(K)$.
For $(\mathbf G,\mathbf P_d)=(\SL_n,\mathbf P_d)$, Borel's
finiteness theorem \cite[Theorem~7.3]{BorelAdeles} gives finitely many
$\SL_n(\cO_K)$-orbits on $(\mathbf G/\mathbf P_d)(K)$. Since $\Gamma_1$
is commensurable with $\SL_n(\cO_K)$, it also has finitely many
orbits on $(\mathbf G/\mathbf P_d)(K)$.

Recall from \eqref{eq:parabolic-character} and \eqref{eq:parabolic-height-character-early} the algebraic character
$\chi_d:\mathbf P_d\to\mathbf G_m$ and the subgroup $P_{d,\infty}^1$.
\begin{proposition}\label{prop:finite-covolume}
Let $C_f$ be a compact open subgroup of
$\mathbf P_d(\A_{K,f})$, and set
\begin{equation}\label{eq:parabolic-arithmetic-subgroup}
 \Gamma_P=\mathbf P_d(K)\cap C_f.
\end{equation}
Then, under the archimedean embedding, $\Gamma_P$ is a lattice in
$P_{d,\infty}^1$.
\end{proposition}

\begin{proof}
The image $\chi_d(C_f)$ is a compact subgroup of
$\A_{K,f}^{\times}$, so its finite idelic norm is identically $1$.
For every $\gamma\in\Gamma_P$, the product formula gives
\begin{equation*}
 \prod_{v\mid\infty}|\chi_d(\gamma_v)|_v^{n_v/\kappa}=1,
\end{equation*}
so $\Gamma_P\subset P_{d,\infty}^1$.

Consider the algebraic group over $\Q$
\begin{equation*}
 \mathbf H=
 \ker\!\left(
   \operatorname{Nm}_{K/\Q}\circ
   \operatorname{Res}_{K/\Q}(\chi_d)
 \right)
 \subset \operatorname{Res}_{K/\Q}\mathbf P_d.
\end{equation*}
We have
\begin{equation*}
 \ker\chi_d=\mathbf N_d\rtimes\mathbf S_d,
 \qquad
 \mathbf S_d\simeq\SL_d\times\SL_e.
\end{equation*}
Every $\Q$-rational character of $\mathbf H$ is therefore trivial on
$\operatorname{Res}_{K/\Q}(\ker\chi_d)$, since unipotent and
semisimple groups have no nontrivial algebraic characters.  Let
\begin{equation*}
 \operatorname{Res}^{1}_{K/\Q}\mathbf G_m
 =
 \ker\!\left(
   \operatorname{Nm}_{K/\Q}:
   \operatorname{Res}_{K/\Q}\mathbf G_m\longrightarrow\mathbf G_m
 \right).
\end{equation*}
Since $\chi_d:\mathbf P_d\to\mathbf G_m$ is surjective, restriction of
scalars gives
\begin{equation*}
 \mathbf H/\operatorname{Res}_{K/\Q}(\ker\chi_d)
 \simeq \operatorname{Res}^{1}_{K/\Q}\mathbf G_m.
\end{equation*}
Both $\operatorname{Res}_{K/\Q}(\ker\chi_d)$ and the norm one torus on
the right are connected, so $\mathbf H$ is connected.  Thus every
$\Q$-rational character of $\mathbf H$ factors through this torus, which has
no nontrivial $\Q$-rational characters.  Indeed, for each
embedding $\sigma:K\hookrightarrow\overline{\Q}$ let $e_\sigma$ denote the
corresponding standard basis vector.  Over $\overline{\Q}$ the character
group is
\begin{equation*}
 \left(
   \bigoplus_{\sigma:K\hookrightarrow\overline{\Q}}\Z e_\sigma
 \right)
 \Big/
 \Z\!\left(\sum_\sigma e_\sigma\right),
\end{equation*}
and the absolute Galois group acts transitively on the embeddings
$\sigma:K\hookrightarrow\overline{\Q}$.  Hence this quotient has no
nonzero element fixed by the Galois group, and $\mathbf H$ has no nontrivial
$\Q$-rational character.

The subgroup $\mathbf H(\Q)\cap C_f$ is arithmetic in $\mathbf H(\Q)$.
Indeed, $C_f\cap\mathbf H(\A_{\Q,f})$ is compact open and hence
commensurable with an integral compact open subgroup attached to any
faithful $\Q$-representation of $\mathbf H$.  The
Borel--Harish-Chandra finiteness theorem
\cite[Theorem~9.4, p.~522]{BorelHC} gives
\begin{equation*}
 \vol\!\left(
   \mathbf H(\R)/(\mathbf H(\Q)\cap C_f)
 \right)<\infty.
\end{equation*}

Finally, for $p_\infty=(p_v)_{v\mid\infty}\in P_{d,\infty}$,
\begin{equation*}
 \left(\prod_{v\mid\infty}|\chi_d(p_v)|_v^{n_v/\kappa}\right)^\kappa
 =
 \left|
   \operatorname{Nm}_{K/\Q}\bigl(\chi_d(p_\infty)\bigr)
 \right|.
\end{equation*}
$p_\infty\in P_{d,\infty}^1$ if and only if
$\operatorname{Nm}_{K/\Q}(\chi_d(p_\infty))=\pm1$, whereas
$\mathbf H(\R)$ is the subgroup on which this norm equals $1$.  Hence
\begin{equation*}
 [P_{d,\infty}^1:\mathbf H(\R)]\le2.
\end{equation*}
Restricting this norm homomorphism to $\Gamma_P$ gives
\begin{equation*}
 [\Gamma_P:\mathbf H(\Q)\cap C_f]\le2.
\end{equation*}
$\Gamma_P$ is discrete and has finite covolume in
$P_{d,\infty}^1$.
\end{proof}

Fix representatives $\gamma_1=1,\gamma_2,\ldots,\gamma_J\in\mathbf G(K)$
of \eqref{eq:finite-parabolic-classes}. For $1\le j\le J$, let
\begin{equation}\label{eq:Gammaj-congruence}
 \Gamma_j=\mathbf G(K)\cap\gamma_j^{-1}\Kc_{E,f}\gamma_j,
\end{equation}
\begin{equation}\label{eq:GammaPj}
 \Gamma_{P,j}=\mathbf P_d(K)\cap\Gamma_j.
\end{equation}

\begin{proposition}\label{prop:arithmetic-sectors}
\begin{enumerate}[label=\textup{(\roman*)}]
\item The archimedean images of $\Gamma_j$ and $\Gamma_{P,j}$ are
lattices in $G$ and $P_{d,\infty}^1$, respectively, and
\[
 \vol(G/\Gamma_j)=\vol(G/\Gamma_1)=\vol(\Kc_{E,f})^{-1}.
\]
\item The maps $W(g)=\gamma_jgW_0$ induce a bijection
\begin{equation}\label{eq:same-plane-action}
 \bigsqcup_{j=1}^{J}\Gamma_j/\Gamma_{P,j}
 \xrightarrow{\ \sim\ }\Gr_d(E)(K),
 \qquad (j,g\Gamma_{P,j})\longmapsto W(g).
\end{equation}
\end{enumerate}
\end{proposition}

\begin{proof}

Since $\gamma_j\in\mathbf G(K)$,
$\gamma_j\Gamma_j\gamma_j^{-1}=\Gamma_1$.  For $\Gamma_1$, strong
approximation identifies $\mathbf G(\A_K)/\mathbf G(K)$ with
$(G\times\Kc_{E,f})/\Gamma_1$, with diagonal right action by $\Gamma_1$.
Projection to $G/\Gamma_1$ has fiber $\Kc_{E,f}$.  Weil's quotient formula
and the normalization giving Tamagawa number one therefore give
\begin{equation*}
 1=\vol(\Kc_{E,f})\vol(G/\Gamma_1),
 \qquad
 \vol(G/\Gamma_j)=\vol(G/\Gamma_1)
 =\frac{1}{\vol(\Kc_{E,f})}.
\end{equation*}
The equality of the two quotient volumes follows from invariance of Haar measure under conjugation by $\gamma_{j,\infty}$.  The archimedean
image of every $\Gamma_j$ is a lattice in $G$.

Let $W\subset E$ be a $d$-dimensional $K$-subspace and choose
$\gamma\in\mathbf G(K)$ with $W=\gamma W_0$.  The double coset of
$\gamma$ in \eqref{eq:finite-parabolic-classes} determines a unique $j$.
Multiplying $\gamma$ on the right by an element of $\mathbf P_d(K)$ does
not change $W$, and we may arrange
$\gamma\in\Kc_{E,f}\gamma_j$ at the finite places.  Hence
$\gamma=\gamma_jg$ for some $g\in\Gamma_j$.  Conversely, every
$g\in\Gamma_j$ gives the subspace $\gamma_jgW_0$.

For $g_1,g_2\in\Gamma_j$, the equality
$\gamma_jg_1W_0=\gamma_jg_2W_0$ is equivalent to
$g_2^{-1}g_1\in\mathbf P_d(K)$.  Since $g_2^{-1}g_1\in\Gamma_j$,
\[
 W(g_1)=W(g_2)
 \iff
 g_2^{-1}g_1\in\mathbf P_d(K)\cap\Gamma_j=\Gamma_{P,j}.
\]
The existence and uniqueness of $j$, together with the preceding equivalence, prove the bijection \eqref{eq:same-plane-action}.

Apply Proposition~\ref{prop:finite-covolume} to the compact open subgroup
\begin{equation}\label{eq:finite-parabolic-level}
 C_f
 =\mathbf P_d(\A_{K,f})\cap
   \gamma_j^{-1}\Kc_{E,f}\gamma_j.
\end{equation}
Its rational intersection is
\begin{equation*}
 \mathbf P_d(K)\cap C_f
 =\mathbf P_d(K)\cap\Gamma_j
 =\Gamma_{P,j}.
\end{equation*}
$P_{d,\infty}^1/\Gamma_{P,j}$ has finite volume.
\end{proof}

Let $dk$ be probability Haar measure on $\Kc_{E,\infty}$; its
pushforward under $k\mapsto kW_0$ is $\sigma_{E,d}$. At each
archimedean place, use the orthogonal splitting to define
\[
 a_{d,E,v}(s)=e^{s/d}I_{W_{0,v}}\oplus
                   e^{-s/e}I_{W_{0,v}^{\perp_E}},
 \qquad a_{d,E}(s)=(a_{d,E,v}(s))_{v\mid\infty}.
\]
The compact subgroup $H_{d,E}\subset P_{d,\infty}^1$ commutes
with this flow. A fixed $u_E\in N_{d,\infty}$
maps the chosen complement of $W_0$ to $W_0^{\perp_E}$, so
\[
 a_{d,E}(s)=u_Ea_d(s)u_E^{-1}=a_d(s)n_s,
 \qquad n_s=a_d(-s)u_Ea_d(s)u_E^{-1}\in N_{d,\infty}.
\]
Since $\Kc_{E,\infty}$ is transitive on $G/P_{d,\infty}$,
\[
 G=\Kc_{E,\infty}a_{d,E}(\R)P_{d,\infty}^1.
\]
Conjugation by $a_d(s)$ multiplies every entry of $N_{d,\infty}$
by $e^{ns/(de)}$, with real Jacobian $e^{\kappa ns}$. Thus the
modular character of $P_{d,\infty}$ is
\[
 p\longmapsto
 \left(\prod_{v\mid\infty}|\chi_d(p_v)|_v^{n_v/\kappa}\right)^{\kappa n},
\]
and $P_{d,\infty}^1$ is unimodular. Normalize its Haar measure
$dp^1$ by the generalized Iwasawa formula, valid for nonnegative
Borel functions $f$ on $G$,
\begin{equation}\label{eq:parabolic-integration}
 \int_G f(g)\,dg
 =\int_{\Kc_{E,\infty}}\int_\R\int_{P_{d,\infty}^1}
 f(k a_{d,E}(s)p)e^{\kappa ns}\,dp^1(p)\,ds\,dk.
\end{equation}
Compare \eqref{eq:parabolic-integration} with
\cite[\S1.2]{NakamuraWatanabe}.
Replacing $a_d(s)$ by $a_{d,E}(s)$ leaves this formula unchanged:
the substitution $p\mapsto n_sp$ preserves $dp^1$.

Give $\Gamma_{P,j}$ counting measure and let $d\bar p_j$ be the
Weil quotient measure on
\[
 Q_j=P_{d,\infty}^1/\Gamma_{P,j}.
\]
This measure is finite by Proposition~\ref{prop:arithmetic-sectors}.

For each sector $j$, let $W_j=\gamma_jW_0$ and use the bases
\[
 (\gamma_je_1,\ldots,\gamma_je_d),\qquad
 (\overline{\gamma_je_{d+1}},\ldots,\overline{\gamma_je_n})
\]
on $W_j$ and $E/W_j$, respectively. For each finite place $v$,
choose matrices $A_v\in\GL_d(K_v)$ and $D_v\in\GL_e(K_v)$ with
\[
 \|x\|_{(W_j)_v}=\|A_vx\|_{d,v},\qquad
 \|y\|_{(E/W_j)_v}=\|D_vy\|_{e,v}.
\]
In the notation $A_v,D_v$ we suppress the dependence on the sector $j$.
These matrices are integral for almost all $v$.

At each archimedean place $v$, identify $W_{0,v}$ and $E_v/W_{0,v}$
with $K_v^d$ and $K_v^e$ using the bases
\[
 (e_1,\ldots,e_d),\qquad
 (\overline e_{d+1},\ldots,\overline e_n),
\]
respectively. Fix matrices $B_{d,v}\in\GL_d(K_v)$ and
$B_{e,v}\in\GL_e(K_v)$ satisfying
\[
 \|x\|_{W_0,v}=\|B_{d,v}x\|_{d,v},\qquad
 \|y\|_{E/W_0,v}=\|B_{e,v}y\|_{e,v}.
\]
These archimedean matrices are fixed independently of the sector $j$.

For $p=(p_v)_{v\mid\infty}\in P_{d,\infty}^1$, write
\[
 p_v=\begin{pmatrix}A_v(p)&*\\0&D_v(p)\end{pmatrix},
\]
and define the adelic matrices
\begin{equation}\label{eq:Az-Dz}
 A(p)=\bigl((A_v)_{v\nmid\infty},
                 (B_{d,v}A_v(p))_{v\mid\infty}\bigr),\qquad D(p)=\bigl((D_v)_{v\nmid\infty},
                 (B_{e,v}D_v(p))_{v\mid\infty}\bigr).
\end{equation}
Define $\tau_j$ by
\begin{equation}\label{eq:tau-j}
 \tau_j=\frac1\kappa\log|\det A(p)|_{\A_K}
 =\frac1\kappa\log\left(
 \prod_{v\nmid\infty}|\det A_v|_v^{n_v}
 \prod_{v\mid\infty}|\det B_{d,v}|_v^{n_v}\right),
\end{equation}
where independence of $p$ follows from the defining character of
$P_{d,\infty}^1$. The determinant isometry and $H(E)=1$ give
$|\det D(p)|_{\A_K}=e^{-\kappa\tau_j}$.
Multiply the archimedean components of $A(p)$ and $D(p)$ by
$e^{-\tau_j/d}$ and $e^{\tau_j/e}$, respectively, to obtain
$\widehat A(p)\in\GL_d(\A_K)^1$ and
$\widehat D(p)\in\GL_e(\A_K)^1$. Define
\begin{equation}\label{eq:pi-jY-infty}
 \psi_j(p)=\bigl(\Kc_d\widehat A(p)\GL_d(K),
                 \Kc_e\widehat D(p)\GL_e(K)\bigr).
\end{equation}

\begin{proposition}\label{prop:sector-decomposition}
For each $j$, the map $\psi_j$ in \eqref{eq:pi-jY-infty} descends
to a continuous map $Q_j\to\cY_{E,d}$, denoted by the same symbol.
The following properties hold.
\begin{enumerate}[label=\textup{(\roman*)}]
\item The map
\begin{equation}\label{eq:parabolic-coordinates}
 \Phi_j(k,t,p)=\gamma_{j,\infty}^{-1}k a_{d,E}(t-\tau_j)p,
 \qquad (k,t,p)\in\Kc_{E,\infty}\times\R\times P_{d,\infty}^1,
\end{equation}
is surjective onto $G$, and its fibers are the orbits
\begin{equation}\label{eq:parabolic-coordinate-ambiguity}
 (k,t,p)\longmapsto(kh,t,h^{-1}p),\qquad h\in H_{d,E}.
\end{equation}
Moreover,
\begin{equation}\label{eq:stabilizer-shape-invariance}
 \psi_j(hp)=\psi_j(p)\qquad(h\in H_{d,E}).
\end{equation}
\item If $g\in\Gamma_j$ has archimedean component $\Phi_j(k,t,p)$
and $W=\gamma_jgW_0$, then
\begin{equation}\label{eq:sector-invariants}
 H(W)=e^t,\qquad W_\infty=kW_0,\qquad
 ([\widehat W],[\widehat{E/W}])=\psi_j(p).
\end{equation}
There is $t_0\in\R$ such that $t_0<\log H(W)$ for every
$d$-dimensional subspace $W$ of $E$.
\end{enumerate}
\end{proposition}

\begin{proof}
Fix $j$. The normalizing scalars cancel in the product of determinant
classes, whose value is $c_E=\det([\widehat E])$. Thus
$\psi_j(p)\in\cY_{E,d}$. Right multiplication by
$\gamma\in\Gamma_{P,j}$ changes the two global bases by its
diagonal blocks. At the finite places, $\gamma_j\gamma\gamma_j^{-1}\in\Kc_{E,f}$; hence these changes of basis preserve the finite norms.
Hence $\psi_j(p\gamma)=\psi_j(p)$, and \eqref{eq:pi-jY-infty}
defines a continuous map on $Q_j$. An element of $H_{d,E}$ acts
isometrically on $W_0$ and $E/W_0$, proving
\eqref{eq:stabilizer-shape-invariance}.

The Iwasawa decomposition gives surjectivity of $\Phi_j$.
For a fixed $g=\Phi_j(k,t,p)$, the position of $\gamma_{j,\infty}gW_0$ determines the coset $kH_{d,E}$. The parabolic height character determines $t$. Since $H_{d,E}$
commutes with $a_{d,E}$, replacing $k$ by $kh$ replaces $p$
by $h^{-1}p$. This proves (i).

For $g\in\Gamma_j$, the finite components satisfy
$\gamma_jg=k_f\gamma_j$ with $k_f\in\Kc_{E,f}$.
Consequently the finite subspace and quotient norms are
represented by the fixed $A_v,D_v$. At infinity,
$\gamma_{j,\infty}g=k a_{d,E}(t-\tau_j)p$ gives the position
$kW_0$ and norm matrices
\[
 e^{(t-\tau_j)/d}B_{d,v}A_v(p),\qquad
 e^{-(t-\tau_j)/e}B_{e,v}D_v(p).
\]
Their heights are $e^t$ and $e^{-t}$; normalizing them gives
$\widehat A(p),\widehat D(p)$. This proves
\eqref{eq:sector-invariants}. Finally, choose $t_0<\log b_E$
using \eqref{eq:beta-E}.
\end{proof}

By Propositions~\ref{prop:arithmetic-sectors} and
\ref{prop:sector-decomposition}, every $d$-dimensional
$K$-subspace $W\subset E$ corresponds to a unique sector $j$ and
a unique coset $g\Gamma_{P,j}\in\Gamma_j/\Gamma_{P,j}$, with
$W=\gamma_jgW_0$. If the archimedean component of $g$ is
$\Phi_j(k,t,p)$, then $t=\log H(W)$, $W_\infty=kW_0$,
and $y_W=\psi_j(p)$.

Thus, for $B>0$ and Borel sets
$\Omega_1\subset\Gr_{d,\infty}(E)$ and
$\Omega_2\subset\cY_{E,d}$,
\[
\begin{aligned}
 &\#\left\{W\subset E:
 \begin{array}{l}
 \dim_K W=d,\quad H(W)\le B,\\
 W_\infty\in\Omega_1,\quad y_W\in\Omega_2
 \end{array}\right\}\\
 &\qquad=\sum_{j=1}^{J}
 \#\left\{g\Gamma_{P,j}\in\Gamma_j/\Gamma_{P,j}:
 \begin{array}{l}
 t\le\log B,\quad kW_0\in\Omega_1,\\
 \psi_j(p)\in\Omega_2
 \end{array}\right\}.
\end{aligned}
\]
Since the sector data form a fixed finite family, constants depending
on these data can be taken uniformly in $j$.

\subsection{Parabolic measure and the main term}\label{subsec:parabolic-measure}

Let $d\bar g_j$ be the quotient of $dg$ by counting measure on
$\Gamma_{P,j}$. Applying Weil's quotient formula
\cite[Theorem~2.51 and (2.52)]{FollandHarmonic} to \eqref{eq:parabolic-integration} and
setting $s=t-\tau_j$ gives, for every nonnegative Borel function
$\varphi$ on $G/\Gamma_{P,j}$,
\begin{equation}\label{eq:full-sector-quotient}
 \begin{aligned}
 \int_{G/\Gamma_{P,j}}\varphi\,d\bar g_j
 &=e^{-\kappa n\tau_j}
 \int_{\Kc_{E,\infty}}\int_\R\int_{Q_j}
 \varphi(k a_{d,E}(t-\tau_j)p\Gamma_{P,j})
 e^{\kappa nt}\,d\bar p_j\,dt\,dk.
 \end{aligned}
\end{equation}
To obtain \eqref{eq:full-sector-quotient}, periodize a
nonnegative function on $G$ over $\Gamma_{P,j}$ and apply Weil's
quotient formula first on $P_{d,\infty}^1$ and then on $G$.

Recall the probability measure $\nu_{E,d}$ from
\eqref{eq:joint-moduli-measure} and the constant
$\mathfrak a_{K,n,d}$ from \eqref{eq:main-term-constant}.
\begin{proposition}\label{prop:parabolic-mass}
As finite Borel measures on $\cY_{E,d}$,
\begin{equation}\label{eq:parabolic-mass-W}
 \vol(\Kc_{E,f})\sum_{j=1}^J e^{-\kappa n\tau_j}
 (\psi_j)_*d\bar p_j
 =\kappa n\,\mathfrak a_{K,n,d}\,\nu_{E,d}.
\end{equation}
\end{proposition}
The proof is given in Appendix~\ref{app:gauge-form}.

To identify the leading coefficient, let
$\Omega_1\subset\Gr_{d,\infty}(E)$ and $\Omega_2\subset\cY_{E,d}$
be Borel sets, and let $B\ge e^{t_0}$. By
Proposition~\ref{prop:arithmetic-sectors},
$\vol(G/\Gamma_j)^{-1}=\vol(\Kc_{E,f})$.
Evaluating \eqref{eq:parabolic-mass-W} on $\Omega_2$ therefore gives
\[
 \sum_{j=1}^J\frac{e^{-\kappa n\tau_j}}{\vol(G/\Gamma_j)}
 \bigl((\psi_j)_*d\bar p_j\bigr)(\Omega_2)
 =\kappa n\,\mathfrak a_{K,n,d}\,\nu_{E,d}(\Omega_2).
\]
Since $k\mapsto kW_0$ pushes $dk$ forward to $\sigma_{E,d}$,
\eqref{eq:full-sector-quotient} gives the sum of the sector volumes
with $t_0\le t\le\log B$, $kW_0\in\Omega_1$, and
$\psi_j(p)\in\Omega_2$, each divided by $\vol(G/\Gamma_j)$, as
\[
 \kappa n\,\mathfrak a_{K,n,d}
 \sigma_{E,d}(\Omega_1)\nu_{E,d}(\Omega_2)
 \int_{t_0}^{\log B}e^{\kappa nt}\,dt
 =\mathfrak a_{K,n,d}
 \bigl(B^{\kappa n}-e^{\kappa nt_0}\bigr)
 \mu_E(\Omega_1\times\Omega_2).
\]
Thus the volume main term is
\[
 \mathfrak a_{K,n,d}B^{\kappa n}
 \mu_E(\Omega_1\times\Omega_2)+O_E(1).
\]
The subsequent counting and cusp estimates show that the difference
between the arithmetic count and this volume is $o(B^{\kappa n})$
when $\mu_E(\partial(\Omega_1\times\Omega_2))=0$.
For the whole space $\mathcal X_E$, the same estimates give the total
count $N_{E,d}(B)=\mathfrak a_{K,n,d}B^{\kappa n}+o(B^{\kappa n})$.
Dividing by this total count yields, as $B\to\infty$,
\[
 \frac1{N_{E,d}(B)}
 \sum_{\substack{W\subset E,\ \dim_KW=d\\H(W)\le B}}
 \mathbf1_{\Omega_1\times\Omega_2}(x_E(W))
 \longrightarrow\mu_E(\Omega_1\times\Omega_2).
\]

\subsection{Quantitative reduction and perturbation estimates}\label{subsec:sector-charts}

Fix an $\operatorname{Ad}(H_{d,E})$-invariant norm on
the Lie algebra of $G$, obtained by averaging an inner product over
the compact group $H_{d,E}$. For sufficiently small
$\epsilon>0$, let
\begin{equation}\label{eq:small-identity-neighborhood}
 O_\epsilon=\exp\{X\in\operatorname{Lie}(G):\|X\|<\epsilon\}\subset G.
\end{equation}
For each archimedean place $v$ and each $K_v$-linear map $A_v:E_v\to E_v$, define the operator norm by
\[
 \|A_v\|_{\op,v}
 :=\sup_{0\ne x\in E_v}\frac{\|A_vx\|_{E,v}}{\|x\|_{E,v}}
 =\sup_{\|x\|_{E,v}=1}\|A_vx\|_{E,v}.
\]
For an archimedean tuple $A=(A_v)_{v\mid\infty}$, set
\[
 \|A\|_{\op}:=\max_{v\mid\infty}\|A_v\|_{\op,v}.
\]

Restrict the Lie algebra norm to $\operatorname{Lie}(P_{d,\infty}^1)$
and set
\[
 V_h=\exp\{X\in\operatorname{Lie}(P_{d,\infty}^1):\|X\|<h\}.
\]
For $p\in P_{d,\infty}^1$, the norm of
$\operatorname{Ad}(p):X\mapsto pXp^{-1}$ is computed on
$\operatorname{Lie}(G)$ with this norm.

For a rigid adelic space $F$, use the norm matrices $g_{F,v}$ from
\eqref{eq:adelic-norm-representatives} relative to a fixed $K$-basis.

\begin{proposition}\label{prop:controlled-representatives}
Let $F$ be an $r$-dimensional rigid adelic space with $H(F)=1$, and let
\[
 M=\{x\in F:\|x\|_{F,v}\le1\text{ for every }v\nmid\infty\}.
\]
If $R\ge2$ and $\lambda_r(F)\le R$, then there exists
$\alpha\in\operatorname{Aut}_{\cO_K}(M)$ such that
\begin{equation}\label{eq:short-integral-coordinates}
 \|g_{F,v}\alpha\|_{\op}\ll_{K,M}R,\qquad
 |\det(g_{F,v}\alpha)|_v\asymp_{K,M}1,\qquad
 \|(g_{F,v}\alpha)^{-1}\|_{\op}\ll_{K,M}R^{r-1}
\end{equation}
for every $v\mid\infty$.
\end{proposition}

\begin{proof}
Let $g=(g_{F,v})_{v\mid\infty}$ and view $\Lambda=gM$ as a lattice
and an $\cO_K$-module in $K_\infty^r$, with trace norm
\[
 \|z\|_{\mathrm{tr}}^2
 =\sum_{v\mid\infty}n_v\|z_v\|_2^2.
\]
By Lemma~\ref{lem:block-minima}(ii) and equivalence of the trace and
maximum norms, $\Lambda$ contains $r$ $K$-linearly independent vectors of trace norm $\ll_{K,r}R$.
Also Lemma~\ref{lem:covol-height}, after applying $g$, gives
$\covol(\Lambda)\asymp_{K,r}H(F)^\kappa=1$.

The short pseudobasis theorem of Fieker--Stehl\'e
\cite[Corollary~1]{FiekerStehle}, which applies to modules in
$K_\infty^r$, gives a decomposition
\[
 \Lambda=\bigoplus_{i=1}^r\mathfrak a_i b_i,\qquad
 b_i\in\Lambda,\quad \cO_K\subset\mathfrak a_i,\quad
 N(\mathfrak a_i)^{-1}\ll_K1,
\]
with
\begin{equation}\label{eq:short-pseudobasis-bounds}
 \max_i\|b_i\|_{\mathrm{tr}}\ll_{K,r}R,
 \qquad
 \prod_{i=1}^r\|b_i\|_{\mathrm{tr}}
 \ll_{K,r}\covol(\Lambda)^{1/\kappa}\ll_{K,r}1.
\end{equation}
The ideals $\mathfrak a_i^{-1}$ are integral and have bounded norm,
so only finitely many tuples
$\mathfrak a=(\mathfrak a_1,\ldots,\mathfrak a_r)$ can occur.
For each such tuple with $\bigoplus_i\mathfrak a_i\simeq M$ as
$\cO_K$-modules, fix a matrix $C_{\mathfrak a}\in\GL_r(K)$, depending
only on $M$ and $\mathfrak a$, such that
\[
 C_{\mathfrak a}\Bigl(\bigoplus_{i=1}^r\mathfrak a_i\Bigr)=M.
\]

For the short pseudobasis above, write $b_i=gx_i$ and let
$X=(x_1,\ldots,x_r)\in\GL_r(K)$. Pulling back the decomposition of
$\Lambda=gM$ gives
\[
 M=\bigoplus_{i=1}^r\mathfrak a_i x_i
   =X\Bigl(\bigoplus_{i=1}^r\mathfrak a_i\Bigr).
\]
Thus its ideal tuple admits one of the fixed matrices
$C_{\mathfrak a}$. Let $\alpha=XC_{\mathfrak a}^{-1}$. Then
\[
 \alpha M=XC_{\mathfrak a}^{-1}M
         =X\Bigl(\bigoplus_{i=1}^r\mathfrak a_i\Bigr)=M,
\]
so $\alpha\in\operatorname{Aut}_{\cO_K}(M)$.

The columns of $g_{F,v}X$ are $b_{i,v}$. The first bound in
\eqref{eq:short-pseudobasis-bounds} and the finiteness of the fixed
collection of matrices $C_{\mathfrak a}$ give
\[
 \|g_{F,v}\alpha\|_{\op}
 \le\|g_{F,v}X\|_{\op}\|C_{\mathfrak a}^{-1}\|_{\op}
 \ll_{K,M}\max_i\|b_i\|_{\mathrm{tr}}\ll_{K,M}R.
\]
Hadamard's inequality and the product bound give
\[
 |\det(g_{F,v}\alpha)|_v
 \ll_{K,M}\prod_i\|b_{i,v}\|_2
 \le\prod_i\|b_i\|_{\mathrm{tr}}\ll_{K,M}1.
\]
Since $\alpha$ preserves $M$, its determinant is a unit.  The
product formula and the real Jacobian of $g$ therefore give
\[
 \prod_{v\mid\infty}|\det(g_{F,v}\alpha)|_v^{n_v}
 =\prod_{v\mid\infty}|\det g_{F,v}|_v^{n_v}
 =\frac{\covol(\Lambda)}{\covol(M)}\asymp_{K,M}1.
\]
The product identity and the upper bounds give
$|\det(g_{F,v}\alpha)|_v\gg_{K,M}1$ at each place. Hence
\[
 \|(g_{F,v}\alpha)^{-1}\|_{\op}
 \le\frac{\|g_{F,v}\alpha\|_{\op}^{r-1}}{|\det(g_{F,v}\alpha)|_v}
 \ll_{K,M}R^{r-1}.
\]
\end{proof}

Recall that $\pi_M:\mathbf P_d\to\mathbf M_d$ is the Levi projection
and that $C_f$ is the compact open subgroup in
\eqref{eq:finite-parabolic-level}.  Let
\[
 \Gamma_N=\Gamma_{P,j}\cap\mathbf N_d(K),\qquad
 \Gamma_M=\pi_M(\Gamma_{P,j}),\qquad L_f=\pi_M(C_f).
\]
Additive strong approximation for $\mathbf N_d\simeq\mathbf G_a^{de}$ gives
\begin{equation}\label{eq:GammaM-arithmetic}
 \Gamma_M=\mathbf M_d(K)\cap L_f.
\end{equation}
Indeed, for $m$ in the right hand side, the nonempty open set of
$n_f\in\mathbf N_d(\A_{K,f})$ with $n_fm\in C_f$ contains a rational
point.  Also $\Gamma_N=\mathbf N_d(K)\cap C_f$ is a full lattice in
$N_{d,\infty}$.  Thus
\begin{equation}\label{eq:GammaPNM-exact}
 1\longrightarrow\Gamma_N\longrightarrow\Gamma_{P,j}
 \xrightarrow{\ \pi_M\ }\Gamma_M\longrightarrow1.
\end{equation}
With $W_j=\gamma_jW_0$ and the local matrices $A_v,D_v$
chosen before Proposition~\ref{prop:sector-decomposition}, we have
\[
 \|x\|_{(W_j)_v}=\|A_vx\|_{d,v},\qquad
 \|y\|_{(E/W_j)_v}=\|D_vy\|_{e,v}
\]
in the bases chosen there.  Their unit balls define the projective $\cO_K$-modules
\begin{equation}\label{eq:sector-block-modules}
 \begin{aligned}
 M_1&=\{x\in K^d:A_vx\in\cO_v^d\text{ for all }v\nmid\infty\},\\
 M_2&=\{x\in K^e:D_vx\in\cO_v^e\text{ for all }v\nmid\infty\}.
 \end{aligned}
\end{equation}
Let
\[
 \Delta=\operatorname{Aut}_{\cO_K}(M_1)\times
          \operatorname{Aut}_{\cO_K}(M_2),\qquad
 \Delta^1=\{(\alpha,\beta)\in\Delta:
                   \det\alpha\det\beta=1\}.
\]
Let $C_{b,f}$ be the compact open stabilizer of the finite local
modules of $M_b$, for $b=1,2$.
Elements of $C_f$ preserve the restricted and quotient finite norms, so
$L_f\subset(C_{1,f}\times C_{2,f})\cap\mathbf M_d(\A_{K,f})$.
Both groups are compact open in $\mathbf M_d(\A_{K,f})$; hence this
inclusion has finite index.  Taking rational intersections in
\eqref{eq:GammaM-arithmetic} gives
\begin{equation}\label{eq:GammaM-finite-index}
 \Gamma_M\subset\Delta^1,\qquad [\Delta^1:\Gamma_M]<\infty.
\end{equation}

\begin{proposition}\label{cor:controlled-parabolic-representatives}
\label{prop:sector-properness}
Fix $j$ and recall that $\beta=\max\{d-1,e-1\}$.
\begin{enumerate}[label=\textup{(\roman*)}]
\item For every $R\ge2$, each point of
$\psi_j^{-1}(\mathcal Y_{3R})$ has a representative $p$ satisfying
\begin{equation}\label{eq:conditioned-lift}
 \|p\|_{\op}+\|p^{-1}\|_{\op}\ll_E R^\beta,\qquad
 \|\operatorname{Ad}p\|+\|\operatorname{Ad}p^{-1}\|\ll_E R^{2\beta}.
\end{equation}
\item The map $\psi_j$ is proper.
\end{enumerate}
\end{proposition}

\begin{proof}
For (i), take a representative $p_0$ and write its Levi component as
$\pi_M(p_0)=\mu_0=\operatorname{diag}(U,V)$.
The two normalized blocks have archimedean norm matrices
\[
 e^{-\tau_j/d}B_{d,v}U_v,\qquad
 e^{\tau_j/e}B_{e,v}V_v.
\]
Their heights are one and their last minima are at most $3R$.
Applying Proposition~\ref{prop:controlled-representatives} to each
block, and absorbing the fixed norm matrices into the constants,
gives $a=(\alpha_1,\alpha_2)\in\Delta$ such that
\[
\begin{gathered}
\|\mu_0a\|_{\op}+\|(\mu_0a)^{-1}\|_{\op}\ll_E R^\beta,\\
|\det(U_v\alpha_1)|_v\asymp_E1,\qquad
|\det(V_v\alpha_2)|_v\asymp_E1.
\end{gathered}
\]
For a rank-one block, the determinant bound controls both operator norms.

Because $\det U_v\det V_v=1$, the unit
$u=\det\alpha_1\det\alpha_2$ has bounded archimedean embeddings,
independently of $R$.  Thus only finitely many $u$ occur.
For each such value $u$, fix $\delta_u\in\Delta$ satisfying
$\det\delta_u=\det a$. Then $a\delta_u^{-1}\in\Delta^1$.
By \eqref{eq:GammaM-finite-index}, write
$a\delta_u^{-1}=\gamma\eta$ with $\gamma\in\Gamma_M$ and
$\eta$ in a fixed finite set of representatives for
$\Gamma_M\backslash\Delta^1$. Hence
\[
 \mu:=\mu_0\gamma=\mu_0a\delta_u^{-1}\eta^{-1},\qquad
 \|\mu\|_{\op}+\|\mu^{-1}\|_{\op}\ll_E R^\beta.
\]
Lift $\gamma$ to $\widetilde\gamma\in\Gamma_{P,j}$ using
\eqref{eq:GammaPNM-exact} and write $p_0\widetilde\gamma=\mu n$.
Right multiplication by $\Gamma_N$ puts $n$ in a fixed compact
fundamental parallelepiped for $N_{d,\infty}/\Gamma_N$.
The resulting representative $p$ satisfies the first estimate in
\eqref{eq:conditioned-lift}.  The second follows from
$\|\operatorname{Ad}(p^{\pm1})\|
\ll_E\|p\|_{\op}\|p^{-1}\|_{\op}$.

For (ii), let $C\subset\cY_{E,d}$ be compact. Continuity of the
last minima gives $C\subset\mathcal Y_{3R}$ for some fixed $R\ge2$.
By (i), the inverse image $\psi_j^{-1}(C)$ is contained in the
image in $Q_j$ of
\[
 \{p\in P_{d,\infty}^1:
      \|p\|_{\op}+\|p^{-1}\|_{\op}\le C_E R^\beta\}.
\]
This matrix set is compact: it is bounded, and the bound on the
inverse prevents singular limits. Since $\psi_j$ is continuous,
$\psi_j^{-1}(C)$ is a closed subset of the image of this compact set.
\end{proof}

\begin{proposition}\label{prop:norm-shape-stability}
\begin{enumerate}[label=\textup{(\roman*)}]
\item For every $m\ge1$, there are constants $C_{K,m}>0$ and
$0<\epsilon_0<1$ such that, for all $L,L'\in\cX_{K,m}$ with
$d_m(L,L')\le\epsilon_0$,
\begin{equation}\label{eq:log-minimum-local-lipschitz}
 |\log\lambda_m(L)-\log\lambda_m(L')|
 \le C_{K,m}d_m(L,L').
\end{equation}
\item There are constants $C_E,h_0>0$ such that, for every sector
$j$, $p\in P_{d,\infty}^1$, $R\ge1$, $0<h\le h_0$, and $b\in V_h$,
\begin{equation}\label{eq:shape-stability}
 \operatorname{dist}(\psi_j(bp),\psi_j(p))\le C_Eh,
 \qquad
 \psi_j(p)\in\mathcal Y_R
 \ \Longrightarrow\
 \psi_j(bp)\in\mathcal Y_{e^{C_Eh}R}.
\end{equation}
Here $\operatorname{dist}$ is the product metric on $\cY_{E,d}$
from Section~\ref{subsec:metrics-exponents}.
\end{enumerate}
\end{proposition}

\begin{proof}
For (i), choose $q,q'\in\mathcal S_M$ representing two points in
the same module component. Write $\|x\|_{q,v}=q_v(x,x)^{1/2}$,
and similarly for $q'$.
Along a piecewise smooth path $q(s)$ in $\mathcal S_M$, let
$\left\|\frac{dq}{ds}(s)\right\|$ denote its speed in the metric
defining $D_M$. For $x\ne0$,
\[
 \left|\frac{d}{ds}\log\|x\|_{q(s),v}\right|
 =\frac{\left|\frac{dq_v}{ds}(s)(x,x)\right|}{2q_v(s)(x,x)}
 \le C\left\|\frac{dq}{ds}(s)\right\|.
\]
The differential inequality follows from the operator norm
bound for symmetric or Hermitian forms. The transitive group action makes the constant in this differential inequality uniform on $\mathcal S_M$.
Integrating and taking the infimum over paths from $q$ to $q'$ gives,
for $D_M(q,q')\le\epsilon$,
\[
 e^{-C\epsilon}\|x\|_{q,v}
 \le\|x\|_{q',v}
 \le e^{C\epsilon}\|x\|_{q,v}
 \qquad(v\mid\infty).
\]
Since the finite norms agree, the projective heights satisfy the same bounds. Taking infima over independent tuples gives
\[
 e^{-C\epsilon}\lambda_m(F_{M,q})
 \le\lambda_m(F_{M,q'})
 \le e^{C\epsilon}\lambda_m(F_{M,q}).
\]
Apply this to $q$ and $\gamma\cdot q'$ and take the infimum over
$\gamma\in\operatorname{Aut}_{\cO_K}(M)$ in
\eqref{eq:moduli-distance}. This proves
\eqref{eq:log-minimum-local-lipschitz}. There are finitely many
module components, so the constants may be chosen uniformly.
The condition $\epsilon_0<1$ excludes distinct components.

For (ii), the definitions give, for $v\mid\infty$,
\[
\begin{aligned}
 \widehat A(bp)_v
 &=\bigl(B_{d,v}A_v(b)B_{d,v}^{-1}\bigr)\widehat A(p)_v,\\
 \widehat D(bp)_v
 &=\bigl(B_{e,v}D_v(b)B_{e,v}^{-1}\bigr)\widehat D(p)_v.
\end{aligned}
\]
Both parenthesized factors are $I+O_E(h)$. The block tuples of $p$
act isometrically on $\mathcal S_{M_1}$ and $\mathcal S_{M_2}$,
reducing both metric comparisons to $p=I$. Smoothness of the Gram
map at the fixed metrics therefore gives
$\operatorname{dist}(\psi_j(bp),\psi_j(p))\ll_E h$, uniformly in $p$.
Choose $h_0$ small enough that (i) applies to both components.
Their last minima then change by factors between $e^{-C_Eh}$ and
$e^{C_Eh}$, proving \eqref{eq:shape-stability}.
\end{proof}

\begin{proposition}\label{prop:uniform-charts}
Fix $t_{\min}\in\R$. There are $C,c>0$ such that the following
holds. Let $1\le j\le J$, $k\in\Kc_{E,\infty}$,
$t\ge t_{\min}$, and $p\in P_{d,\infty}^1$. Suppose that
$\epsilon_L,\epsilon_R>0$ satisfy
\[
 \epsilon_L+\max\{1,\|p\|_{\op},\|p^{-1}\|_{\op}\}^2\epsilon_R<c.
\]
Then, for every $u_L\in O_{\epsilon_L}$ and $u_R\in O_{\epsilon_R}$,
there exist $k'\in\Kc_{E,\infty}$, $t'\in\mathbb R$, and
$b\in P_{d,\infty}^1$ such that
\[
 u_L\Phi_j(k,t,p)u_R=\Phi_j(k',t',bp)
\]
and
\begin{equation}\label{eq:local-coordinate-stability}
 \|k'-k\|_{\op}+|t'-t|+\|b-I\|_{\op}
 \le C\bigl(\epsilon_L+
       \max\{1,\|p\|_{\op},\|p^{-1}\|_{\op}\}^2\epsilon_R\bigr).
\end{equation}
The constants depend only on $E,t_{\min}$ and the fixed norms.
\end{proposition}

\begin{proof}
Use the orthogonal splitting
$E_v=W_{0,v}\oplus W_{0,v}^{\perp_E}$ for the block calculations
at every archimedean place. The resulting matrix norms are
equivalent to the fixed norms, with constants depending only on $E$.
Choose a smooth section $k_0$ of
$\Kc_{E,\infty}\to\Gr_{d,\infty}(E)$ near $W_0$, with
$k_0(W_0)=I$. Such a section follows from the inverse function
theorem applied to $X\mapsto\exp(X)W_0$ on a complement to
$\operatorname{Lie}(H_{d,E})$.

For left multiplication, let
$v=k^{-1}\gamma_{j,\infty}u_L\gamma_{j,\infty}^{-1}k$.
Compactness of $\Kc_{E,\infty}$ gives $v=I+O_E(\epsilon_L)$.
Set $k'=k k_0(vW_0)$ and
\[
 q=k'^{-1}\gamma_{j,\infty}u_L\gamma_{j,\infty}^{-1}k
   \in P_{d,\infty},\qquad q=I+O_E(\epsilon_L).
\]
Writing $s=t-\tau_j$ and
\[
 \sigma=\log\!\left(
       \prod_{v\mid\infty}|\chi_d(q_v)|_v^{n_v/\kappa}\right),
\]
we have $\sigma=O_E(\epsilon_L)$ and
\[
 q a_{d,E}(s)=a_{d,E}(s+\sigma)b,
 \qquad b=a_{d,E}(-s-\sigma)q a_{d,E}(s)\in P_{d,\infty}^1.
\]
The diagonal blocks of $b$ are $I+O_E(\epsilon_L)$, while its
upper right blocks are $O_E(e^{-ns/(de)}\epsilon_L)$.
Since $s\ge t_{\min}-\tau_j$,
$b=I+O_{E,t_{\min}}(\epsilon_L)$ uniformly in $t$.
Also $k'-k=O_E(\epsilon_L)$.

For right multiplication, write $pu_R=vp$, where
\[
 v=pu_Rp^{-1}=I+O_E(\epsilon),\qquad
 \epsilon=\max\{1,\|p\|_{\op},\|p^{-1}\|_{\op}\}^2\epsilon_R.
\]
At each archimedean place, write
$v=\left(\begin{smallmatrix}I+A&B\\C&I+D\end{smallmatrix}\right)$
and let $\zeta=e^{-ns/(de)}$. The first $d$ columns of $a_{d,E}(s)v$
span the graph of $\zeta C(I+A)^{-1}$. Thus
\[
 k'=k k_0(a_{d,E}(s)vW_0),\qquad
 k^{-1}k'=I+O_E(\zeta\epsilon).
\]
Conjugation by $a_{d,E}(-s)$ multiplies the lower left block by
$\zeta^{-1}$, the upper right block by $\zeta$, and leaves the
diagonal blocks unchanged. Since $\zeta$ is bounded above,
\[
 q=a_{d,E}(-s)k'^{-1}k a_{d,E}(s)v
   \in P_{d,\infty},\qquad q=I+O_{E,t_{\min}}(\epsilon).
\]
The inclusion follows because $k'W_0=k a_{d,E}(s)vW_0$.
With $\sigma$ defined from this $q$ as above,
\[
 k a_{d,E}(s)v=k'a_{d,E}(s+\sigma)b,
 \qquad b=a_{d,E}(-\sigma)q=I+O_{E,t_{\min}}(\epsilon).
\]
To combine the perturbations, apply the left estimate first.
Its new parabolic coordinate $b_Lp$
satisfies
\[
 \max\{1,\|b_Lp\|_{\op},\|(b_Lp)^{-1}\|_{\op}\}
 \ll_{E,t_{\min}}\max\{1,\|p\|_{\op},\|p^{-1}\|_{\op}\},
\]
and its height coordinate is at least $t_{\min}-1$ after decreasing
$c$. The right estimate then applies with this lower bound.
The total parabolic multiplier satisfies
\[
 b_Rb_L-I=(b_R-I)b_L+(b_L-I)
 =O_{E,t_{\min}}\!\left(\epsilon_L+
       \max\{1,\|p\|_{\op},\|p^{-1}\|_{\op}\}^2\epsilon_R\right).
\]
Together with the bounds for the compact factor and the height
coordinate, this proves \eqref{eq:local-coordinate-stability}.
\end{proof}

\section{Spectral estimates and lattice counting}\label{sec:spectral-smoothing}

Throughout this section, we assume that $H(E)=1$.

\subsection{Preliminaries on unitary representations}\label{subsec:lattice-counting}

On a quotient $G/\Gamma$ we use the measure $d\bar x$ induced by the
fixed Haar measure $dg$.  Write
\[
\begin{aligned}
\|f\|_2&=\left(\int_{G/\Gamma}|f(x)|^2\,d\bar x\right)^{1/2},
 \\
L^2_0(G/\Gamma)&=
 \left\{f\in L^2(G/\Gamma):\int_{G/\Gamma}f\,d\bar x=0\right\}.
\end{aligned}
\]
Left translation acts by
$\pi_\Gamma(g)f(x)=f(g^{-1}x)$; its restriction to the mean-zero subspace
is denoted by $\pi_\Gamma^0$.  Inner products use the same quotient measure.

For a unitary representation $\pi$ of $G$ on a Hilbert space
$\mathcal H$, a matrix coefficient is a function
$g\mapsto\langle\pi(g)\xi,\eta\rangle$ with
$\xi,\eta\in\mathcal H$.  The representation is $L^{p+}$ if there
is a dense subspace $\mathcal H_0\subset\mathcal H$ such that these
coefficients belong to $L^{p+\varepsilon}(G)$ for every
$\xi,\eta\in\mathcal H_0$ and every $\varepsilon>0$; see
\cite[Definition~3.1]{GorodnikNevo}.

For a unitary representation $\pi$ of a locally compact group $H$
and $a\in L^1(H)$, let $\pi(a)=\int_H a(g)\pi(g)\,dg$.
A unitary representation $\sigma$ of $H$ is \emph{weakly contained}
in $\pi$, written $\sigma\prec\pi$, if
\[
 \|\sigma(a)\|_{\op}\le\|\pi(a)\|_{\op}
 \qquad\text{for every }a\in L^1(H).
\]
A representation of $H$ is \emph{tempered} if it is weakly contained
in the left regular representation of $H$ on $L^2(H)$.

Recall $G=\prod_{v\mid\infty}G_v$ from Section~\ref{subsec:group-measures}.
Choose Haar measures on the factors whose product is $dg$.
All Hilbert spaces below are complex and separable, all unitary
representations are strongly continuous, and all tensor products of
Hilbert spaces are completed Hilbert tensor products.

A vector in a unitary representation is $\Kc_{E,v}$-finite
(respectively, $\Kc_{E,\infty}$-finite) if its orbit under that
compact group spans a finite dimensional subspace. The $\Kc_{E,\infty}$-finite vectors
are dense by the decomposition theorem for representations of compact groups
\cite[Theorem~5.2]{FollandHarmonic}.

Let $H$ be a locally compact group and let $\Kc_H\subset H$ be a
compact subgroup. An irreducible representation $\tau$ of $\Kc_H$
is called a \emph{$\Kc_H$-type}. In a unitary representation of $H$,
its \emph{isotypic subspace} is the closed sum of all
$\Kc_H$-invariant subspaces on which the representation is equivalent
to $\tau$. An irreducible unitary representation $\pi$ of $H$
is \emph{$\Kc_H$-spherical} if there is a nonzero vector $\xi$ such that
$\pi(k)\xi=\xi$ for every $k\in\Kc_H$.
In this paper, spherical representations of $G_v$ are understood
relative to $\Kc_{E,v}$.

For each $v\mid\infty$, choose an orthonormal basis of $E_v$ and let
$P_v\subset G_v$ be the subgroup of upper triangular matrices in this
basis. It is a minimal parabolic subgroup, and
$G_v=P_v\Kc_{E,v}$. The modulus character of $P_v$ is given,
for $p\in P_v$ with diagonal entries $a_1,\ldots,a_n$, by
\[
 \delta_{P_v}(p)=\prod_{i<j}|a_i/a_j|_v^{n_v}.
\]
Let $\mathcal H_v$ be the Hilbert completion of the space of smooth
functions $f:G_v\to\C$ satisfying
\[
 f(pg)=\delta_{P_v}(p)^{1/2}f(g)
 \qquad(p\in P_v,\ g\in G_v),
\]
with respect to the inner product
\[
 \langle f_1,f_2\rangle_{\mathcal H_v}
 =\int_{\Kc_{E,v}}f_1(k)\overline{f_2(k)}\,dk,
\]
where $dk$ is probability Haar measure. Restriction to $\Kc_{E,v}$
identifies $\mathcal H_v$ with
$L^2((P_v\cap\Kc_{E,v})\backslash\Kc_{E,v})$, equipped with the
quotient probability measure. The normalized induction of the trivial
representation $\mathbf1$ of $P_v$ on $\C$ is the unitary representation
$\operatorname{Ind}_{P_v}^{G_v}\mathbf1$ on $\mathcal H_v$, defined on
smooth functions by
\[
 \bigl(\bigl(\operatorname{Ind}_{P_v}^{G_v}\mathbf1\bigr)(g_0)f\bigr)(g)
 =f(gg_0)\qquad(g,g_0\in G_v).
\]

Let $u_v\in\mathcal H_v$ be the function determined by
$u_v(pk)=\delta_{P_v}(p)^{1/2}$ for $p\in P_v$ and
$k\in\Kc_{E,v}$. Since $u_v(k)=1$ on $\Kc_{E,v}$, it is a
$\Kc_{E,v}$-fixed unit vector. The Harish--Chandra spherical function
associated with $\Kc_{E,v}$ is
\[
 \Xi_{G_v}(g)=
 \left\langle\bigl(\operatorname{Ind}_{P_v}^{G_v}\mathbf1\bigr)(g)u_v,u_v\right\rangle_{\mathcal H_v}
 \qquad(g\in G_v).
\]
It is positive and $\Kc_{E,v}$-bi-invariant, with $\Xi_{G_v}(1)=1$
\cite[p.~99]{CowlingHaagerupHowe}.

When $n=2$ and $0<s<1$, the normalized smooth induction
$\operatorname{Ind}_{P_v}^{G_v}(\delta_{P_v}^{s/2})$ consists of smooth
functions $f:G_v\to\C$ satisfying
\[
 f(pg)=\delta_{P_v}(p)^{(1+s)/2}f(g)
 \qquad(p\in P_v,\ g\in G_v),
\]
with $G_v$ acting by right translation. This representation admits a
positive definite $G_v$-invariant inner product, different from the
$L^2$ inner product on the compact quotient above. Its Hilbert completion
is irreducible and unitary, is not tempered, and has a nonzero
$\Kc_{E,v}$-fixed vector. These completions, for $0<s<1$, form the
\emph{spherical complementary series}.
By the classification in
\cite[Chapter~II, \S\S4--5, pp.~33--38]{KnappRepresentation},
every irreducible unitary representation of $\SL_2(\R)$ or $\SL_2(\C)$
is trivial, tempered, or belongs to the spherical complementary series.

To transfer bounds for irreducible matrix coefficients to
$L^2_0(G/\Gamma)$, we use the following direct integral construction
\cite[\S7.4, pp.~238--241]{FollandHarmonic}.

\begin{definition}
Let $(\Omega,\nu)$ be a $\sigma$-finite measure space.
A measurable field of Hilbert spaces $(\mathcal H_\omega)_{\omega\in\Omega}$
is a family equipped with sections $s_r(\omega)\in\mathcal H_\omega$,
$r\ge1$, such that
\begin{enumerate}[label=\textup{(\roman*)}]
\item $\omega\mapsto\langle s_r(\omega),s_s(\omega)\rangle$ is
measurable for every $r,s$;
\item the linear span of $\{s_r(\omega):r\ge1\}$ is dense in
$\mathcal H_\omega$ for every $\omega$.
\end{enumerate}
A section $\xi=(\xi_\omega)$, with $\xi_\omega\in\mathcal H_\omega$,
is measurable if $\omega\mapsto\langle\xi_\omega,s_r(\omega)\rangle$
is measurable for every $r$. The direct integral is the Hilbert space
\[
 \int_\Omega^\oplus\mathcal H_\omega\,d\nu(\omega)
 =\left\{\xi\text{ measurable}:\;
       \int_\Omega\|\xi_\omega\|^2\,d\nu(\omega)<\infty\right\}\big/\!\sim,
\]
where $\xi\sim\eta$ means equality almost everywhere. Its vector
operations are defined in each fiber, and its inner product is
\[
 \langle\xi,\eta\rangle
 =\int_\Omega\langle\xi_\omega,\eta_\omega\rangle_{\mathcal H_\omega}
       \,d\nu(\omega).
\]
\end{definition}

Let $H$ be a second countable locally compact group. A family of unitary representations
$\sigma_\omega$ of $H$ on $\mathcal H_\omega$ is measurable if
$\omega\mapsto\langle\sigma_\omega(g)s_r(\omega),s_s(\omega)\rangle$
is measurable for every $g\in H$ and every $r,s$.
Its direct integral acts on the preceding Hilbert space by
\[
 \left(\left(\int_\Omega^\oplus\sigma_\omega\,d\nu(\omega)\right)(g)
       \xi\right)_\omega=\sigma_\omega(g)\xi_\omega.
\]

\begin{proposition}[{\cite[Theorems~7.29 and~7.30(b)]{FollandHarmonic}}]
\label{prop:unitary-direct-integral}
Let $H$ be a second countable locally compact group, and let $\pi$
be a unitary representation of $H$ on $\mathcal H$.
Then there exist a standard Borel space $\Omega$ with a
$\sigma$-finite measure $\nu$, a measurable family of irreducible
unitary representations $(\sigma_\omega,\mathcal H_\omega)$ of $H$,
and a unitary isomorphism
\[
 U:\mathcal H\longrightarrow
       \int_\Omega^\oplus\mathcal H_\omega\,d\nu(\omega)
\]
such that, for every $g\in H$ and $\xi\in\mathcal H$,
\[
 (U\pi(g)\xi)_\omega=\sigma_\omega(g)(U\xi)_\omega
 \quad\text{for almost every }\omega.
\]
\end{proposition}

For the direct integral in Proposition~\ref{prop:unitary-direct-integral},
integrated operators act in each fiber and satisfy
\[
 \|\pi(a)\|_{\op}
 =\operatorname*{ess\,sup}_{\omega\in\Omega}
       \|\sigma_\omega(a)\|_{\op}
 \qquad(a\in L^1(H)).
\]
Applying this identity to a countable dense subset of $L^1(H)$ and
then using continuity in the $L^1$ norm gives
$\sigma_\omega\prec\pi$ for almost every $\omega$.
The same argument applies after restriction to a closed subgroup of $H$.

Each $G_v$ is a connected semisimple Lie group and hence is type~I
\cite[Theorems~7.7 and~7.8(a)]{FollandHarmonic}.

\begin{proposition}[{\cite[Theorem~7.17]{FollandHarmonic}}]
\label{prop:irreducible-product}
Every irreducible unitary representation $\sigma$ of
$G=\prod_{v\mid\infty}G_v$ on $\mathcal H$ is unitarily
equivalent to a tensor product of irreducible unitary representations
$(\sigma_v,\mathcal H_v)$ of the factors $G_v$. Under this identification,
\[
 \mathcal H\simeq\bigotimes_{v\mid\infty}\mathcal H_v,
 \qquad
 \sigma(g)=\bigotimes_{v\mid\infty}\sigma_v(g_v).
\]
\end{proposition}

With $U$ as in Proposition~\ref{prop:unitary-direct-integral},
write $\xi_\omega=(U\xi)_\omega$ and $\eta_\omega=(U\eta)_\omega$
for $\xi,\eta\in\mathcal H$.
Unitarity and the intertwining identity give
\begin{equation}\label{eq:direct-integral-identities}
 \|\xi\|^2
 =\int_\Omega\|\xi_\omega\|^2\,d\nu(\omega),\qquad \langle\pi(g)\xi,\eta\rangle
 =\int_\Omega
       \langle\sigma_\omega(g)\xi_\omega,\eta_\omega\rangle
       \,d\nu(\omega).
\end{equation}

\subsection{Matrix coefficient estimates}\label{subsec:matrix-coefficients}

\begin{lemma}\label{lem:product-coefficient-bounds}
Let $(\sigma_{\omega,v},\mathcal H_{\omega,v})$ be unitary
representations of $G_v$ whose tensor products form a measurable
family over $(\Omega,\nu)$. Let
\[
 \pi=\int_\Omega^\oplus
       \left(\bigotimes_{v\mid\infty}\sigma_{\omega,v}\right)
       \,d\nu(\omega),\qquad
 \mathcal H=\int_\Omega^\oplus
       \left(\bigotimes_{v\mid\infty}\mathcal H_{\omega,v}\right)
       \,d\nu(\omega)
\]
be the resulting representation and its Hilbert space, and fix
$1\le q<\infty$.
Suppose that there are constants $0\le M_v<\infty$ such that, for
almost every $\omega$ and each $v$, there is a nonnegative
measurable function $\Psi_{\omega,v}$ on $G_v$ satisfying
$\|\Psi_{\omega,v}\|_{L^q(G_v)}\le M_v$ and
\[
 |\langle\sigma_{\omega,v}(g)\xi,\eta\rangle|
 \le\sqrt{d_\xi d_\eta}\,\|\xi\|\,\|\eta\|\Psi_{\omega,v}(g)
\]
for $\Kc_{E,v}$-finite vectors, where
$d_\xi=\dim\operatorname{span}(\sigma_{\omega,v}(\Kc_{E,v})\xi)$
and similarly for $d_\eta$.
Then, for each finite set $\mathcal T$ of irreducible
$\Kc_{E,\infty}$-types, there is $C_{\mathcal T}>0$, depending only
on $\mathcal T$, such that all vectors $\xi,\eta$ with types in
$\mathcal T$ satisfy
\begin{equation}\label{eq:product-coefficient-bound}
 \|\langle\pi(\cdot)\xi,\eta\rangle\|_{L^q(G)}
 \le C_{\mathcal T}\left(\prod_{v\mid\infty}M_v\right)
                         \|\xi\|\,\|\eta\|.
\end{equation}
\end{lemma}

\begin{proof}
Write $\xi=(\xi_\omega)$ and $\eta=(\eta_\omega)$ in the direct
integral. The isotypic projection for an irreducible
$\Kc_{E,\infty}$-type $\tau$ is obtained by integrating the
representation against $(\dim\tau)\overline{\chi_\tau}$, where
$\chi_\tau$ is its character and Haar measure on the compact
group has mass one. The isotypic projection acts fiberwise, so
$\xi_\omega$ and $\eta_\omega$ have types in $\mathcal T$
for almost every $\omega$.
Fix such an $\omega$.
For an irreducible $\Kc_{E,v}$-type $\tau_v$, let $P_{\tau_v}$
denote projection onto its isotypic subspace in $\sigma_{\omega,v}$.
A vector in this subspace has orbit dimension at most
$(\dim\tau_v)^2$. Indeed, its expansion in a basis of the
representation space of $\tau_v$ involves at most $\dim\tau_v$
multiplicity vectors. Hence the local hypothesis gives
\[
 \|P_{\tau_v}\sigma_{\omega,v}(g_v)P_{\upsilon_v}\|_{\op}
 \le(\dim\tau_v)(\dim\upsilon_v)\Psi_{\omega,v}(g_v).
\]
For product types $\tau=\bigotimes_v\tau_v$ and
$\upsilon=\bigotimes_v\upsilon_v$, let
$P_\tau=\bigotimes_vP_{\tau_v}$ and
$\sigma_\omega=\bigotimes_v\sigma_{\omega,v}$. Multiplicativity
of operator norms gives
\[
 \|P_\tau\sigma_\omega(g)P_\upsilon\|_{\op}
 \le(\dim\tau)(\dim\upsilon)\prod_v\Psi_{\omega,v}(g_v).
\]
Sum over $\tau,\upsilon\in\mathcal T$ and use Fubini's theorem:
\[
 \|\langle\sigma_\omega(\cdot)\xi_\omega,\eta_\omega\rangle\|_{L^q(G)}
 \le C_{\mathcal T}\left(\prod_vM_v\right)
                   \|\xi_\omega\|\,\|\eta_\omega\|.
\]
Using \eqref{eq:direct-integral-identities}, Minkowski's inequality
and Cauchy--Schwarz give
\[
 \begin{aligned}
 \|\langle\pi(\cdot)\xi,\eta\rangle\|_{L^q(G)}
 &\le\int_\Omega
 \|\langle\sigma_\omega(\cdot)\xi_\omega,\eta_\omega\rangle\|_{L^q(G)}
 \,d\nu(\omega)\\
 &\le C_{\mathcal T}\left(\prod_vM_v\right)
       \int_\Omega\|\xi_\omega\|\,\|\eta_\omega\|\,d\nu(\omega)\\
 &\le C_{\mathcal T}\left(\prod_vM_v\right)\|\xi\|\,\|\eta\|.
 \end{aligned}
\]
\end{proof}

\begin{proposition}\label{prop:uniform-matrix-coefficients}
For every congruence subgroup $\Gamma'$ of one of the lattices $\Gamma_j$
from Proposition~\ref{prop:arithmetic-sectors}, the representation
$\pi_{\Gamma'}^0$ is $L^{p+}$, where
\[
 p=\begin{cases}
 2(n-1),&n\ge3,\\
 3,&n=2.
 \end{cases}
\]
\end{proposition}

\begin{proof}
By Propositions~\ref{prop:unitary-direct-integral}
and~\ref{prop:irreducible-product}, write
\[
 \pi_{\Gamma'}^0\simeq
 \int_\Omega^\oplus\left(\bigotimes_{v\mid\infty}\sigma_{\omega,v}\right)
       \,d\nu(\omega),
\]
where each $\sigma_{\omega,v}$ is an irreducible unitary representation
of $G_v$.

Assume first that $n\ge3$. The lattice $\Gamma'$ is irreducible
\cite[Proposition~5.5.8 and Remark~5.5.11]{MorrisArithmetic}.
Since each $G_v$ is noncompact, Moore's ergodicity theorem
\cite[Theorem~14.2.4]{MorrisArithmetic} gives
$L_0^2(G/\Gamma')^{G_v}=\{0\}$.
Since invariance under $G_v$ can be tested on a countable dense subset,
the $G_v$-fixed subspace is the direct integral of the fixed subspaces
in the fibers. In each fiber, $G_v$ acts through the irreducible factor
$\sigma_{\omega,v}$, so a nonzero fixed subspace occurs precisely when
$\sigma_{\omega,v}$ is trivial. Thus $\sigma_{\omega,v}$ is nontrivial
for almost every $\omega$.
For every nontrivial irreducible unitary representation $\tau_v$ of $G_v$,
Oh's theorem \cite[\S7.5, Theorem~(1), and \S7.6]{OhMatrix} gives
\[
 |\langle\tau_v(g)\xi,\eta\rangle|
 \le\sqrt{d_\xi d_\eta}\,\|\xi\|\,\|\eta\|
       \Xi_{G_v}(g)^{1/(n-1)}
\]
for $\Kc_{E,v}$-finite vectors, where
$d_\xi=\dim\operatorname{span}(\tau_v(\Kc_{E,v})\xi)$
and similarly for $d_\eta$.
The function $\Xi_{G_v}$ belongs to $L^{2+\varepsilon}(G_v)$
for every $\varepsilon>0$ \cite[\S2.2, Proposition~(3)]{OhMatrix}.
Thus $\Xi_{G_v}^{1/(n-1)}\in L^q(G_v)$ for $q>2(n-1)$,
with a bound independent of $\tau_v$ and the congruence level.

For $n=2$, choose a compact open subgroup
$K_f'\subset\SL_2(\A_{K,f})$ such that
$\Gamma'=\SL_2(K)\cap K_f'$.
Strong approximation \cite[Theorem~7.12]{PlatonovRapinchuk}, followed by
inversion, gives an identification of unitary $G$-representations
\[
 L_0^2(G/\Gamma')\simeq
 L_0^2\bigl(\SL_2(K)\backslash\SL_2(\A_K)\bigr)^{K_f'}.
\]
Let $\Pi$ be the right regular representation on
$L_0^2(\SL_2(K)\backslash\SL_2(\A_K))$ and let $\Pi_v=\Pi|_{G_v}$.
In the direct integral, restriction to $G_v$ acts as
$\sigma_{\omega,v}$ tensored with the identity on the other factors.
The weak containment property from Section~\ref{subsec:lattice-counting}
and the automorphic realization above therefore give
\[
 \sigma_{\omega,v}\prec\left.\pi_{\Gamma'}^0\right|_{G_v}\prec\Pi_v
 \qquad\text{for almost every }\omega.
\]
For a spherical irreducible representation $\tau_v$, let
$\varphi_{\tau_v}(g)=\langle\tau_v(g)e_v,e_v\rangle$, where $e_v$
is a $\Kc_{E,v}$-fixed unit vector.
The bound in \cite[\S1.2, following (1.10)]{GhoshGorodnikNevo}
and Corollary~3.6 there, with $S=\{v\}$, give
\[
 \sup_{\substack{\tau_v\ \mathrm{irreducible\ spherical}\\
                  \tau_v\prec\Pi_v}}
 \|\varphi_{\tau_v}\|_{L^q(G_v)}<\infty
 \qquad(q>18/7).
\]
The space $L^2_{00}$ used in that reference equals $L_0^2$:
all automorphic characters of $\SL_2$
are trivial \cite[Lemma~4.1]{GhoshGorodnikNevo}.
This bound excludes the trivial representation from the representations
weakly contained in $\Pi_v$.

The classification recalled in Section~\ref{subsec:lattice-counting}
therefore shows that every irreducible $\tau_v\prec\Pi_v$ is tempered
or belongs to the spherical complementary series. Let
$\psi_{\tau_v}=\Xi_{G_v}$ in the tempered case and
$\psi_{\tau_v}=\varphi_{\tau_v}$ in the complementary series case.
For tempered representations, \cite[Theorem~2]{CowlingHaagerupHowe}
gives
\[
 |\langle\tau_v(g)\xi,\eta\rangle|
 \le\sqrt{d_\xi d_\eta}\,\|\xi\|\,\|\eta\|\,\psi_{\tau_v}(g)
\]
for $\Kc_{E,v}$-finite vectors, with $d_\xi,d_\eta$ as above.
For the complementary series, the spherical coefficient
$\varphi_{\tau_v}$ is positive and $\Kc_{E,v}$-bi-invariant. Since
$e_v$ is fixed by $\Kc_{E,v}$ and Haar measure on this group has mass one,
\[
 \left(\int_{\Kc_{E,v}}\int_{\Kc_{E,v}}
   |\langle\tau_v(k_1gk_2)e_v,e_v\rangle|^2\,dk_1\,dk_2\right)^{1/2}
 =\varphi_{\tau_v}(g).
\]
Thus the nonzero pair $e_v,e_v$ satisfies the hypothesis in the final
assertion of \cite[Theorem~B]{CowlingDecay}. Since $\tau_v$ is
irreducible, that assertion gives the preceding coefficient bound for
all $\Kc_{E,v}$-finite vectors $\xi,\eta$.
Since $\Xi_{G_v}\in L^q(G_v)$ for $q>2$, the spherical estimate yields
\[
 \sup_{\substack{\tau_v\ \mathrm{irreducible}\\
                  \tau_v\prec\Pi_v}}
 \|\psi_{\tau_v}\|_{L^q(G_v)}<\infty
 \qquad(q>18/7),
\]
independently of the congruence level.

The local bounds for $\sigma_{\omega,v}$ satisfy the hypotheses of
Lemma~\ref{lem:product-coefficient-bounds}, with
$\Psi_{\omega,v}=\Xi_{G_v}^{1/(n-1)}$ for $n\ge3$ and
$\Psi_{\omega,v}=\psi_{\sigma_{\omega,v}}$ for $n=2$.
The lemma gives $L^q(G)$ coefficients of $\Kc_{E,\infty}$-finite
vectors for $q>2(n-1)$ and $q>18/7$, respectively. These vectors are dense, and $18/7<3$; hence the asserted $L^{p+}$ bounds follow.
\end{proof}

\subsection{Spectral estimates for averaging operators}\label{subsec:spectral-averages}

Recall that $q_n=2\lceil(n-1)/2\rceil$ in
\eqref{eq:basic-parameters}.

\begin{proposition}\label{prop:spectral}
Let $\Gamma'$ be a congruence subgroup of one of the lattices $\Gamma_j$.
For every $1<u<2$ and every nonnegative $f\in L^1(G)\cap L^u(G)$,
\begin{equation}\label{eq:spectral-operator}
 \|\pi_{\Gamma'}^0(f)\|_{2\to2}
 \ll_{G,u}\|f\|_1^{1-1/q_n}\|f\|_u^{1/q_n}.
\end{equation}
The estimate is uniform over the congruence subgroups $\Gamma'$.
\end{proposition}

\begin{proof}
Proposition~\ref{prop:uniform-matrix-coefficients} gives the $L^{p+}$
bound with $p=2(n-1)$ for $n\ge3$ and $p=3$ for $n=2$.
Following \cite[(3.1)]{GorodnikNevo}, let
\[
 n_e(p)=
 \begin{cases}
  1, & p=2,\\
  2\lceil p/4\rceil, & p>2.
 \end{cases}
\]
For these values of $p$,
\[
 n_e(2(n-1))=2\left\lceil\frac{n-1}{2}\right\rceil=q_n
 \quad(n\ge3),\qquad n_e(3)=2=q_2.
\]
Normalize the invariant measure on $G/\Gamma'$ to have mass one;
this leaves the operator norms unchanged. Since the action is ergodic
and each algebraic factor $\SL_n$ is simply connected and isotropic,
\cite[Theorem~4.5]{GorodnikNevo} gives, for every nonnegative
probability density $h\in L^u(G)$, $1<u<2$,
\[
 \|\pi_{\Gamma'}^0(h)\|_{2\to2}
 \le C_{G,u}\|h\|_{L^u(G)}^{1/q_n}.
\]
For $\|f\|_1>0$, apply this bound to $h=f/\|f\|_1$. Then
\[
 \|\pi_{\Gamma'}^0(f)\|_{2\to2}
 =\|f\|_1\|\pi_{\Gamma'}^0(h)\|_{2\to2}
 \ll_{G,u}\|f\|_1^{1-1/q_n}\|f\|_u^{1/q_n},
\]
which proves \eqref{eq:spectral-operator}.
\end{proof}

\subsection{A lattice counting inequality}\label{subsec:smoothing}

The following estimate holds for any unimodular real Lie group $G$
with a lattice $\Gamma$. Fix a Haar measure and symmetric exponential
neighborhoods $O_\epsilon$ defined by a norm on $\operatorname{Lie}(G)$.
Choose $\epsilon_\Gamma>0$ so small that exponential coordinates are
valid and
\[
 O_{\epsilon_\Gamma}^{-1}O_{\epsilon_\Gamma}\cap\Gamma=\{1\}.
\]
Writing $\rho=\dim_{\R}G$, exponential coordinates give
\begin{equation}\label{eq:small-ball-volume}
 c\epsilon^\rho\le\vol(O_\epsilon)\le C\epsilon^\rho
 \qquad(0<\epsilon\le\epsilon_\Gamma).
\end{equation}
For $G=\SL_n(K_\infty)$ and the neighborhoods in
\eqref{eq:small-identity-neighborhood}, we have
$\rho=\kappa(n^2-1)$.

For a nonnegative Borel function $D$ on $G$ and
$0<\epsilon_L,\epsilon_R\le\epsilon_\Gamma$, define
\begin{equation}\label{eq:weighted-thickening}
 D^+(g)=\sup_{\substack{u\in O_{\epsilon_L}\\v\in O_{\epsilon_R}}}D(ugv),
 \qquad
 D^-(g)=\inf_{\substack{u\in O_{\epsilon_L}\\v\in O_{\epsilon_R}}}D(ugv).
\end{equation}
For $f\in L^1(G)$, write
$\pi_\Gamma^0(f)=\int_G f(g)\pi_\Gamma^0(g)\,dg$.

\begin{proposition}\label{prop:smoothing}
Let $D:G\to[0,\infty)$ be Borel. If its envelopes $D^\pm$ are
measurable with respect to Haar measure and $D^+\in L^1(G)$, then
\begin{equation}\label{eq:raw-smoothing}
 \begin{aligned}
 &\left|\sum_{\gamma\in\Gamma}D(\gamma)
       -\frac1{\vol(G/\Gamma)}\int_GD\,dg\right|\\
 &\quad\le\frac1{\vol(G/\Gamma)}\int_G(D^+-D^-)\,dg
 +\frac{\max\{\|\pi_\Gamma^0(D^+)\|_{2\to2},
                    \|\pi_\Gamma^0(D^-)\|_{2\to2}\}}
        {\sqrt{\vol(O_{\epsilon_L})\vol(O_{\epsilon_R})}}.
 \end{aligned}
\end{equation}
\end{proposition}

\begin{proof}
For $0<\epsilon\le\epsilon_\Gamma$, let
\[
 \psi_\epsilon=\frac{\mathbf1_{O_\epsilon}}{\vol(O_\epsilon)},
 \qquad
 \phi_\epsilon(g\Gamma)=\sum_{\gamma\in\Gamma}
                              \psi_\epsilon(g\gamma).
\]
With the quotient measure $d\bar g$ induced by $dg$, unfolding
and disjointness of the translates give
\begin{equation}\label{eq:periodized-l2}
 \int_{G/\Gamma}\phi_\epsilon\,d\bar g=1,\qquad
 \|\phi_\epsilon\|_2^2=\|\psi_\epsilon\|_2^2
 =\vol(O_\epsilon)^{-1}.
\end{equation}
For nonnegative $f\in L^1(G)$, let
\[
 \mathcal C(f)=\int_G f(g)
       \langle\pi_\Gamma(g)\phi_{\epsilon_R},\phi_{\epsilon_L}\rangle\,dg.
\]
Tonelli's theorem and unfolding give
\[
 \mathcal C(f)=\sum_{\gamma\in\Gamma}\int_G\int_G
 \psi_{\epsilon_L}(u)\psi_{\epsilon_R}(v)f(u\gamma v^{-1})\,du\,dv.
\]
For $u\in O_{\epsilon_L}$ and $v\in O_{\epsilon_R}$, symmetry of
the neighborhoods implies
$D^-(u\gamma v^{-1})\le D(\gamma)\le D^+(u\gamma v^{-1})$.
Averaging these inequalities against the two kernels gives
\begin{equation}\label{eq:smoothing-bracket}
 \mathcal C(D^-)\le\sum_{\gamma\in\Gamma}D(\gamma)
                  \le\mathcal C(D^+).
\end{equation}
Decompose
\[
 \phi_\epsilon=\vol(G/\Gamma)^{-1}+\phi_\epsilon^0,
 \qquad \phi_\epsilon^0\in L^2_0(G/\Gamma).
\]
Orthogonality and \eqref{eq:periodized-l2} give
\[
 \|\phi_\epsilon^0\|_2^2
 =\vol(O_\epsilon)^{-1}-\vol(G/\Gamma)^{-1}
 \le\vol(O_\epsilon)^{-1}
\]
and
\[
 \mathcal C(f)=\frac1{\vol(G/\Gamma)}\int_G f\,dg
 +\left\langle\pi_\Gamma^0(f)\phi_{\epsilon_R}^0,
       \phi_{\epsilon_L}^0\right\rangle.
\]
Consequently,
\[
 \left|\mathcal C(f)-\frac1{\vol(G/\Gamma)}\int_Gf\,dg\right|
 \le\|\pi_\Gamma^0(f)\|_{2\to2}
       \|\phi_{\epsilon_R}^0\|_2\|\phi_{\epsilon_L}^0\|_2
 \le\frac{\|\pi_\Gamma^0(f)\|_{2\to2}}
       {\sqrt{\vol(O_{\epsilon_L})\vol(O_{\epsilon_R})}}.
\]
Apply this estimate to $f=D^\pm$. Since $D^-\le D\le D^+$,
\eqref{eq:smoothing-bracket} proves \eqref{eq:raw-smoothing}.
\end{proof}

\section{Effective counting with shape cutoffs}\label{sec:compact-truncations}

Throughout this section, we assume that $H(E)=1$.

\subsection{Shape cutoffs and periodization}\label{subsec:growing-compact}

For $R\ge2$, choose Lipschitz functions $\chi_R:\cY_{E,d}\to[0,1]$
satisfying
\begin{equation}\label{eq:shape-cutoff-properties}
 \chi_R|_{\mathcal Y_R}=1,\qquad
 \supp\chi_R\subset\mathcal Y_{2R},\qquad
 \sup_{R\ge2}\Lip(\chi_R)\ll_{K,d,e}1.
\end{equation}

To construct $\chi_R$ satisfying \eqref{eq:shape-cutoff-properties},
choose a fixed smooth function $\vartheta:\R\to[0,1]$ equal to $1$ on
$(-\infty,0]$ and $0$ on $[\log2,\infty)$, and set
\[
 \chi_R(L_1,L_2)=
 \vartheta(\log\lambda_d(L_1)-\log R)
 \vartheta(\log\lambda_e(L_2)-\log R).
\]
The range and support conditions follow from those of $\vartheta$.
By Proposition~\ref{prop:norm-shape-stability}(i) and the bounded
derivative of $\vartheta$, the Lipschitz bound is uniform in $R$
when $\operatorname{dist}(y,y')<\epsilon_0$.
For $\operatorname{dist}(y,y')\ge\epsilon_0$, use
$|\chi_R(y)-\chi_R(y')|\le1\le\epsilon_0^{-1}
\operatorname{dist}(y,y')$.
For $x=(w,y)\in\mathcal X_E$, write $\chi_R(x)=\chi_R(y)$.

Recall that $\beta=\max\{d-1,e-1\}$ and that
\[
 V_h=\exp\{X\in\operatorname{Lie}(P_{d,\infty}^1):\|X\|<h\},
\]
where the norm is the fixed $\operatorname{Ad}(H_{d,E})$-invariant
norm from Section~\ref{subsec:sector-charts}.

\begin{proposition}\label{prop:periodized-weights}
For each $j$ there is a left $H_{d,E}$-invariant function
$\omega_j\in C^\infty(P_{d,\infty}^1)\cap L^1(P_{d,\infty}^1)$,
with $0\le\omega_j\le1$, satisfying the following properties.
\begin{enumerate}[label=\textup{(\roman*)}]
\item For every $p\in P_{d,\infty}^1$,
\begin{equation}\label{eq:quotient-periodization}
 \sum_{\gamma\in\Gamma_{P,j}}\omega_j(p\gamma)=1.
\end{equation}
\item For $R\ge2$, the function $\omega_j\cdot(\chi_R\circ\psi_j)$
has compact support, on which
\[
 \|p\|_{\op}+\|p^{-1}\|_{\op}\ll_E R^\beta.
\]
\item There are constants $C_E,h_0(E)>0$ such that, for every
$0<h\le h_0(E)$, the envelopes
\begin{equation}\label{eq:periodized-weight-brackets}
 \omega_{j,h}^+(p)=\sup_{b\in V_h}\omega_j(bp),\qquad
 \omega_{j,h}^-(p)=\inf_{b\in V_h}\omega_j(bp)
\end{equation}
are continuous and left $H_{d,E}$-invariant, and satisfy
\begin{equation}\label{eq:periodized-weight-oscillation}
 \sum_{\gamma\in\Gamma_{P,j}}
 (\omega_{j,h}^+-\omega_{j,h}^-)(p\gamma)\le C_Eh
 \qquad(p\in P_{d,\infty}^1).
\end{equation}
\end{enumerate}
\end{proposition}

\begin{proof}
The group $P_{d,\infty}^1$ is unimodular, and
$Q_j=P_{d,\infty}^1/\Gamma_{P,j}$ has finite volume by
Proposition~\ref{prop:finite-covolume}.

Let $x_1(p),\ldots,x_N(p)$ list the real and imaginary parts of
all matrix entries of $p$ in a fixed order, and let
\[
 c(p)=\bigl(\|p\|_{\op}+\|p^{-1}\|_{\op},x_1(p),\ldots,x_N(p)\bigr).
\]
Define
\[
 \varphi_j=\bigcap_{\gamma\in\Gamma_{P,j}}
 \{p\in P_{d,\infty}^1:c(p)\le_{\mathrm{lex}}c(p\gamma)\}.
\]
The norm sum has compact sublevel sets and thus attains its minimum
at finitely many points of each closed discrete orbit $p\Gamma_{P,j}$.
Lexicographic ordering of the remaining coordinates selects a unique
representative. The comparison sets are Borel, being finite unions of sets defined by
coordinate equalities and inequalities; countability of
$\Gamma_{P,j}$ makes $\varphi_j$ a Borel fundamental domain.

It follows that
\[
 \sum_{\gamma\in\Gamma_{P,j}}\mathbf1_{\varphi_j}(p\gamma)=1,
 \qquad \vol(\varphi_j)=\vol(Q_j)<\infty.
\]
Minimality and Proposition~\ref{prop:sector-properness}(i) give
\[
 p\in \varphi_j,\quad\psi_j(p)\in\mathcal Y_{3R}
 \quad\Longrightarrow\quad
 \|p\|_{\op}+\|p^{-1}\|_{\op}\ll_E R^\beta.
\]

Choose $h_0>0$ within the range of
Proposition~\ref{prop:norm-shape-stability}(ii), with
$e^{C_Eh_0}\le3/2$ for the constant in that proposition.
Take a symmetric neighborhood $U$ with $\overline U\subset V_{h_0}$.
Then, for every $R\ge2$,
\[
 \psi_j(p)\in\mathcal Y_{2R},\quad u\in\overline U
 \quad\Longrightarrow\quad
 \psi_j(u^{-1}p)\in\mathcal Y_{2e^{C_Eh_0}R}\subset\mathcal Y_{3R}.
\]
Take $0\le\eta_0\in C_c^\infty(U)$ with integral one, and let
\[
 \eta(p)=\int_{H_{d,E}}\eta_0(hp)\,dh,
\]
where $dh$ is probability Haar measure. Then $\eta$ is smooth,
left $H_{d,E}$-invariant, has integral one, and is supported in
the fixed compact set $H_{d,E}\supp\eta_0$. Define
\begin{equation}\label{eq:convolution-periodization}
\omega_j(p)
 =(\eta*\mathbf1_{\varphi_j})(p)
   =\int_{P_{d,\infty}^1}\eta(u)\mathbf1_{\varphi_j}(u^{-1}p)\,dp^1(u)
=\int_{\varphi_j}\eta(pv^{-1})\,dp^1(v).
\end{equation}
The convolution formula gives $0\le\omega_j\le1$ and left
$H_{d,E}$-invariance. The compact support of $\eta$ gives a common
compact integration region locally in $p$, so the second integral
can be differentiated to show that $\omega_j$ is smooth. Tonelli's theorem and the
periodization of $\mathbf1_{\varphi_j}$ give
\[
 \begin{aligned}
 &\sum_{\gamma\in\Gamma_{P,j}}\omega_j(p\gamma)\\
 &\quad=\int_{P_{d,\infty}^1}\eta(u)
       \sum_{\gamma\in\Gamma_{P,j}}
       \mathbf1_{\varphi_j}(u^{-1}p\gamma)\,dp^1(u)
 =\int_{P_{d,\infty}^1}\eta\,dp^1=1,
 \end{aligned}
\]
and
\[
 \int_{P_{d,\infty}^1}\omega_j\,dp^1
 =\left(\int_{P_{d,\infty}^1}\eta\,dp^1\right)\vol(\varphi_j)
 =\vol(Q_j).
\]
This proves (i).

If $\omega_j(p)>0$ and $\psi_j(p)\in\mathcal Y_{2R}$, write
$p=huv$ with $h\in H_{d,E}$, $u\in\supp\eta_0$, and $v\in \varphi_j$.
Since $\psi_j(h^{-1}p)=\psi_j(p)$, the choice of $U$ implies
$\psi_j(v)=\psi_j(u^{-1}h^{-1}p)\in\mathcal Y_{3R}$.
The bound for $v$ and compactness of $H_{d,E}\supp\eta_0$ give
$\|p\|_{\op}+\|p^{-1}\|_{\op}\ll_E R^\beta$.
Passing to the closure proves (ii).

For small $h>0$, set
\[
 \eta_h^+(x)=\sup_{b\in V_h}\eta(bx),\qquad
 \eta_h^-(x)=\inf_{b\in V_h}\eta(bx).
\]
The extrema defining $\eta_h^\pm$ and $\omega_{j,h}^\pm$ can be
taken over the compact set $\overline{V_h}$, so these functions
are continuous. Since $V_h$ is invariant under conjugation by $H_{d,E}$ and $\omega_j$ is left $H_{d,E}$-invariant, the envelopes $\omega_{j,h}^\pm$ inherit this left invariance. The kernels $\eta_h^\pm$ vanish outside
a fixed compact set $K_0$, and bounded first derivatives of $\eta$ give
\[
 0\le\eta_h^+-\eta_h^-\ll_E h\mathbf1_{K_0}.
\]
For $b\in V_h$, apply the last integral formula in
\eqref{eq:convolution-periodization} to $\omega_j(bp)$.
The integrand $\eta(bpv^{-1})$ lies between $\eta_h^-(pv^{-1})$ and
$\eta_h^+(pv^{-1})$. Taking the infimum and supremum over $b$ gives
\[
 \eta_h^-*\mathbf1_{\varphi_j}
 \le\omega_{j,h}^-\le\omega_{j,h}^+
 \le\eta_h^+*\mathbf1_{\varphi_j}.
\]
Subtracting yields
\[
 0\le\omega_{j,h}^+-\omega_{j,h}^-
 \le(\eta_h^+-\eta_h^-)*\mathbf1_{\varphi_j}.
\]
Tonelli's theorem and the periodization of $\mathbf1_{\varphi_j}$ give
\[
 \begin{aligned}
 \sum_{\lambda\in\Gamma_{P,j}}
 (\omega_{j,h}^+-\omega_{j,h}^-)(p\lambda)
 &\le\int_{P_{d,\infty}^1}(\eta_h^+-\eta_h^-)(u)
     \sum_{\lambda\in\Gamma_{P,j}}\mathbf1_{\varphi_j}(u^{-1}p\lambda)\,dp^1(u)\\
 &=\int_{P_{d,\infty}^1}(\eta_h^+-\eta_h^-)\,dp^1\ll_E h.
 \end{aligned}
\]
\end{proof}

By \eqref{eq:parabolic-coordinate-ambiguity} and left
$H_{d,E}$-invariance, the value $\omega_j(p)$ depends only on
$g=\Phi_j(k,t,p)$.

\begin{proposition}\label{prop:quotient-integration}
Let $\omega_j$, $\omega_{j,h}^\pm$, $C_E$, and $h_0(E)$ be as in
Proposition~\ref{prop:periodized-weights}.
\begin{enumerate}[label=\textup{(\roman*)}]
\item For every nonnegative function $a$ on $\Gamma_j/\Gamma_{P,j}$,
define $D:\Gamma_j\to[0,\infty)$ by
\[
 D(g)=\omega_j(p)a(g\Gamma_{P,j}),\qquad g=\Phi_j(k,t,p)\in\Gamma_j.
\]
Then
\begin{equation}\label{eq:periodized-counting-identity}
 \sum_{g\Gamma_{P,j}\in\Gamma_j/\Gamma_{P,j}}a(g\Gamma_{P,j})
 =\sum_{g\in\Gamma_j}D(g).
\end{equation}
\item For every nonnegative Borel function $a$ on $Q_j=P_{d,\infty}^1/\Gamma_{P,j}$,
\begin{equation}\label{eq:periodized-shape-integral}
 \int_{P_{d,\infty}^1}\omega_j(p)a(p\Gamma_{P,j})\,dp^1
 =\int_{Q_j}a(\bar p)\,d\bar p_j.
\end{equation}
\item For every $a$ as in \textup{(ii)} and every $0<h\le h_0(E)$,
\begin{equation}\label{eq:periodized-weighted-brackets}
 \begin{aligned}
 \int_{P_{d,\infty}^1}\omega_{j,h}^+(p)a(p\Gamma_{P,j})\,dp^1
 &\le(1+C_Eh)\int_{Q_j}a(\bar p)\,d\bar p_j,\\
 \int_{P_{d,\infty}^1}(\omega_{j,h}^+-\omega_{j,h}^-)(p)a(p\Gamma_{P,j})\,dp^1
 &\le C_Eh\int_{Q_j}a(\bar p)\,d\bar p_j.
 \end{aligned}
\end{equation}
\end{enumerate}
\end{proposition}

\begin{proof}
For (i), write $g=\Phi_j(k,t,p)$. For $\lambda\in\Gamma_{P,j}$,
one has $g\lambda=\Phi_j(k,t,p\lambda)$. Hence
\eqref{eq:quotient-periodization} gives
\[
 \sum_{\lambda\in\Gamma_{P,j}}D(g\lambda)
 =a(g\Gamma_{P,j})\sum_{\lambda\in\Gamma_{P,j}}\omega_j(p\lambda)
 =a(g\Gamma_{P,j}).
\]
Summing over $\Gamma_j/\Gamma_{P,j}$ proves
\eqref{eq:periodized-counting-identity}.

For (ii) and (iii), let $a$ be a nonnegative Borel function on $Q_j$.
For every nonnegative Borel function $u$ on $P_{d,\infty}^1$, Weil's quotient
formula gives
\[
 \int_{P_{d,\infty}^1}u(p)a(p\Gamma_{P,j})\,dp^1
 =\int_{Q_j}a(\bar p)
       \sum_{\gamma\in\Gamma_{P,j}}u(p\gamma)\,d\bar p_j.
\]
Here $\bar p=p\Gamma_{P,j}$.
Taking $u=\omega_j$ and using \eqref{eq:quotient-periodization}
proves (ii). Since $\omega_{j,h}^-\le\omega_j\le\omega_{j,h}^+$,
\eqref{eq:periodized-weight-oscillation} gives
\[
 \sum_{\gamma\in\Gamma_{P,j}}\omega_{j,h}^+(p\gamma)
 \le 1+\sum_{\gamma\in\Gamma_{P,j}}
           (\omega_{j,h}^+-\omega_{j,h}^-)(p\gamma)
 \le 1+C_Eh.
\]
Taking $u=\omega_{j,h}^+$ and
$u=\omega_{j,h}^+-\omega_{j,h}^-$ proves (iii).
\end{proof}

\subsection{Counting densities}\label{subsec:counting-densities}

Let $\Omega$ be closed in $\mathcal Z_E$. For $B,R\ge2$ and a bounded
Lipschitz function $\varphi:\mathcal Z_E\to\C$, write
\begin{equation}\label{eq:cutoff-arithmetic-sums}
 S_R(B;\Omega,\varphi)=
 \sum_{\substack{W\subset E,\ \dim_KW=d\\
       (H(W)/B,x_E(W))\in\Omega}}
 \varphi(H(W)/B,x_E(W))\chi_R(x_E(W))
\end{equation}
whenever the sum is absolutely convergent; for nonnegative $\varphi$,
the value $+\infty$ is allowed. Recall that
\[
 x_E(W)=\bigl(W_\infty,([\widehat W],[\widehat{E/W}])\bigr),
 \qquad \|\varphi\|_{\BL}=\|\varphi\|_\infty+\Lip(\varphi),
\]
where the Lipschitz constant uses $|r-r'|+d_E(x,x')$.
Tensoring Proposition~\ref{prop:parabolic-mass} with $\sigma_{E,d}$ gives
\begin{equation}\label{eq:joint-sector-mass}
 \vol(\Kc_{E,f})\sum_{j=1}^J e^{-\kappa n\tau_j}
 \bigl(\sigma_{E,d}\otimes(\psi_j)_*d\bar p_j\bigr)
 =\kappa n\,\mathfrak a_{K,n,d}\,\mu_E.
\end{equation}

Recall that $t_0$ in Proposition~\ref{prop:sector-decomposition}
satisfies $t_0<\log H(W)$ for every $d$-dimensional subspace of $E$.

Throughout this subsection, assume that $\varphi\ge0$ and
$\|\varphi\|_{\BL}\le1$. For $g=\Phi_j(k,t,p)$, let
$x=(kW_0,\psi_j(p))$ and define
\begin{equation}\label{eq:weighted-density-explicit}
 D_{j,B,R}(g)=\omega_j(p)\chi_R(x)\varphi(e^t/B,x)
             \mathbf1_{\{t\ge t_0\}}\mathbf1_\Omega(e^t/B,x).
\end{equation}

\begin{proposition}\label{prop:weighted-sectors}
For $B,R\ge2$, the densities \eqref{eq:weighted-density-explicit}
are well defined, upper semicontinuous, and take values in $[0,1]$.
They satisfy the following identities, with nonnegative sums and
integrals allowed to be infinite.
\begin{enumerate}[label=\textup{(\roman*)}]
\item
\begin{equation}\label{eq:weighted-arithmetic-identity}
 S_R(B;\Omega,\varphi)=\sum_{j=1}^J\sum_{\gamma\in\Gamma_j}D_{j,B,R}(\gamma).
\end{equation}
\item
\begin{equation}\label{eq:weighted-haar-total-f}
 \sum_{j=1}^J\frac1{\vol(G/\Gamma_j)}\int_G D_{j,B,R}\,dg
 =\mathfrak a_{K,n,d}B^{\kappa n}
 \int_\Omega\chi_R(x)\varphi(r,x)
       \mathbf1_{\{r\ge e^{t_0}/B\}}\,d\lambda_E(r,x).
\end{equation}
\end{enumerate}
\end{proposition}

\begin{proof}
By \eqref{eq:parabolic-coordinate-ambiguity},
\eqref{eq:stabilizer-shape-invariance}, and left
$H_{d,E}$-invariance of $\omega_j$, the value of
$D_{j,B,R}(g)$ is independent of the choice of $(k,t,p)$
in $g=\Phi_j(k,t,p)$.
Local smooth sections of
$\Kc_{E,\infty}\to\Kc_{E,\infty}/H_{d,E}$ provide
continuous local choices of these coordinates.
Since $\Omega$ is closed, the indicator functions of the closed sets in the definition are upper semicontinuous. Their product with the remaining continuous nonnegative factors is therefore upper semicontinuous. All factors take values in $[0,1]$.

For $g\in\Gamma_j$ and $W=\gamma_jgW_0$, the sector identities
\eqref{eq:sector-invariants} give $t=\log H(W)>t_0$ and $x=x_E(W)$.
Apply Proposition~\ref{prop:quotient-integration}(i) with
\[
 a(g\Gamma_{P,j})
 =\varphi(H(W)/B,x_E(W))\chi_R(x_E(W))
   \mathbf1_\Omega(H(W)/B,x_E(W)).
\]
By \eqref{eq:weighted-density-explicit}, the function $D$
in Proposition~\ref{prop:quotient-integration}(i),
with the above choice of $a$, is $D_{j,B,R}|_{\Gamma_j}$.
Summing over $j$ proves (i).

For (ii), apply \eqref{eq:parabolic-integration}, set $s=t-\tau_j$,
and use Proposition~\ref{prop:quotient-integration}(ii) for fixed
$w=kW_0$ and $t$, with
\[
 a(\bar p)=\chi_R(x)\varphi(e^t/B,x)\mathbf1_\Omega(e^t/B,x),
 \qquad x=(w,\psi_j(\bar p)).
\]
The pushforward of $dk$ is $\sigma_{E,d}$, so
\begin{equation}\label{eq:weighted-haar-f}
 \begin{aligned}
 \int_G D_{j,B,R}(g)\,dg
 &=e^{-\kappa n\tau_j}
 \int_{\Gr_{d,\infty}(E)}\int_{Q_j}\chi_R(x)\\
 &\quad\times\left(\int_{t_0}^{\infty}\varphi(e^t/B,x)
 \mathbf1_\Omega(e^t/B,x)e^{\kappa nt}\,dt\right)
 d\bar p_j\,d\sigma_{E,d}(w).
 \end{aligned}
\end{equation}
The substitution $r=e^t/B$ gives
\begin{equation}\label{eq:radial-haar-substitution}
 e^{\kappa nt}\,dt=B^{\kappa n}r^{\kappa n-1}\,dr,
 \qquad t\ge t_0\iff r\ge e^{t_0}/B.
\end{equation}
Summing over $j$ and using \eqref{eq:joint-sector-mass} and
$\vol(G/\Gamma_j)^{-1}=\vol(\Kc_{E,f})$ proves (ii).
\end{proof}

Since $0\le\chi_R\varphi\le1$ and $\mu_E$ is a probability measure,
removing the lower radial cutoff in \eqref{eq:weighted-haar-total-f}
changes its right-hand side by at most
\[
 \mathfrak a_{K,n,d}B^{\kappa n}
 \int_{\mathcal X_E}\int_0^{e^{t_0}/B}
       \kappa n r^{\kappa n-1}\,dr\,d\mu_E(x)
 =\mathfrak a_{K,n,d}e^{\kappa nt_0}.
\]
Therefore, whenever $\int_\Omega\chi_R\varphi\,d\lambda_E<\infty$,
\begin{equation}\label{eq:weighted-haar-main-term}
 \sum_{j=1}^J\frac1{\vol(G/\Gamma_j)}\int_GD_{j,B,R}\,dg
 =\mathfrak a_{K,n,d}B^{\kappa n}
       \int_\Omega\chi_R\varphi\,d\lambda_E+O_E(1).
\end{equation}

For $h>0$, let
\[
 \Omega_h^+=\{z\in\mathcal Z_E:d_{\log,E}(z,\Omega)\le h\}.
\]

By Proposition~\ref{prop:periodized-weights}(ii), every
$g=\Phi_j(k,t,p)$ in the support of $D_{j,B,R}$ satisfies
\begin{equation}\label{eq:weighted-support-bounds}
 t\ge t_0,\qquad
 \|p\|_{\op}+\|p^{-1}\|_{\op}\ll_E R^\beta.
\end{equation}

The factor $R^{2\beta}$ in the width condition arises from
conjugation of a right perturbation. Indeed, if
$g=\Phi_j(k,t,p)$ lies in the support of $D_{j,B,R}$
and $u_R\in O_{\epsilon_R}$, then
\eqref{eq:weighted-support-bounds} gives
\[
 \|pu_Rp^{-1}-I\|_{\op}
 \le\|p\|_{\op}\|p^{-1}\|_{\op}\|u_R-I\|_{\op}
 \ll_E R^{2\beta}\epsilon_R.
\]

\begin{proposition}\label{prop:density-envelopes}
There are constants $C_E\ge2$ and $h_0(E)>0$ such that the following
holds. Let $B,R\ge2$, $0<h<h_0(E)$, and $\epsilon_L,\epsilon_R>0$
satisfy
\begin{equation}\label{eq:admissible-widths}
 C_E(\epsilon_L+R^{2\beta}\epsilon_R)\le h.
\end{equation}
The envelopes $D_{j,B,R}^\pm$ defined by
\eqref{eq:weighted-thickening} are Borel and satisfy:
\begin{enumerate}[label=\textup{(\roman*)}]
\item
\begin{equation}\label{eq:density-envelope-mass}
 \int_GD_{j,B,R}^+\,dg
 \ll_E B^{\kappa n}\int_{\Omega_h^+}\chi_{C_ER}(x)\,d\lambda_E(r,x).
\end{equation}
\item
\begin{equation}\label{eq:density-envelope-variation}
 \begin{aligned}
 \int_G(D_{j,B,R}^+-D_{j,B,R}^-)\,dg
 \ll_E B^{\kappa n}\biggl[&h\int_{\Omega_h^+}(1+r)\chi_{C_ER}(x)\,d\lambda_E(r,x)\\
 &+\int_{\partial_h\Omega}\chi_R(x)\,d\lambda_E(r,x)\biggr]+h.
 \end{aligned}
\end{equation}
\end{enumerate}
The nonnegative integrals may be infinite.
\end{proposition}

\begin{proof}
By Proposition~\ref{prop:weighted-sectors}, the density is upper
semicontinuous. Its lower envelope is an infimum of upper
semicontinuous functions. Exhausting
$O_{\epsilon_L}\times O_{\epsilon_R}$ by compact sets expresses
its upper envelope as an increasing limit of upper semicontinuous
functions. Both envelopes are Borel.

Suppose $D_{j,B,R}^+(g)>0$. Choose $g_0=v_Lgv_R$, with
$v_L\in O_{\epsilon_L}$, $v_R\in O_{\epsilon_R}$, and
$D_{j,B,R}(g_0)>0$. By symmetry of the neighborhoods, apply
Proposition~\ref{prop:uniform-charts} from $g_0$ to $g$, using
\eqref{eq:weighted-support-bounds} and $t_{\min}=t_0-1$.
Compactness of $H_{d,E}$ and
\eqref{eq:parabolic-coordinate-ambiguity} imply that every
decomposition $g=\Phi_j(k,t,p)$ satisfies
\[
 t\ge t_0-O_E(\epsilon_L+R^{2\beta}\epsilon_R),\qquad
 \|p\|_{\op}+\|p^{-1}\|_{\op}\ll_E R^\beta.
\]
Choose $h_0(E)$ small enough that $t\ge t_0-1$.
Apply Proposition~\ref{prop:uniform-charts} again from $g$.
For $u_L\in O_{\epsilon_L}$ and $u_R\in O_{\epsilon_R}$, write
\[
 u_Lgu_R=\Phi_j(k',t',u_Pp),\qquad
 x=(kW_0,\psi_j(p)),\quad x'=(k'W_0,\psi_j(u_Pp)).
\]
The Lipschitz bound for $k\mapsto kW_0$ on the compact group and
Proposition~\ref{prop:norm-shape-stability}(ii) give
\begin{equation}\label{eq:weighted-coordinate-displacement}
 d_E(x,x')+|t'-t|+\|\log u_P\|
 \ll_E\epsilon_L+R^{2\beta}\epsilon_R.
\end{equation}
Choose $C_E$ sufficiently large and decrease $h_0(E)$ so that
\eqref{eq:admissible-widths} implies
\[
 u_P\in V_h,\qquad
 d_{\log,E}((e^t/B,x),(e^{t'}/B,x'))\le h,
 \qquad t\ge t_0-h,
\]
and the estimates of Proposition~\ref{prop:periodized-weights} apply.
For $u_L=v_L$ and $u_R=v_R$, the joint parameter
$(e^{t'}/B,x')$ of $g_0$ belongs to $\Omega$ because
$D_{j,B,R}(g_0)>0$. The preceding logarithmic distance bound
therefore gives $(e^t/B,x)\in\Omega_h^+$.
The shape of $g_0$ lies in $\mathcal Y_{2R}$, so
Proposition~\ref{prop:norm-shape-stability}(ii) and the defining
properties of $\chi_R$ give $\chi_{C_ER}(x)=1$ after enlarging
$C_E$ once more. Thus
\begin{equation}\label{eq:enlarged-density-support}
 D_{j,B,R}^+(g)>0
 \quad\Longrightarrow\quad
 (e^t/B,x)\in\Omega_h^+,\qquad \chi_{C_ER}(x)=1.
\end{equation}

Write $r=e^t/B$. Since all factors other than $\omega_j$ lie in
$[0,1]$, \eqref{eq:enlarged-density-support} gives
\[
 D_{j,B,R}^+(g)
 \le\omega_{j,h}^+(p)\chi_{C_ER}(x)\mathbf1_{\Omega_h^+}(r,x).
\]
Apply Proposition~\ref{prop:quotient-integration}(iii) for fixed
$k,t$, with
$a(\bar p)=\chi_{C_ER}(x)\mathbf1_{\Omega_h^+}(e^t/B,x)$ and
$x=(kW_0,\psi_j(\bar p))$.
Using \eqref{eq:parabolic-integration} and
\eqref{eq:joint-sector-mass}, with the substitution $r=e^t/B$ from
\eqref{eq:radial-haar-substitution}, gives
\[
 \sum_{j=1}^J\frac1{\vol(G/\Gamma_j)}\int_GD_{j,B,R}^+\,dg
 \ll_E B^{\kappa n}\int_{\Omega_h^+}\chi_{C_ER}(x)\,d\lambda_E(r,x).
\]
Nonnegativity of the sector contributions gives (i).

For (ii), it suffices to consider $D_{j,B,R}^+(g)>0$.
Let $r'=e^{t'}/B$. The coordinate bounds above,
\eqref{eq:shape-cutoff-properties}, and $\|\varphi\|_{\BL}\le1$ give
\begin{equation}\label{eq:radial-test-displacement}
 \begin{aligned}
 |r'-r|&\le r(e^h-1)\ll hr,\\
 |\chi_R(x')\varphi(r',x')-\chi_R(x)\varphi(r,x)|
 &\le |\chi_R(x')-\chi_R(x)|\\
 &\quad+|\varphi(r',x')-\varphi(r,x)|\\
 &\ll_E d_E(x,x')+|r'-r|\ll_E h(1+r).
 \end{aligned}
\end{equation}
Since $u_P\in V_h$,
$|\omega_j(u_Pp)-\omega_j(p)|\le
(\omega_{j,h}^+-\omega_{j,h}^-)(p)$.
The bounds $d_{\log,E}((r,x),(r',x'))\le h$ and $|t'-t|\le h$ give
\[
 \begin{aligned}
 |\mathbf1_\Omega(r',x')-\mathbf1_\Omega(r,x)|
 &\le\mathbf1_{\partial_h\Omega}(r,x),\\
 |\mathbf1_{\{t'\ge t_0\}}-\mathbf1_{\{t\ge t_0\}}|
 &\le\mathbf1_{\{|t-t_0|\le h\}}.
 \end{aligned}
\]
Expanding the product in \eqref{eq:weighted-density-explicit} and
using $0\le\chi_R,\varphi\le1$, we obtain
\[
 \begin{aligned}
 &|D_{j,B,R}(u_Lgu_R)-D_{j,B,R}(g)|\\
 &\quad\le |\omega_j(u_Pp)-\omega_j(p)|\\
 &\qquad+\omega_j(p)
    |\chi_R(x')\varphi(r',x')-\chi_R(x)\varphi(r,x)|\\
 &\qquad+\omega_j(p)\chi_R(x)
    \bigl(|\mathbf1_{\{t'\ge t_0\}}-\mathbf1_{\{t\ge t_0\}}|
          +|\mathbf1_\Omega(r',x')-\mathbf1_\Omega(r,x)|\bigr).
 \end{aligned}
\]
Now
\[
 D_{j,B,R}^+(g)-D_{j,B,R}^-(g)
 \le 2\sup_{\substack{u_L\in O_{\epsilon_L}\\u_R\in O_{\epsilon_R}}}
       |D_{j,B,R}(u_Lgu_R)-D_{j,B,R}(g)|.
\]
On $\{D_{j,B,R}^+>0\}$, \eqref{eq:enlarged-density-support} gives
$\chi_{C_ER}(x)\mathbf1_{\Omega_h^+}(r,x)=1$.
Combining the preceding bounds therefore gives
\begin{equation}\label{eq:weighted-pointwise-variation}
 \begin{aligned}
 &D_{j,B,R}^+(g)-D_{j,B,R}^-(g)\\
 &\quad\ll_E
 \bigl[(\omega_{j,h}^+-\omega_{j,h}^-)(p)+h(1+r)\omega_j(p)\bigr]
       \chi_{C_ER}(x)\mathbf1_{\Omega_h^+}(r,x)\\
 &\qquad+\omega_j(p)\mathbf1_{\{|t-t_0|\le h\}}
       +\omega_j(p)\chi_R(x)\mathbf1_{\partial_h\Omega}(r,x).
 \end{aligned}
\end{equation}

Proposition~\ref{prop:quotient-integration}(ii)--(iii),
\eqref{eq:parabolic-integration}, and \eqref{eq:joint-sector-mass},
with the substitution $r=e^t/B$, bound the integral of the first
term in \eqref{eq:weighted-pointwise-variation}, summed over sectors
with weights $\vol(G/\Gamma_j)^{-1}$, by
\[
 O_E(hB^{\kappa n})
       \int_{\Omega_h^+}(1+r)\chi_{C_ER}(x)\,d\lambda_E(r,x).
\]
The lower radial boundary contributes $O_E(h)$, since
\begin{equation}\label{eq:weighted-radial-boundary-integrals}
 \int_{t_0-h}^{t_0+h}e^{\kappa nt}\,dt\ll_E h
\end{equation}
and $\int_{P_{d,\infty}^1}\omega_j\,dp^1=\vol(Q_j)$.
For the last term, Proposition~\ref{prop:quotient-integration}(ii)
and \eqref{eq:joint-sector-mass} give
\begin{equation}\label{eq:joint-boundary-unfolding}
 \sum_{j=1}^J\frac1{\vol(G/\Gamma_j)}
   \int_G\omega_j(p)\chi_R(x)\mathbf1_{\partial_h\Omega}(e^t/B,x)\,dg
 =\mathfrak a_{K,n,d}B^{\kappa n}
                   \int_{\partial_h\Omega}\chi_R(x)\,d\lambda_E(r,x).
\end{equation}
Here the $j$th integral uses $g=\Phi_j(k,t,p)$ and
$x=(kW_0,\psi_j(p))$.
Nonnegativity again bounds each sector integral. Integrating
\eqref{eq:weighted-pointwise-variation} proves (ii).
\end{proof}

\subsection{Effective counting estimates}\label{subsec:compact-counting-estimates}

Recall that $\mu_E=\sigma_{E,d}\otimes\nu_{E,d}$ is the probability
measure on $\mathcal X_E$. The parameters from
Section~\ref{subsec:numerical-parameters} used below are
\[
 \rho=\kappa(n^2-1),\qquad
q_n=2\left\lceil\frac{n-1}{2}\right\rceil,\qquad \delta_0=\frac{\kappa n}{2q_n(\rho+1)}.
\]
Let $C_E$ be the constant in Proposition~\ref{prop:density-envelopes}.
For a closed set $\Omega\subset\mathcal Z_E$ and a bounded Lipschitz
function $\varphi:\mathcal Z_E\to\C$, write
$S_R(B)=S_R(B;\Omega,\varphi)$ as in \eqref{eq:cutoff-arithmetic-sums}, and let
\begin{equation}\label{eq:cutoff-main-terms}
 M_R(B)=\mathfrak a_{K,n,d}B^{\kappa n}
 \int_\Omega\chi_R(x)\varphi(r,x)\,d\lambda_E(r,x)
\end{equation}
whenever the integral is absolutely convergent.

\Needspace{14\baselineskip}
\begin{proposition}\label{prop:compact-count}
There is $h_1(E)>0$ such that the following holds.
Let $C_0\ge1$, let $\Omega\subset\mathcal Z_E$ be closed, and let
$\varphi:\mathcal Z_E\to\C$ be bounded Lipschitz.
If $B,R\ge2$ and $0<h<h_1(E)$ satisfy
\begin{equation}\label{eq:counting-mass-hypothesis}
 \int_{\Omega_h^+}\chi_{C_ER}(x)\,d\lambda_E(r,x)\le C_0,
\end{equation}
then, for every $\varepsilon>0$,
\begin{equation}\label{eq:cutoff-counting-estimate}
 \begin{aligned}
 |S_R(B)-M_R(B)|
 &\ll_{E,C_0,\varepsilon}\|\varphi\|_{\BL}
 \biggl[B^{\kappa n}\biggl(
 h\int_{\Omega_h^+}(1+r)\chi_{C_ER}(x)\,d\lambda_E(r,x)\\
 &\qquad+\int_{\partial_h\Omega}\chi_R(x)\,d\lambda_E(r,x)\biggr)
 +R^{\rho\beta}h^{-\rho}
 B^{\kappa n-\kappa n/(2q_n)+\varepsilon}+1\biggr].
 \end{aligned}
\end{equation}
\end{proposition}

\begin{proof}
Decompose the real and imaginary parts of $\varphi$ into positive
and negative parts. Each component has bounded Lipschitz norm at most
$\|\varphi\|_{\BL}$. By linearity and scaling, we may therefore
assume $\varphi\ge0$ and $\|\varphi\|_{\BL}\le1$.
Use the densities $D_{j,B,R}$ in \eqref{eq:weighted-density-explicit}.

Let $h_0(E)$ be as in Proposition~\ref{prop:density-envelopes}, and choose
\[
 0<h_1(E)<\min\{1,h_0(E),2C_E\min_j\epsilon_{\Gamma_j}\}.
\]
For $0<h<h_1(E)$, let
\begin{equation}\label{eq:modulus-smoothing-width}
 \epsilon_L=\frac{h}{2C_E},\qquad
 \epsilon_R=\frac{h}{2C_ER^{2\beta}}.
\end{equation}
Then $\epsilon_L,\epsilon_R\le\epsilon_{\Gamma_j}$ and
$C_E(\epsilon_L+R^{2\beta}\epsilon_R)=h$.
Proposition~\ref{prop:density-envelopes}(i) and
\eqref{eq:counting-mass-hypothesis} give
\[
 \int_G D_{j,B,R}^+\,dg\ll_E C_0B^{\kappa n}.
\]
Both envelopes satisfy $0\le D_{j,B,R}^-\le D_{j,B,R}^+\le1$.
They therefore belong to $L^1(G)\cap L^u(G)$, and
$\|D_{j,B,R}^\pm\|_u^u\le\|D_{j,B,R}^\pm\|_1$ for $1<u<2$.
Choose $1<u<2$ so that $\kappa n(1/u-1/2)/q_n<\varepsilon$.
Since
\[
 1-\frac1{q_n}+\frac1{uq_n}
 =1-\frac1{2q_n}+\frac1{q_n}\left(\frac1u-\frac12\right),
\]
Proposition~\ref{prop:spectral} gives
\[
 \begin{aligned}
 \|\pi_{\Gamma_j}^0(D_{j,B,R}^\pm)\|_{2\to2}
 &\ll_{E,u}\|D_{j,B,R}^\pm\|_1^{\,1-1/q_n+1/(uq_n)}\\
 &\ll_{E,C_0,\varepsilon}
 B^{\kappa n-\kappa n/(2q_n)+\varepsilon}.
 \end{aligned}
\]
By \eqref{eq:small-ball-volume} and \eqref{eq:modulus-smoothing-width},
\[
 \frac1{\sqrt{\vol(O_{\epsilon_L})\vol(O_{\epsilon_R})}}
 \ll_E(\epsilon_L\epsilon_R)^{-\rho/2}
 =(2C_E)^\rho R^{\rho\beta}h^{-\rho}.
\]
As $C_E\ge2$, the cutoff properties imply
$\chi_R\le\chi_{C_ER}$. Thus
\eqref{eq:counting-mass-hypothesis} ensures that the main integral
is finite. We now apply Proposition~\ref{prop:smoothing}, the counting
identity in Proposition~\ref{prop:weighted-sectors}(i), and the Haar
identity \eqref{eq:weighted-haar-main-term} to obtain
\[
 \begin{aligned}
 |S_R(B)-M_R(B)|
 &\ll_{E,C_0,\varepsilon}
 \sum_{j=1}^J\int_G(D_{j,B,R}^+-D_{j,B,R}^-)\,dg\\
 &\quad+R^{\rho\beta}h^{-\rho}
 B^{\kappa n-\kappa n/(2q_n)+\varepsilon}+1.
 \end{aligned}
\]
The right-hand side is finite, and Proposition~\ref{prop:density-envelopes}(ii)
gives \eqref{eq:cutoff-counting-estimate} for normalized $\varphi$, with
its $O_E(h)$ term absorbed into $O_E(1)$. Applying this estimate to the
positive and negative parts of the real and imaginary parts of $\varphi$
restores the factor $\|\varphi\|_{\BL}$ and gives absolute convergence
for complex $\varphi$.
\end{proof}

\subsection{Equidistribution for compactly supported tests}\label{subsec:equidistribution}

\begin{theorem}\label{prop:sector-equidistribution}
\begin{enumerate}[label=\textup{(\roman*)}]
\item Let $C\subset\mathcal X_E$ be compact, $0<a<b<\infty$, and
$0<\delta<\delta_0$. Then, for every bounded Lipschitz function
$\Psi:\mathcal Z_E\to\C$ with $\supp\Psi\subset[a,b]\times C$
and every $B\ge2$,
\begin{align}
 B^{-\kappa n}\sum_{\substack{W\subset E\\\dim_KW=d}}
 \Psi(H(W)/B,x_E(W))
 &=\mathfrak a_{K,n,d}
 \int_0^\infty\int_{\mathcal X_E}
 \Psi(r,x)\kappa n r^{\kappa n-1}\,d\mu_E(x)\,dr\notag\\*
 &\quad+O_{E,C,a,b,\delta}(\|\Psi\|_{\BL}B^{-\delta}).
 \label{eq:compact-radial-count}
\end{align}
\item For every $f\in C_c(\mathcal X_E)$,
\begin{equation}\label{eq:compact-test-limit}
 \lim_{B\to\infty}B^{-\kappa n}
 \sum_{\substack{W\subset E,\ \dim_KW=d\\H(W)\le B}}f(x_E(W))
 =\mathfrak a_{K,n,d}\int_{\mathcal X_E}f\,d\mu_E.
\end{equation}
\item If $A_1\subset\Gr_{d,\infty}(E)$ is a
$\sigma_{E,d}$-continuity set and $A_2\subset\cY_{E,d}$ is a
relatively compact $\nu_{E,d}$-continuity set, then, as $B\to\infty$,
\begin{align}
 &\#\{W\subset E:\dim_KW=d,\ H(W)\le B,\ W_\infty\in A_1,
       \ ([\widehat W],[\widehat{E/W}])\in A_2\}\notag\\*
 &\qquad=
 \mathfrak a_{K,n,d}\sigma_{E,d}(A_1)\nu_{E,d}(A_2)
 B^{\kappa n}+o(B^{\kappa n}).
 \label{eq:sector-equidistribution}
\end{align}
\end{enumerate}
\end{theorem}

\begin{proof}
Fix $R\ge2$ and $b>0$, and let
$\varphi:\mathcal Z_E\to\C$ be bounded Lipschitz.
Apply Proposition~\ref{prop:compact-count} with
$\Omega=(0,b]\times\mathcal X_E$. For $0<h<h_1(E)$,
\[
 \begin{aligned}
 \int_{\Omega_h^+}(1+r)\chi_{C_ER}(x)\,d\lambda_E(r,x)
 &\le\kappa n\int_0^{be^h}(1+r)r^{\kappa n-1}\,dr\ll_{E,b}1,\\
 \lambda_E(\partial_h\Omega)
 &=b^{\kappa n}(e^{\kappa nh}-e^{-\kappa nh})\ll_{E,b}h.
 \end{aligned}
\]
The integral bound verifies \eqref{eq:counting-mass-hypothesis}.
Choose $h=B^{-\delta_0}<h_1(E)$ when $B$ is sufficiently large.
Recall that $\delta_0=\kappa n/(2q_n(\rho+1))$; hence
\[
 h^{-\rho}B^{-\kappa n/(2q_n)}
 =B^{\rho\delta_0-\kappa n/(2q_n)}=B^{-\delta_0}.
\]
Proposition~\ref{prop:compact-count}, with $R$ fixed, therefore gives,
for every $\varepsilon>0$,
\begin{equation}\label{eq:fixed-cutoff-count}
 B^{-\kappa n}|S_R(B)-M_R(B)|
 \ll_{E,R,b,\varepsilon}\|\varphi\|_{\BL}B^{-\delta_0+\varepsilon}.
\end{equation}
The term $B^{-\kappa n}$ is absorbed, since
$\delta_0<\kappa n$.
On the remaining bounded interval of $B\ge2$, Northcott finiteness
bounds the sum over $H(W)\le bB$ by a constant times
$\|\varphi\|_\infty$, and the main term has the same bound. Enlarging
the implied constant gives
\eqref{eq:fixed-cutoff-count} for all $B\ge2$.

For (i), choose $R$ so that
$C\subset\Gr_{d,\infty}(E)\times\mathcal Y_R$.
Apply \eqref{eq:fixed-cutoff-count} with $\varphi=\Psi$ and
$0<\varepsilon<\delta_0-\delta$.
Since $\chi_R\Psi=\Psi$ and $\Psi(r,x)=0$ for $r>b$,
the arithmetic sum has no height restriction, while the radial
integral extends to $(0,\infty)$. This proves
\eqref{eq:compact-radial-count}.

For a compactly supported Lipschitz function $f$ on $\mathcal X_E$,
choose $R$ with $\chi_Rf=f$. Apply \eqref{eq:fixed-cutoff-count}
with $b=1$, $\varphi(r,x)=f(x)$, and $0<\varepsilon<\delta_0$.
Since $\int_0^1\kappa n r^{\kappa n-1}\,dr=1$, this proves (ii)
for such $f$.

For $f\in C_c(\mathcal X_E)$, choose a compact neighborhood $C'$
of $\supp f$ and $S\ge2$ with $\chi_S=1$ on $C'$.
Equation~\eqref{eq:fixed-cutoff-count}, with $R=S$, $b=1$,
$\varphi=1$, and $\varepsilon=\delta_0/2$, gives
\[
 \begin{aligned}
 &B^{-\kappa n}\#\{W:\dim_KW=d,\ H(W)\le B,\ x_E(W)\in C'\}\\
 &\qquad\le B^{-\kappa n}
   \sum_{\substack{W\subset E,\ \dim_KW=d\\H(W)\le B}}
     \chi_S(x_E(W))=O_{E,C'}(1).
 \end{aligned}
\]
First approximate $f$ uniformly on $C'$ by Lipschitz functions
and extend them to $\mathcal X_E$. Multiplying these functions by
a fixed Lipschitz cutoff supported in $C'$ and equal to one on
$\supp f$, we obtain Lipschitz functions $f_m$ supported in $C'$
with $\|f_m-f\|_\infty\to0$.
Applying the limit for compactly supported Lipschitz
functions to $f_m$, and using the preceding uniform bound and
$\mu_E(C')<\infty$, we obtain
\[
 \limsup_{B\to\infty}\left|
 B^{-\kappa n}\sum_{\substack{W\subset E,\ \dim_KW=d\\H(W)\le B}}
 f(x_E(W))-\mathfrak a_{K,n,d}\int_{\mathcal X_E}f\,d\mu_E
 \right|\ll_{E,C'}\|f-f_m\|_\infty.
\]
Letting $m\to\infty$ proves (ii).

For (iii), recall the following form of the Portmanteau
theorem: on a locally compact second countable Hausdorff
space, vague convergence of positive Radon measures implies
convergence of their masses on every relatively compact Borel
set whose boundary has zero limiting measure
\cite[Lemma~23.17]{KallenbergProbability}. By (ii),
\[
 B^{-\kappa n}
 \sum_{\substack{W\subset E,\ \dim_KW=d\\H(W)\le B}}
 \delta_{x_E(W)}
 \xrightarrow[B\to\infty]{\mathrm{vague}}
 \mathfrak a_{K,n,d}\mu_E,
\]
where $\delta_x$ denotes the Dirac measure at $x$.
The set $A_1\times A_2$ is relatively compact because
$\Gr_{d,\infty}(E)$ is compact and $A_2$ is relatively compact.
Its boundary is contained in
\[
 (\partial A_1\times\overline{A_2})
 \cup(\overline{A_1}\times\partial A_2),
\]
The displayed union has $\mu_E$-measure zero: indeed, $\mu_E=\sigma_{E,d}\otimes\nu_{E,d}$, and $A_1,A_2$ are continuity sets for the respective factor measures.
Applying the Portmanteau criterion to $A_1\times A_2$ proves
\eqref{eq:sector-equidistribution}.
\end{proof}

\section{Quantitative cusp estimates}\label{sec:counting-no-escape}

\subsection{Upper bounds for subspace and flag counts}

\begin{proposition}
\label{prop:projective-bound}
Let $F$ be an $n$-dimensional rigid adelic space over $K$.
Use the convention $H(0)=1$.
\begin{enumerate}[label=\textup{(\roman*)}]
\item There is a flag
\begin{equation}\label{eq:projective-counting-flag}
 0=F_0\subset F_1\subset\cdots\subset F_n=F,
 \qquad\dim_K F_i=i,
\end{equation}
such that, for every $B>0$,
\begin{equation}\label{eq:projective-line-bound}
 \#\{L\subset F:\dim_KL=1,\ H(L)\le B\}
 \ll_{K,n}\sum_{i=0}^n\frac{B^{\kappa i}}{H(F_i)^\kappa}.
\end{equation}

\item Let $1\le r<n$, let $U\subset F$ have dimension $r-1$,
and let $m=n-r+1$.  Then, for $T>0$,
\begin{equation}\label{eq:extension-count-bound}\begin{aligned}
 &\#\{W:U\subset W\subset F,\ \dim_KW=r,\ H(W)\le T\}\\
 &\qquad\ll_F
 1+T^{\kappa m}H(U)^{-\kappa(m-1)}
 +\sum_{i=1}^{m-1}T^{\kappa i}H(U)^{-\kappa(i-1)}.
\end{aligned}\end{equation}
\end{enumerate}
\end{proposition}

\begin{proof}
For (i), the case $n=1$ is immediate.  Assume $n\ge2$.
Recall that $\mathcal M_F$ consists of the vectors $x\in F$ with
$\|x\|_{F,v}\le1$ at every finite place.  On
$F_\infty=\prod_{v\mid\infty}F_v$, the notation from
Section~\ref{subsec:adelic-minima} is
\[
 |x|_F=\max_{v\mid\infty}\|x_v\|_{F,v},\qquad
 \mathcal B_F=\{x\in F_\infty:|x|_F\le1\}.
\]
Lemma~\ref{lem:unit-balance} associates to every line $L$ with
$H(L)\le B$ a distinct vector in
\begin{equation*}
 \mathcal M_F\cap C_{K,n}B\mathcal B_F.
\end{equation*}
Choose $x_i$ and $F_i$ as in Lemma~\ref{lem:block-minima}.
Henk's lattice point estimate \cite[Theorem~1.5]{Henk2002}, applied
in real dimension $\kappa n$, gives
\[
 \#(\mathcal M_F\cap C_{K,n}B\mathcal B_F)
 \ll_{K,n}\prod_{j=1}^{\kappa n}
 \left(1+\frac{B}{\lambda_j(\mathcal M_F,\mathcal B_F)}\right).
\]
Group the $\kappa$ minima in each block using
\eqref{eq:block-comparison} and $(1+x)^\kappa\ll_\kappa1+x^\kappa$:
\begin{equation*}\begin{aligned}
 \#(\mathcal M_F\cap C_{K,n}B\mathcal B_F)
 &\ll_{K,n}\prod_{i=1}^n
       \left(1+\frac{B}{|x_i|_F}\right)^\kappa
 \ll_{K,n}\prod_{i=1}^n
       \left(1+\frac{B^\kappa}{|x_i|_F^\kappa}\right)\\
 &\ll_{K,n}\sum_{i=0}^n
       \frac{B^{\kappa i}}{(|x_1|_F\cdots|x_i|_F)^\kappa}.
\end{aligned}\end{equation*}
The last inequality follows by expanding the product and using
$|x_1|_F\le\cdots\le|x_n|_F$.  Equation~\eqref{eq:flag-height-mu}
then proves (i).

For (ii), put $V=F/U$. The map $W\mapsto W/U$ gives a bijection
between the subspaces counted in (ii) and the $K$-lines $L\subset V$
satisfying $H(L)\le T/H(U)$. By \eqref{eq:subspace-quotient-height},
\begin{equation*}
 H(W/U)=\frac{H(W)}{H(U)},
 \qquad H(V)=\frac{H(F)}{H(U)}.
\end{equation*}
Apply (i) to $V$ with $B=T/H(U)$, and let
\begin{equation*}
 0=V_0\subset V_1\subset\cdots\subset V_m=V
\end{equation*}
be the resulting flag.  Let $b_F>0$ be a fixed lower bound for the heights of all
nonzero proper $K$-subspaces of $F$, as in \eqref{eq:beta-E}.
For $1\le i<m$, the inverse image of $V_i$ in $F$ is a nonzero
proper $K$-subspace. Thus \eqref{eq:beta-E} and
\eqref{eq:subspace-quotient-height} give
\begin{equation*}
 H(V_i)\ge\frac{b_F}{H(U)}.
\end{equation*}
For $i=m$, since $V_m=V$,
\begin{equation*}
 \frac{(T/H(U))^{\kappa m}}{H(V)^\kappa}
 =H(F)^{-\kappa}T^{\kappa m}H(U)^{-\kappa(m-1)}
 \ll_F T^{\kappa m}H(U)^{-\kappa(m-1)}.
\end{equation*}
For $1\le i<m$,
\begin{equation*}
 \frac{(T/H(U))^{\kappa i}}{H(V_i)^\kappa}
 \ll_F T^{\kappa i}H(U)^{-\kappa(i-1)}.
\end{equation*}
\end{proof}

Let $H$ be a height function on a countable set with
$H(x)\ge h_0>0$ for every $x$. Suppose that, for some $C,q>0$,
\begin{equation*}
 \#\{x:H(x)\le A\}\le C(1+A^q)\qquad(A>0).
\end{equation*}
For $0\le s<q$ and $A>0$,
\begin{equation}\label{eq:weighted-sum}
 \sum_{H(x)\le A}H(x)^{-s}\ll_{h_0,q,s} C A^{q-s}.
\end{equation}
For $A\ge h_0$, Abel--Stieltjes summation gives the following
identity, with the integral term interpreted as zero when $s=0$:
\begin{equation*}
 \sum_{H(x)\le A}H(x)^{-s}
 =A^{-s}\#\{x:H(x)\le A\}
 +s\int_{h_0}^{A}t^{-s-1}\#\{x:H(x)\le t\}\,dt.
\end{equation*}
For $t\ge h_0$, the counting bound is at most
$C(1+h_0^{-q})t^q$, which proves \eqref{eq:weighted-sum}.
If $A<h_0$, the sum is empty.

\begin{lemma}\label{lem:adapted-hyperplane}
Let $F$ be an $r$-dimensional rigid adelic space over $K$, with $r\ge2$.
There is a hyperplane $U\subset F$ such that
\begin{equation}\label{eq:adapted-hyperplane-height}
 H(U)\ll_{K,r}H(F)^{(r-1)/r}.
\end{equation}
If $R>0$ and $\lambda_r(\widehat F)>R$, then $U$ may be chosen so that
\begin{equation}\label{eq:adapted-hyperplane-saving}
 H(U)\ll_{K,r}R^{-1}H(F)^{(r-1)/r}.
\end{equation}
\end{lemma}

\begin{proof}
Choose the flag $F_0\subset\cdots\subset F_r=F$ from
Lemma~\ref{lem:block-minima} and let $U=F_{r-1}$.  By
\eqref{eq:flag-height-mu},
\[
 H(U)\asymp_{K,r}\prod_{i=1}^{r-1}|x_i|_F
 \asymp_{K,r}\frac{H(F)}{|x_r|_F}
 \asymp_{K,r}\frac{H(F)}{\lambda_r(F)}.
\]
Lemma~\ref{lem:block-minima}(iii) gives
$\lambda_r(F)\gg_{K,r}H(F)^{1/r}$, which proves
\eqref{eq:adapted-hyperplane-height}.  If $\lambda_r(\widehat F)>R$, then $\lambda_r(F)>R H(F)^{1/r}$. Substituting this inequality into the preceding hyperplane estimate gives \eqref{eq:adapted-hyperplane-saving}.
\end{proof}

\begin{proposition}\label{prop:subspace-flag-bounds}
\begin{enumerate}[label=\textup{(\roman*)}]
\item For every $1\le r<n$ and $B>0$,
\begin{equation}\label{eq:grassmann-upper}
 \#\{W\subset E:\dim_KW=r,\ H(W)\le B\}
 \ll_E 1+B^{\kappa n}.
\end{equation}
\item For every $2\le r<n$, $B\ge1$, and $0<u\le1$,
\begin{equation}\label{eq:scaled-flag-bound}
\begin{aligned}
&\#\{(U,W):U\subset W\subset E,\ \dim_KU=r-1,\ \dim_KW=r,\\
&\hspace{25mm}H(U)\le uB^{(r-1)/r},\ H(W)\le B\} \ll_E u^{\kappa r}B^{\kappa n}.
\end{aligned}
\end{equation}
\end{enumerate}
\end{proposition}

\begin{proof}
We argue by induction on $r$. For $r=1$, assertion (i) follows
from Proposition~\ref{prop:projective-bound}(i) applied to the fixed
space $E$.

Let $2\le r<n$ and assume (i) in rank $r-1$. Let $m=n-r+1$.
For $A>0$, summing Proposition~\ref{prop:projective-bound}(ii)
over the $(r-1)$-dimensional subspaces of height at most $A$ and using
\eqref{eq:weighted-sum} gives
\[
\begin{aligned}
&\#\left\{(U,W):\begin{array}{l}U\subset W\subset E,\ \dim_KU=r-1,\ \dim_KW=r,\\ H(U)\le A,\ H(W)\le B\end{array}\right\}\\
&\quad\ll_E
 \sum_{\substack{\dim_KU=r-1\\H(U)\le A}}
 \left(1+B^{\kappa m}H(U)^{-\kappa(m-1)}
       +\sum_{i=1}^{m-1}B^{\kappa i}H(U)^{-\kappa(i-1)}\right)\\
&\quad\ll_E A^{\kappa n}
       +B^{\kappa m}A^{\kappa r}
       +\sum_{i=1}^{m-1}B^{\kappa i}A^{\kappa(n-i+1)}.
\end{aligned}
\]
The weighted-sum estimate applies: every exponent of $H(U)^{-1}$ is less than $\kappa n$, and $H(U)\ge b_E$ by \eqref{eq:beta-E}.
Set $A=uB^{(r-1)/r}$. For $B\ge1$ and $0<u\le1$,
\[
\begin{aligned}
 A^{\kappa n}
 &=u^{\kappa n}B^{\kappa n(r-1)/r}
 \le u^{\kappa r}B^{\kappa n},\\
 B^{\kappa m}A^{\kappa r}
 &=u^{\kappa r}B^{\kappa n},\\
 B^{\kappa i}A^{\kappa(n-i+1)}
 &=u^{\kappa(n-i+1)}
   B^{\kappa n-\kappa(m-i)/r}
 \le u^{\kappa r}B^{\kappa n}
 \qquad(1\le i<m).
\end{aligned}
\]
Here $n-i+1\ge r+1$. This proves (ii) in rank $r$.

To prove (i), suppose $B\ge1$. By Lemma~\ref{lem:adapted-hyperplane},
every $r$-dimensional subspace $W$ counted in (i) contains a
hyperplane $U$ with $H(U)\le C_{K,r}B^{(r-1)/r}$, where we may take
$C_{K,r}\ge1$. We now apply (ii), already proved in rank $r$,
with $u=1$ and height parameter
\[
 B'=C_{K,r}^{\,r/(r-1)}B.
\]
Then $H(U)\le(B')^{(r-1)/r}$ and $H(W)\le B'$. Counting these
pairs gives
\[
 \#\{W:\dim_KW=r,\ H(W)\le B\}
 \ll_E(B')^{\kappa n}\ll_E B^{\kappa n}.
\]
For $0<B<1$, Northcott finiteness gives (i). This completes the
induction.
\end{proof}

\subsection{Arithmetic and measure bounds in the cusp}

Recall that $\mathcal Y_R$ consists of the pairs $(L_1,L_2)\in\cY_{E,d}$
with $\lambda_d(L_1)\le R$ and $\lambda_e(L_2)\le R$.
For $R,B\ge1$, let
\begin{equation*}\begin{aligned}
 \mathcal T_d^{(1)}(B,R)
 &=\#\{W\subset E:\dim_KW=d,\ H(W)\le B,
                    \ \lambda_d(\widehat W)>R\},\\
 \mathcal T_d^{(2)}(B,R)
 &=\#\{W\subset E:\dim_KW=d,\ H(W)\le B,
                    \ \lambda_e(\widehat{E/W})>R\}.
\end{aligned}\end{equation*}
Here \(e=n-d\).

\begin{proposition}
\label{prop:arithmetic-tail}
Let $B,R\ge2$.
\begin{enumerate}[label=\textup{(\roman*)}]
\item If $d\ge2$, then
\begin{equation}\label{eq:subspace-tail-explicit}
 \mathcal T_d^{(1)}(B,R)\ll_E R^{-\kappa d}B^{\kappa n}.
\end{equation}
If $d=1$, then $\mathcal T_d^{(1)}(B,R)=0$ for every $R>1$.
\item If $e\ge2$, then
\begin{equation}\label{eq:quotient-tail-explicit}
 \mathcal T_d^{(2)}(B,R)\ll_E R^{-\kappa e/(e-1)}B^{\kappa n}.
\end{equation}
If $e=1$, then $\mathcal T_d^{(2)}(B,R)=0$ for every $R>1$.
\end{enumerate}
\end{proposition}

\begin{proof}
We prove the subspace estimate for every $B\ge1$ and $R\ge2$.
Let $d\ge2$.  For each $W$ counted by $\mathcal T_d^{(1)}(B,R)$,
Lemma~\ref{lem:adapted-hyperplane} gives a $(d-1)$-dimensional subspace $U\subset W$
with
\[
 H(U)\le C_{K,d}R^{-1}H(W)^{(d-1)/d}
 \le C_{K,d}R^{-1}B^{(d-1)/d}.
\]
For $R\ge\max\{2,C_{K,d}\}$, apply
Proposition~\ref{prop:subspace-flag-bounds}(ii) with $r=d$ and
$u=C_{K,d}/R$. Counting all such pairs $(U,W)$ gives
\[
 \mathcal T_d^{(1)}(B,R)\ll_E R^{-\kappa d}B^{\kappa n}.
\]
For the remaining bounded range of $R$, the same estimate follows
from \eqref{eq:grassmann-upper}, since $B\ge1$.

For (ii), first let $F$ be an $r$-dimensional rigid adelic space. Then
\begin{equation}\label{eq:dual-minima-inequality}
 1\le\lambda_1(F)\lambda_r(F^\vee).
\end{equation}
Choose $0\ne x\in F$ and independent $f_1,\ldots,f_r\in F^\vee$
realizing $\lambda_1(F)$ and $\lambda_r(F^\vee)$; these minima are
attained by Northcott's theorem. For some $j$, $f_j(x)\ne0$.
The local operator norm inequalities and the product formula give
\[
 1=\prod_v|f_j(x)|_v^{n_v/\kappa}
 \le H_F(x)H_{F^\vee}(f_j)
 \le\lambda_1(F)\lambda_r(F^\vee).
\]
If
$H(F)=1$, $r\ge2$, and $\lambda_r(F)>R>0$, then
Lemma~\ref{lem:block-minima}(iii) and monotonicity give
\begin{equation*}
 \lambda_1(F)^{r-1}\lambda_r(F)
 \le\prod_{i=1}^r\lambda_i(F)\ll_{K,r}1.
\end{equation*}
Hence $\lambda_1(F)\ll_{K,r}R^{-1/(r-1)}$, and
\eqref{eq:dual-minima-inequality} yields
\begin{equation}\label{eq:dual-last-minimum-bound}
 \lambda_r(F^\vee)\gg_{K,r}R^{1/(r-1)}.
\end{equation}

Now let $Q=E/W$ and
\[
 W^\perp=\{\varphi\in E^\vee:\varphi|_W=0\}.
\]
Gaudron's canonical isometry and height identities
\cite[Theorem~4 and Proposition~5]{Gaudron} give a bijection from
the $d$-dimensional subspaces of $E$ to the $e$-dimensional subspaces of $E^\vee$, with
\[
 Q^\vee\simeq W^\perp,\qquad
 H(W^\perp)=\frac{H(W)}{H(E)},\qquad
 (\widehat Q)^\vee\simeq\widehat{W^\perp}.
\]
If $e\ge2$ and $\lambda_e(\widehat Q)>R$, then
\eqref{eq:dual-last-minimum-bound} gives
\[
 \lambda_e(\widehat{W^\perp})
 \gg_{K,e}R^{1/(e-1)}.
\]
For sufficiently large $R$, apply the subspace estimate just proved to $E^\vee$, with rank $e$. Take the height cutoff to be $B^\vee=\max\{1,B/H(E)\}$ and the threshold to be $c_{K,e}R^{1/(e-1)}\ge2$, for a suitable fixed $c_{K,e}>0$.
Since $B^\vee\ll_E B$ for $B\ge2$, this gives
\[
 \mathcal T_d^{(2)}(B,R)
 \ll_E R^{-\kappa e/(e-1)}B^{\kappa n}.
\]
For bounded $R\ge2$, use \eqref{eq:grassmann-upper} again.

The rank-one tails vanish for $R>1$, since $\lambda_1=1$ at height one.
\end{proof}

\begin{theorem}\label{thm:no-escape}
Assume $n\ge3$, and let $b_{d,e}>0$ be the exponent defined in
Section~\ref{subsec:numerical-parameters}. For every $R\ge2$, the
following estimates hold.
\begin{enumerate}[label=\textup{(\roman*)}]
\item For every $B\ge2$,
\begin{equation}\label{eq:arithmetic-tail}
 \#\{W\subset E:\dim_KW=d,\ H(W)\le B,\ y_W\notin\mathcal Y_R\}
 \ll_E R^{-b_{d,e}}B^{\kappa n}.
\end{equation}
\item
\begin{equation}\label{eq:measure-tail}
 \mu_E\!\left(\mathcal X_E\setminus
   (\Gr_{d,\infty}(E)\times\mathcal Y_R)\right)
 =\nu_{E,d}(\cY_{E,d}\setminus\mathcal Y_R)
 \ll_E R^{-b_{d,e}}.
\end{equation}
\end{enumerate}
\end{theorem}

\begin{proof}
For (i), the exceptional set is the union of the two sets counted
in Proposition~\ref{prop:arithmetic-tail}. Recall from
\eqref{eq:bde} that $b_{d,e}$ is the minimum of the applicable tail
exponents: $\kappa d$ for $d\ge2$ and $\kappa e/(e-1)$ for $e\ge2$.
The estimates
\eqref{eq:subspace-tail-explicit} and \eqref{eq:quotient-tail-explicit}
give \eqref{eq:arithmetic-tail}.

For (ii), the normalization at the end of
Section~\ref{subsec:metrics-exponents} preserves $\mathcal Y_R$ and
$\nu_{E,d}$. We may therefore assume $H(E)=1$.
By Theorem~\ref{prop:sector-equidistribution}(ii), the normalized
counting measures for subspaces of bounded height converge vaguely
on $\mathcal X_E$ to $\mathfrak a_{K,n,d}\mu_E$.
Since $\mathcal Y_R$ is compact, the set
$\Gr_{d,\infty}(E)\times(\cY_{E,d}\setminus\mathcal Y_R)$ is open.
Exhaust this open set by relatively compact continuity sets, as in \cite[Lemma~23.17]{KallenbergProbability}, and apply the Portmanteau inequality for vague convergence. Together with (i), this gives
\[
\begin{aligned}
&\mathfrak a_{K,n,d}\nu_{E,d}(\cY_{E,d}\setminus\mathcal Y_R)\\
&\quad\le\liminf_{B\to\infty}B^{-\kappa n}
 \#\{W:\dim_KW=d,\ H(W)\le B,\ y_W\notin\mathcal Y_R\}
 \ll_E R^{-b_{d,e}}.
\end{aligned}
\]
The equality in \eqref{eq:measure-tail} follows because
$\mu_E=\sigma_{E,d}\otimes\nu_{E,d}$ and $\sigma_{E,d}$ has mass one.
\end{proof}

\section{Proofs of the main theorems}\label{sec:main-theorem-proofs}

\subsection{Proofs of the results for joint regions}
\label{subsec:weighted-count-proof}

Recall $\mathcal E_B(\Omega,\varphi)$ from
\eqref{eq:joint-region-discrepancy}, $\tau_E$ from \eqref{eq:cusp-size},
and the cutoffs $\chi_R$ from \eqref{eq:shape-cutoff-properties}.
For $n\ge3$ and $0\le\gamma<b_{d,e}/(\kappa n)$, let
$b_\gamma=b_{d,e}-\gamma\kappa n$ and
\begin{equation}\label{eq:radial-moment-factor}
 J_\gamma(R)=
 \begin{cases}
 1,&\gamma(\kappa n+1)<b_{d,e},\\
 1+\log R,&\gamma(\kappa n+1)=b_{d,e},\\
 R^{\gamma(\kappa n+1)-b_{d,e}},&\gamma(\kappa n+1)>b_{d,e}.
 \end{cases}
\end{equation}

\begin{theorem}\label{thm:localized-joint-count}
Let $M>0$. If $n\ge3$, fix $0\le\gamma<b_{d,e}/(\kappa n)$;
if $n=2$, set $\gamma=0$. Let $\Omega$ be closed in $\mathcal Z_E$
and contained in the region \eqref{eq:joint-region-growth}.
Let $\varphi:\mathcal Z_E\to\C$ be bounded Lipschitz.
There is $h_1(E)>0$ such that the following estimates hold for
every $B,R\ge2$, $0<h<h_1(E)$, and $\varepsilon>0$.
\begin{enumerate}[label=\textup{(\roman*)}]
\item If $n\ge3$, then
\begin{equation}\label{eq:global-error}
 \begin{aligned}
 |\mathcal E_B(\Omega,\varphi)|
 &\ll_{E,M,\gamma,\varepsilon}\|\varphi\|_{\BL}
 \biggl[B^{\kappa n}\left(R^{-b_\gamma}+hJ_\gamma(R)+
                         \int_{\partial_h\Omega}\chi_R(x)\,d\lambda_E(r,x)\right)\\
 &\hspace{35mm}+R^{\rho\beta}h^{-\rho}
              B^{\kappa n-\kappa n/(2q_n)+\varepsilon}+1\biggr].
 \end{aligned}
\end{equation}
\item If $n=2$, then
\begin{equation}\label{eq:global-modulus-rank-two}
 |\mathcal E_B(\Omega,\varphi)|
 \ll_{E,M,\varepsilon}\|\varphi\|_{\BL}
 \biggl[B^{\kappa n}\bigl(h+\lambda_E(\partial_h\Omega)\bigr)
 {}+h^{-\rho}B^{\kappa n-\kappa n/(2q_n)+\varepsilon}+1\biggr].
\end{equation}
\end{enumerate}
\end{theorem}

\begin{proof}
For $n\ge3$, first establish the tail bounds for any Borel $\Omega$
satisfying \eqref{eq:joint-region-growth}. On the shell
$2^jR<\tau_E(x_E(W))\le2^{j+1}R$, the condition
$(H(W)/B,x_E(W))\in\Omega$ implies $H(W)\le T_j$, where
\[
 T_j=\max\{2,MB(2^{j+1}R)^\gamma\}\ll_{M,\gamma}B(2^jR)^\gamma.
\]
Applying Theorem~\ref{thm:no-escape}(i) with height bound $T_j$ gives
\begin{equation}\label{eq:growing-region-arithmetic-tail}
 \begin{aligned}
 &\#\{W:\dim_KW=d,\ (H(W)/B,x_E(W))\in\Omega,\
                              \tau_E(x_E(W))>R\}\\
 &\quad\ll_{E,M,\gamma}B^{\kappa n}
                 \sum_{j\ge0}(2^jR)^{\gamma\kappa n-b_{d,e}}
 \ll_{E,M,\gamma}B^{\kappa n}R^{-b_\gamma}.
 \end{aligned}
\end{equation}
Applying Theorem~\ref{thm:no-escape}(ii) on the same dyadic shells and summing gives
\begin{equation}\label{eq:growing-region-measure-tail}
 \int_{\{\tau_E>R\}}\tau_E^{\gamma\kappa n}\,d\mu_E
 \ll_{E,\gamma}R^{-b_\gamma},\qquad
 \int_{\Omega\cap\{\tau_E>R\}}d\lambda_E
 \ll_{E,M,\gamma}R^{-b_\gamma}.
\end{equation}
On $\{\tau_E\le2\}$ the height bound is $H(W)\le M2^\gamma B$.
Equation \eqref{eq:grassmann-upper} and the preceding bounds imply
\begin{equation}\label{eq:growing-region-uniform-bounds}
\begin{gathered}
\#\{W:\dim_KW=d,\ (H(W)/B,x_E(W))\in\Omega\}
 \ll_{E,M,\gamma}B^{\kappa n},\\
\lambda_E(\Omega)\ll_{E,M,\gamma}1.
\end{gathered}
\end{equation}
For every fixed $C\ge1$, dyadic decomposition also gives
\begin{equation}\label{eq:truncated-radial-moment}
 \int_{\{\tau_E\le CR\}}\tau_E^{\gamma(\kappa n+1)}\,d\mu_E
 \ll_{E,C,\gamma}
 1+\sum_{0\le j\le\lceil\log_2(CR)\rceil}
         2^{j(\gamma(\kappa n+1)-b_{d,e})}
 \ll_{E,C,\gamma}J_\gamma(R).
\end{equation}

Assume now that $\Omega$ is closed and $H(E)=1$. Let $C_E$ be the constant in
Proposition~\ref{prop:density-envelopes}, and take $h_1(E)$ from
Proposition~\ref{prop:compact-count}. After decreasing $h_1(E)$,
Proposition~\ref{prop:norm-shape-stability}(i) bounds the change of
$\log\tau_E$ by $O_E(h)$ at distance $h<h_1(E)$.
The growth condition \eqref{eq:joint-region-growth} and the support
of $\chi_{C_ER}$ therefore give
\begin{equation}\label{eq:growing-density-support}
 \Omega_h^+\cap\supp\chi_{C_ER}
 \subset\{(r,x):\tau_E(x)\le CR,\ 0<r\le C\tau_E(x)^\gamma\},
\end{equation}
where the support is understood in $\mathcal Z_E$ and
$C\ge1$ depends only on $E,M,\gamma$.
Indeed, logarithmic distance at most $h$ changes $r$ by at most
$e^h$ and $\tau_E$ by at most $e^{O_E(h)}$.

The radial integral and \eqref{eq:growing-region-measure-tail} give
\[
 \begin{aligned}
 \int_{\Omega_h^+}\chi_{C_ER}(x)\,d\lambda_E(r,x)
 &\le\int_{\mathcal X_E}\int_0^{C\tau_E(x)^\gamma}
                  \kappa n r^{\kappa n-1}\,dr\,d\mu_E(x)\\
 &=C^{\kappa n}\int_{\mathcal X_E}\tau_E^{\gamma\kappa n}\,d\mu_E
 \ll_{E,M,\gamma}1.
 \end{aligned}
\]
This verifies \eqref{eq:counting-mass-hypothesis} with $C_0$ depending
only on $E,M,\gamma$. Equations~\eqref{eq:growing-density-support}
and \eqref{eq:truncated-radial-moment} give
\[
 \begin{aligned}
 \int_{\Omega_h^+}(1+r)\chi_{C_ER}(x)\,d\lambda_E(r,x)
 &\le\int_{\{\tau_E\le CR\}}\int_0^{C\tau_E^\gamma}
              (1+r)\kappa n r^{\kappa n-1}\,dr\,d\mu_E\\
 &\ll_{E,M,\gamma}\int_{\{\tau_E\le CR\}}
       \bigl(\tau_E^{\gamma\kappa n}
                   +\tau_E^{\gamma(\kappa n+1)}\bigr)\,d\mu_E
 \ll_{E,M,\gamma}J_\gamma(R).
 \end{aligned}
\]
Since $\mathcal E_B(\Omega,\varphi\chi_R)=S_R(B)-M_R(B)$ when
$H(E)=1$, Proposition~\ref{prop:compact-count} yields
\[
\begin{aligned}
|\mathcal E_B(\Omega,\varphi\chi_R)|
 &\ll_{E,M,\gamma,\varepsilon}\|\varphi\|_{\BL}
 \biggl[B^{\kappa n}\left(hJ_\gamma(R)+
                   \int_{\partial_h\Omega}\chi_R\,d\lambda_E\right)\\
 &\hspace{30mm}+R^{\rho\beta}h^{-\rho}
            B^{\kappa n-\kappa n/(2q_n)+\varepsilon}+1\biggr].
\end{aligned}
\]
The tail bounds \eqref{eq:growing-region-arithmetic-tail} and
\eqref{eq:growing-region-measure-tail}, together with
$0\le1-\chi_R\le\mathbf1_{\{\tau_E>R\}}$, give
\[
 |\mathcal E_B(\Omega,\varphi)-\mathcal E_B(\Omega,\varphi\chi_R)|
 \ll_{E,M,\gamma}\|\varphi\|_\infty B^{\kappa n}R^{-b_\gamma}.
\]
Adding these estimates proves (i).
For $n=2$, $\gamma=0$, $\chi_R=1$, and $\beta=0$. Since
\[
 \int_{\Omega_h^+}(1+r)\,d\lambda_E(r,x)
 \le\kappa n\int_0^{Me^h}(1+r)r^{\kappa n-1}\,dr\ll_{E,M}1,
\]
Proposition~\ref{prop:compact-count} gives (ii).

For general $E$, use the normalization described at the end of
Section~\ref{subsec:metrics-exponents}: apply the normalized estimate
to $\widehat E$ with $B'=H(E)^{-d/n}B$. Normalization preserves $x_E(W)$,
$\tau_E$, $\chi_R$, $\lambda_E$, and $\|\varphi\|_{\BL}$. Since
\[
 \frac{H_{\widehat E}(W)}{B'}=\frac{H_E(W)}{B},\qquad
 (B')^{\kappa n}=H(E)^{-\kappa d}B^{\kappa n},
\]
the normalized estimate on $\Omega$ gives the stated main term and
error when $B'\ge2$. If $B'<2$, then $B$ is bounded, and the additive
constant suffices by \eqref{eq:growing-region-uniform-bounds} when
$n\ge3$ and by Northcott finiteness, $H(W)\le MB$, and
$\lambda_E(\Omega)\le M^{\kappa n}$ when $n=2$.
\end{proof}

\begin{proof}[Proof of Theorem~\ref{thm:boundary-modulus-count}]
For $n\ge3$, finiteness follows from
\eqref{eq:growing-region-uniform-bounds}. For $n=2$, use Northcott
finiteness and $\lambda_E(\Omega)\le M^{\kappa n}$.
To prove the estimates, first suppose that $\Omega$ is closed in
$\mathcal Z_E$.
For (ii), first let $n\ge3$ and let
$\theta=\max\{b_\gamma/\alpha,\gamma\}$. Apply
Theorem~\ref{thm:localized-joint-count}(i) with
\begin{equation}\label{eq:R-choice}
 R=B^{\delta_{\alpha,\gamma}/b_\gamma},\qquad h=R^{-\theta}.
\end{equation}
The definition \eqref{eq:delta-holder-cusp} gives
\begin{equation}\label{eq:growing-region-balance}
 R^{-b_\gamma}
 =R^{\rho\beta}h^{-\rho}B^{-\kappa n/(2q_n)}
 =B^{-\delta_{\alpha,\gamma}},\qquad
 h^\alpha\le R^{-b_\gamma}.
\end{equation}
Since $\theta\ge b_\gamma$ and $\theta\ge\gamma$,
\eqref{eq:radial-moment-factor} gives
\begin{equation}\label{eq:balanced-radial-moment}
 hJ_\gamma(R)\le R^{-b_\gamma}(1+\log R).
\end{equation}
Indeed, if $\gamma(\kappa n+1)>b_{d,e}$, then
$\gamma(\kappa n+1)-b_{d,e}=\gamma-b_\gamma$ and
$hR^{\gamma-b_\gamma}\le R^{-b_\gamma}$; the other two cases
follow from $h\le R^{-b_\gamma}$.
For large $B$, $R\ge2$ and $h<\min\{h_1(E),h_\partial,1\}$, so
\[
 \int_{\partial_h\Omega}\chi_R\,d\lambda_E
 \le C_\partial h^\alpha\le C_\partial B^{-\delta_{\alpha,\gamma}}.
\]
Taking $0<\varepsilon<\delta_{\alpha,\gamma}-\delta$ in
\eqref{eq:global-error} and absorbing $\log B$ proves the claim,
since $\delta_{\alpha,\gamma}<\kappa n$.

For $n=2$, take $h=B^{-\delta_{\alpha,0}/\alpha}$. Then
\[
 h^\alpha=h^{-\rho}B^{-\kappa n/(2q_n)}=B^{-\delta_{\alpha,0}},
\]
and Theorem~\ref{thm:localized-joint-count}(ii) proves the
claim with $0<\varepsilon<\delta_{\alpha,0}-\delta$.
For bounded $B\ge2$, the implied constant in (ii) may be enlarged
using \eqref{eq:growing-region-uniform-bounds} if $n\ge3$, and
Northcott finiteness with $H(W)\le MB$ and
$\lambda_E(\Omega)\le M^{\kappa n}$ if $n=2$.

For (i) and $n\ge3$, use \eqref{eq:R-choice} with $\alpha=1$.
Choose $B_0=B_0(E,\gamma,h_\partial)$ so that $R\ge2$ and
$h\le B^{-\delta_{1,\gamma}}<\min\{h_1(E),h_\partial,1\}$
for $B\ge B_0$. Monotonicity of $\omega$ gives
\[
 \int_{\partial_h\Omega}\chi_R\,d\lambda_E
 \le\omega(h)\le\omega(B^{-\delta_{1,\gamma}}).
\]
Substituting this bound together with
\eqref{eq:growing-region-balance} and \eqref{eq:balanced-radial-moment}
into \eqref{eq:global-error}, and using
$1+\log B\ll_\varepsilon B^\varepsilon$, proves
\eqref{eq:general-modulus-convergence}.
For $n=2$, take $h=B^{-\delta_{1,0}}$. Since
$h^{-\rho}B^{-\kappa n/(2q_n)}=B^{-\delta_{1,0}}$,
\eqref{eq:global-modulus-rank-two} proves (i).

Now let $\Omega$ be Borel, and take its closure $\overline\Omega$
and topological boundary $\partial\Omega$ in $\mathcal Z_E$.
The closure and the boundary are both closed and satisfy \eqref{eq:joint-region-growth}, since $\tau_E$ is continuous. Since distance to a set equals distance
to its closure, $\mathcal Z_E\setminus\overline\Omega\subset
\mathcal Z_E\setminus\Omega$, and
$\partial\Omega=\overline\Omega\cap
\overline{\mathcal Z_E\setminus\Omega}$, we have, for every $h>0$,
\[
 \partial_h(\overline\Omega)\subset\partial_h\Omega,
 \qquad \partial_h(\partial\Omega)\subset\partial_h\Omega.
\]
Thus both closed sets satisfy the boundary bound in (i) or (ii)
with the original parameters. Since $\partial\Omega\subset\partial_h\Omega$,
either boundary hypothesis gives $\lambda_E(\partial\Omega)=0$.
The inclusion $\overline\Omega\setminus\Omega\subset\partial\Omega$ then yields
\[
 \int_{\overline\Omega}\varphi\,d\lambda_E=\int_\Omega \varphi\,d\lambda_E,
\]
and
\[
 |\mathcal E_B(\Omega,\varphi)|
 \le |\mathcal E_B(\overline\Omega,\varphi)|
       +\|\varphi\|_\infty\,\mathcal E_B(\partial\Omega,1).
\]
Applying the closed-set estimates to $(\overline\Omega,\varphi)$
and $(\partial\Omega,1)$ proves (i) and (ii) with unchanged uniformity.
\end{proof}

\begin{corollary}\label{cor:truncated-boundary-count}
Assume $n\ge3$. Let $M,C_\partial>0$, $0<h_\partial\le1$,
$0<\alpha\le1$, and $\sigma\ge0$. Suppose that
$\Omega\subset(0,M]\times\mathcal X_E$ is closed in $\mathcal Z_E$ and
\begin{equation}\label{eq:truncated-boundary-condition}
 \int_{\partial_h\Omega}\chi_R(x)\,d\lambda_E(r,x)
 \le C_\partial R^\sigma h^\alpha
 \qquad(R\ge2,\ 0<h<h_\partial).
\end{equation}
Let $\theta=\alpha b_{d,e}/(b_{d,e}+\sigma)$.
Then, for every bounded Lipschitz function $\varphi:\mathcal Z_E\to\C$,
every $0<\delta<\delta_{\theta,0}$, and every $B\ge2$,
\[
 |\mathcal E_B(\Omega,\varphi)|
 \ll_{E,M,C_\partial,h_\partial,\alpha,\sigma,\delta}
       \|\varphi\|_{\BL}B^{\kappa n-\delta}.
\]
\end{corollary}

\begin{proof}
For $0<h<h_\partial$, we have
$\partial_h\Omega\subset(0,Me^h]\times\mathcal X_E$ and
$1-\chi_R\le\mathbf1_{\cY_{E,d}\setminus\mathcal Y_R}$.
Hence \eqref{eq:truncated-boundary-condition} and
Theorem~\ref{thm:no-escape}(ii) give
\[
 \begin{aligned}
 \lambda_E(\partial_h\Omega)
 &=\int_{\partial_h\Omega}\chi_R\,d\lambda_E
   +\int_{\partial_h\Omega}(1-\chi_R)\,d\lambda_E\\
 &\le C_\partial R^\sigma h^\alpha
       +(Me^h)^{\kappa n}
         \nu_{E,d}(\cY_{E,d}\setminus\mathcal Y_R)\\
 &\ll_{E,M,C_\partial}R^\sigma h^\alpha+R^{-b_{d,e}}.
 \end{aligned}
\]
For $h<\min\{h_\partial,2^{-(b_{d,e}+\sigma)/\alpha}\}$, take
$R=h^{-\alpha/(b_{d,e}+\sigma)}\ge2$. Then
$R^\sigma h^\alpha=R^{-b_{d,e}}=h^\theta$.
Theorem~\ref{thm:boundary-modulus-count}(ii) applies with $\gamma=0$, and
\[
 \delta_{\theta,0}=
 \frac{\kappa n\,\alpha b_{d,e}}
 {2q_n\bigl((\rho+\alpha)b_{d,e}
                 +\rho\beta\alpha+\rho\sigma\bigr)}
\]
by \eqref{eq:delta-holder-cusp} evaluated at $(\theta,0)$.
\end{proof}

\subsection{Modified heights and equidistribution}
\label{subsec:main-theorem-proofs}

\begin{proof}[Proof of Theorem~\ref{thm:cusp-decaying-height}]
Set $\Omega_Q=\{(r,x):rQ(x)\le1\}$.
Since $\log Q$ is locally Lipschitz, $\Omega_Q$ is closed.
The lower bound $Q(x)\ge c\,\tau_E(x)^{-\gamma}$ gives
\[
 \Omega_Q\subset\{(r,x):0<r\le c^{-1}\tau_E(x)^\gamma\}.
\]
If $n\ge3$, equation~\eqref{eq:growing-region-measure-tail}
bounds $\int_{\mathcal X_E}\tau_E^{\gamma\kappa n}\,d\mu_E$
in terms of $E,\gamma$. When $n=2$, the integral equals one,
since $\gamma=0$ and $\mu_E$ is a probability measure. Thus
\begin{equation}\label{eq:decaying-weight-uniform-bounds}
 \int_{\mathcal X_E}Q(x)^{-\kappa n}\,d\mu_E(x)
 \le c^{-\kappa n}\int_{\mathcal X_E}\tau_E(x)^{\gamma\kappa n}\,d\mu_E(x)
 \ll_{E,c,\gamma}1.
\end{equation}

For $0<h<\min\{1,\epsilon_0/2\}$, the function
$\log r+\log Q(x)$ changes by at most $(1+L)h$ between points at
$d_{\log,E}$-distance at most $h$.
For $(r,x)\in\partial_h\Omega_Q$ and every sufficiently small
$\eta>0$, choose a point in $\Omega_Q$ and a point in its complement,
each at distance less than $h+\eta$ from $(r,x)$.
The function $\log r+\log Q(x)$ is nonpositive on $\Omega_Q$
and positive on its complement. The Lipschitz estimate then gives
$-(1+L)(h+\eta)<\log r+\log Q(x)\le(1+L)(h+\eta)$.
Letting $\eta\to0^+$ yields
\[
 \partial_h\Omega_Q
 \subset\{(r,x):|\log r+\log Q(x)|\le(1+L)h\}.
\]
Integration in $r$ yields
\begin{equation}\label{eq:graph-boundary-estimate}
 \lambda_E(\partial_h\Omega_Q)
 \le\left(e^{\kappa n(1+L)h}-e^{-\kappa n(1+L)h}\right)
      \int_{\mathcal X_E}Q(x)^{-\kappa n}\,d\mu_E(x).
\end{equation}
By \eqref{eq:decaying-weight-uniform-bounds}, the right hand side is
$O_{E,c,\gamma,L}(h)$. Theorem~\ref{thm:boundary-modulus-count}(ii)
applies with $M=c^{-1}$ and $\alpha=1$. Fubini's theorem gives
\begin{equation}\label{eq:height-graph-main-term}
 \int_{\Omega_Q}\varphi\,d\lambda_E
 =\int_{\mathcal X_E}\int_0^{Q(x)^{-1}}
       \varphi(r,x)\kappa n r^{\kappa n-1}\,dr\,d\mu_E(x),
\end{equation}
proving \eqref{eq:cusp-decaying-height-count} with the stated uniformity.
\end{proof}

\begin{proof}[Proof of Theorem~\ref{thm:general-equidistribution}]
For (i), apply Theorem~\ref{thm:cusp-decaying-height} with
$\gamma=0$, $Q\equiv b^{-1}$, and $\varphi=\Psi$. Since $\Psi(r,x)=0$
for $r>b$, the height restriction $H(W)\le bB$ can be removed
from the sum and the radial integral can be extended from $(0,b)$
to $(0,\infty)$. Dividing by $B^{\kappa n}$ gives
\eqref{eq:effective-radial}.
Theorem~\ref{thm:cusp-decaying-height}, with $\gamma=0$,
$Q\equiv b^{-1}$, and $\varphi=1$, also gives
\[
 \sup_{B\ge2}B^{-\kappa n}
       \#\{W\subset E:\dim_KW=d,\ H(W)\le bB\}\ll_{E,b}1.
\]
Every compactly supported continuous test on $\mathcal Z_E$ can
be approximated uniformly by bounded Lipschitz tests supported
in a fixed radial interval $[a,b]$. The uniform counting bound controls the approximation error in the sum. The approximation error in the integral is controlled by $\lambda_E([a,b]\times\mathcal X_E)<\infty$. Letting
$B\to\infty$ and then the approximation error tend to zero proves
vague convergence to $\mathfrak a_{K,n,d}H(E)^{-\kappa d}\lambda_E$.

For (ii), apply Theorem~\ref{thm:cusp-decaying-height} with
$\gamma=0$, $Q\equiv1$, and $\varphi(r,x)=f(x)$. Since
$\int_0^1\kappa n r^{\kappa n-1}\,dr=1$, this gives
\eqref{eq:effective-bounded} for every $B\ge2$.
Taking $f=1$ gives \eqref{eq:total-subspace-count}.

Equation \eqref{eq:total-subspace-count} shows that
$N_{E,d}(B)>0$ for all sufficiently large $B$. Dividing
\eqref{eq:effective-bounded} by \eqref{eq:total-subspace-count} gives,
for every bounded Lipschitz function $g$ on $\mathcal X_E$,
\begin{equation}\label{eq:normalized-equidistribution}
 \frac{1}{N_{E,d}(B)}
 \sum_{\substack{W\subset E,\ \dim_KW=d\\H(W)\le B}}g(x_E(W))
 =\int_{\mathcal X_E}g\,d\mu_E
   +O_{E,\delta}\bigl(\|g\|_{\BL}B^{-\delta}\bigr).
\end{equation}
Let
\[
 P_B=\frac1{N_{E,d}(B)}
 \sum_{\substack{W\subset E,\ \dim_KW=d\\H(W)\le B}}
 \delta_{x_E(W)}.
\]
Convergence against all bounded Lipschitz functions determines weak
convergence of probability measures on the separable metric space
$\mathcal X_E$ \cite[Theorem~11.3.3]{DudleyProbability}.
Equation~\eqref{eq:normalized-equidistribution}
therefore gives $P_B\longrightarrow\mu_E$ weakly.
Combining this with \eqref{eq:total-subspace-count} proves convergence
of the normalized counting measures for subspaces of bounded height to
$\mathfrak a_{K,n,d}H(E)^{-\kappa d}\mu_E$ against every bounded
continuous function.
\end{proof}

\section{Applications}\label{sec:applications}

Recall that
\[
 y_W=([\widehat W],[\widehat{E/W}]),
\]
and let $\operatorname{dist}$ denote the metric on $\cY_{E,d}$ fixed
in Section~\ref{subsec:metrics-exponents}, so that
\[
 d_E((w,y),(w',y'))=d_\infty(w,w')+\operatorname{dist}(y,y').
\]
For $q:\cY_{E,d}\to(0,\infty)$, we say that $\log q$
is uniformly locally Lipschitz if \eqref{eq:shape-weight-regularity}
holds with $q,y,\operatorname{dist}$ in place of $Q,x,d_E$.

We apply Theorem~\ref{thm:cusp-decaying-height} to height
weights that depend only on shapes.

\begin{corollary}\label{thm:weighted-height-count}
\label{thm:shape-dependent-height-count}
Let $0<\delta<\delta_{1,0}$, and let $q:\cY_{E,d}\to(0,\infty)$
have a positive lower bound,
and assume that $\log q$ is uniformly locally Lipschitz.
Then, for every bounded Lipschitz function
$f$ on $\mathcal X_E$, as $B\to\infty$,
\begin{align}
\sum_{\substack{W\subset E,\ \dim_KW=d\\H(W)q(y_W)\le B}}
       f(x_E(W)) &=\frac{\mathfrak a_{K,n,d}}{H(E)^{\kappa d}}
       B^{\kappa n}\int_{\mathcal X_E}
       f(w,y)q(y)^{-\kappa n}\,d\mu_E(w,y)\notag\\*
 &\quad+O_{E,q,\delta}\bigl(\|f\|_{\BL}B^{\kappa n-\delta}\bigr).
 \label{eq:shape-dependent-height-count}
\end{align}
\end{corollary}

\begin{proof}
Apply Theorem~\ref{thm:cusp-decaying-height} with $\gamma=0$,
$\varphi(r,w,y)=f(w,y)$, and $Q(w,y)=q(y)$. The product metrics give
$\|\varphi\|_{\BL}=\|f\|_{\BL}$ and preserve the local Lipschitz bound
for $\log q$. The conclusion follows from
\[
 \int_0^{q(y)^{-1}}\varphi(r,w,y)\kappa n r^{\kappa n-1}\,dr
 =f(w,y)q(y)^{-\kappa n}.\qedhere
\]
\end{proof}

\begin{remark}\label{rem:shape-height-comparison}
For standard $E=\Q^n$, suppose that $q$ depends only on the
shape of the tangent lattice and satisfies the assumptions of
Corollary~\ref{thm:weighted-height-count}. The corollary recovers
the limiting count in Browning, Horesh and Wilsch
\cite[Theorem~5.8(iv), equation~(5.1)]{BrowningHoreshWilsch}
and supplies a power saving under these regularity assumptions.
Their anticanonical height is $H(W)^n$, so the height multiplier
in their formula is $q^{-n}$.
\end{remark}

The height range may also depend jointly on position and shapes.
Fix $L\ge0$, a function $u:\mathcal X_E\to\R$ with $\Lip(u)\le L$,
and $0\le\eta\le1$. Define
\[
 \Omega=\{(r,x):0<r\le1,\ |u(x)-r|\le\eta\}.
\]
The condition $(H(W)/B,x_E(W))\in\Omega$ is equivalent to
the two inequalities $H(W)\le B$ and
$|u(x_E(W))-H(W)/B|\le\eta$.
For $0<h<1$, a perturbation of size $h$ in $d_{\log,E}$ changes
$u(x)-r$ by $O_L(h)$ on $r\le e^h$. Thus, for a constant $C_L$,
\[
 \partial_h\Omega\subset
 \{e^{-h}\le r\le e^h\}
 \cup\{0<r\le e^h:\ \bigl||u(x)-r|-\eta\bigr|\le C_Lh\}.
\]
For each $x$, the second set has radial length $O_L(h)$, so
\[
 \lambda_E(\partial_h\Omega)
 \ll_{K,n,L} e^{\kappa nh}-e^{-\kappa nh}+h\ll_{K,n,L}h.
\]
Theorem~\ref{thm:boundary-modulus-count}(ii), with $\alpha=1$ and
$\gamma=0$, gives
\[
 |\mathcal E_B(\Omega,\varphi)|
 \ll_{E,L,\delta}\|\varphi\|_{\BL}B^{\kappa n-\delta}
\]
for every bounded Lipschitz $\varphi:\mathcal Z_E\to\C$,
$B\ge2$, and $0<\delta<\delta_{1,0}$, uniformly in $u$ and $\eta$.

The proof of Proposition~\ref{prop:norm-shape-stability}(i)
gives the norm comparison used below. If $L,L'\in\cX_{K,r}$ and
$d_r(L,L')<\epsilon<\epsilon_0$, they have representatives $F,F'$
on a common $K$-vector space with identical finite norms and
\begin{equation}\label{eq:application-norm-comparison}
 e^{-C\epsilon}\|z\|_{F,v}
 \le\|z\|_{F',v}
 \le e^{C\epsilon}\|z\|_{F,v}\qquad(v\mid\infty).
\end{equation}
Here $C,\epsilon_0>0$ depend only on $K,r$ and the fixed metrics.
The bounds pass to duals, subspaces and quotients; for tensor products,
the exponents add. On a subspace or quotient of rank $q$, the
corresponding heights differ by factors between $e^{-Cq\epsilon}$
and $e^{Cq\epsilon}$.

If a map $S:(X,d_X)\to(Y,d_Y)$ satisfies
$d_Y(Sx,Sx')\le Ld_X(x,x')$ whenever $d_X(x,x')<\epsilon_0$, then
for every bounded Lipschitz function $\Phi:Y\to\C$,
\begin{equation}\label{eq:local-bl-composition}
 \|\Phi\circ S\|_{\BL}
 \le(1+L+2/\epsilon_0)\|\Phi\|_{\BL}.
\end{equation}
Indeed, the local Lipschitz bound controls pairs at distance less
than $\epsilon_0$; for the remaining pairs use
$|\Phi(Sx)-\Phi(Sx')|\le2\|\Phi\|_\infty
\le2\epsilon_0^{-1}\|\Phi\|_\infty d_X(x,x')$.

\subsection{Counting with bounds on generators and defining equations}
\label{subsec:geometric-size-counts}

To count subspaces admitting independent generators of bounded
height, define
\begin{equation*}
 \cG_{E,d}(B)
 =\{W\subset E:\dim_KW=d,\ \lambda_d(W)\le B\}.
\end{equation*}
By the definition of the last Roy--Thunder minimum,
$\cG_{E,d}(B)$ consists of the $d$-dimensional subspaces spanned by
$d$ $K$-linearly independent projective points of height at most $B$.

Set
\[
\mathfrak M_{K,n,d}
 =\int_{\cX_{K,d}}\lambda_d(L)^{-\kappa nd}\,d\mu_{K,d}(L).
\]

\begin{corollary}\label{cor:generated-subspaces}
The constant $\mathfrak M_{K,n,d}$ is finite and positive.
Let $0<\delta<\delta_{1,0}$, let $f$ be bounded Lipschitz on $\mathcal X_E$, and write $x=(w,(L_1,L_2))$.
As $B\to\infty$,
\begin{align}
\sum_{W\in\cG_{E,d}(B)}f(x_E(W)) &=\frac{\mathfrak a_{K,n,d}}{H(E)^{\kappa d}}B^{\kappa nd}
 \int_{\mathcal X_E}f(x)\lambda_d(L_1)^{-\kappa nd}\,d\mu_E(x)\notag\\*
 &\quad+O_{E,\delta}\left(\|f\|_{\BL}B^{\kappa nd-d\delta}\right).
 \label{eq:generated-weighted}
\end{align}
\end{corollary}

In particular,
\begin{equation}\label{eq:generated-subspaces}
 |\cG_{E,d}(B)|
 =\frac{\mathfrak a_{K,n,d}\mathfrak M_{K,n,d}}
             {H(E)^{\kappa d}}B^{\kappa nd}
 +O_{E,\delta}(B^{\kappa nd-d\delta}).
\end{equation}

\begin{proof}
Put $q(y)=\lambda_d(L_1)^d$.
Lemma~\ref{lem:block-minima} and
\eqref{eq:log-minimum-local-lipschitz} show that $q$ is bounded below
by a positive constant and that $\log q$ is uniformly locally Lipschitz. Since
\[
 \lambda_d(W)\le B
 \iff H(W)\lambda_d(\widehat W)^d\le B^d,
\]
Corollary~\ref{thm:weighted-height-count}, applied at height bound $B^d$, gives \eqref{eq:generated-weighted}.
Taking $f=1$ and using \eqref{eq:marginals} gives
\eqref{eq:generated-subspaces}. Its leading integral is finite and
positive because the integrand is bounded and strictly positive.
\end{proof}

\begin{remark}\label{rem:generator-count-comparison}
For standard $E=\Q^n$ and $d=n-1$,
Corollary~\ref{cor:generated-subspaces} refines the order of growth
$B^{n(n-1)}$ in B\'ar\'any, Harcos, Pach and Tardos
\cite[Theorem~3]{BHPT} to an asymptotic with an explicit leading
constant and a power saving.
\end{remark}

Dividing \eqref{eq:generated-weighted} by
\eqref{eq:generated-subspaces} gives convergence of the normalized
counting measures on $\cG_{E,d}(B)$ to
\[
 \frac{\lambda_d(L_1)^{-\kappa nd}}{\mathfrak M_{K,n,d}}\,d\mu_E(x),
 \qquad x=(w,(L_1,L_2)),
\]
with error $O_{E,\delta}(\|f\|_{\BL}B^{-d\delta})$ for bounded
Lipschitz $f$ and $0<\delta<\delta_{1,0}$. The limiting measure has
archimedean marginal $\sigma_{E,d}$, since its density depends only
on $L_1$.

We next impose a bound on defining equations in addition to the
bound on generators. Let
\[
 W^\circ=\{\varphi\in E^\vee:\varphi|_W=0\}
\]
have the induced adelic norms. It is isometric to $(E/W)^\vee$.
Any $e=n-d$ $K$-linearly independent points in $W^\circ$ give
$e$ independent linear equations with common kernel $W$.
This leads to the set
\begin{equation}\label{eq:generators-equations-set}
 \mathcal C_{E,d}(B)=
 \left\{W\subset E:
     \dim_KW=d,\quad\lambda_d(W)^d\le B,\quad
     H(E)\lambda_e(W^\circ)^e\le B
     \right\}.
\end{equation}
For $y=(L_1,L_2)\in\cY_{E,d}$ let
\begin{equation}\label{eq:generators-equations-shape-weight}
 G_{d,e}(y)=
 \max\{\lambda_d(L_1)^d,\lambda_e(L_2^\vee)^e\},
 \qquad
 \mathfrak M^\circ_{E,d}
   =\int_{\cY_{E,d}}G_{d,e}(y)^{-\kappa n}\,d\nu_{E,d}(y).
\end{equation}

\begin{corollary}\label{cor:generators-equations-count}
The constant $\mathfrak M^\circ_{E,d}$ is finite and positive.
For every $0<\delta<\delta_{1,0}$ and every bounded Lipschitz function
$f$ on $\mathcal X_E$, as $B\to\infty$,
\begin{align}
\sum_{W\in\mathcal C_{E,d}(B)}f(x_E(W)) &=\frac{\mathfrak a_{K,n,d}}{H(E)^{\kappa d}}B^{\kappa n}
     \int_{\mathcal X_E}f(w,y)G_{d,e}(y)^{-\kappa n}\,d\mu_E(w,y)\notag\\*
 &\quad+O_{E,\delta}\left(\|f\|_{\BL}B^{\kappa n-\delta}\right).
 \label{eq:generators-equations-weighted}
\end{align}
\end{corollary}

In particular,
\begin{equation}\label{eq:generators-equations-scalar}
 |\mathcal C_{E,d}(B)|
 =\frac{\mathfrak a_{K,n,d}\mathfrak M^\circ_{E,d}}
        {H(E)^{\kappa d}}B^{\kappa n}
       +O_{E,\delta}(B^{\kappa n-\delta}).
\end{equation}

\begin{proof}
Since $W^\circ\simeq(E/W)^\vee$,
\[
 H(W^\circ)=H(W)/H(E),\qquad
 \widehat{W^\circ}\simeq(\widehat{E/W})^\vee.
\]
Normalization of the minima gives
\begin{equation}\label{eq:generators-equations-height-bound}
 W\in\mathcal C_{E,d}(B)
 \quad\Longleftrightarrow\quad H(W)G_{d,e}(y_W)\le B.
\end{equation}
Lemma~\ref{lem:block-minima} gives $G_{d,e}\gg_{K,n,d}1$.
The norm comparison \eqref{eq:application-norm-comparison}
also applies to dual norms. Consequently,
\[
 \log G_{d,e}(L_1,L_2)
 =\max\{d\log\lambda_d(L_1),e\log\lambda_e(L_2^\vee)\}
\]
is uniformly locally Lipschitz. We apply
Corollary~\ref{thm:weighted-height-count} with $q=G_{d,e}$.
Taking $f=1$ gives \eqref{eq:generators-equations-scalar}.
Its leading coefficient is finite and positive, since
$G_{d,e}^{-\kappa n}$ is bounded and positive.
\end{proof}

For example, take $E=\Q^{2d}$ with the standard norms and fix
$0<\delta<\delta_{1,0}$. Applying \eqref{eq:generators-equations-scalar}
at height bound $B^d$ and identifying $W^\circ$ with $W^\perp$ gives
\[
\begin{aligned}
 &\#\{W\subset\Q^{2d}:\dim_\Q W=d,\
           \lambda_d(W)\le B,\ \lambda_d(W^\perp)\le B\}\\
 &\qquad=\mathfrak a_{\Q,2d,d}\mathfrak M^\circ_{\Q^{2d},d}
             B^{2d^2}+O_{d,\delta}(B^{2d^2-d\delta}).
\end{aligned}
\]

\subsection{Slopes, freeness, and semistability}
\label{subsec:freeness}\label{subsec:maximal-slope}
\label{subsec:semistable-counts}

Let $X=\Gr_d(E)$ and let $m=d(n-d)=de$.  We equip the tangent
space
\[
 T_WX\simeq\operatorname{Hom}_K(W,E/W)=W^\vee\otimes(E/W)
\]
with the tensor product of the dual and quotient adelic structures.
At an archimedean place this is the Euclidean or Hermitian tensor
norm, for which tensor products of orthonormal bases are orthonormal.
At a finite place the unit ball is the tensor product of the two
unit ball lattices.  Over $\Q$ with the standard ambient norms, this
is the tangent lattice of \cite[\S~3]{BrowningHoreshWilsch}.

For a nonzero rigid adelic space $F$, define its normalized Arakelov
degree and slope by
\[
 \widehat{\deg}(F)=-\log H(F),\qquad
 \mu(F)=\frac{\widehat{\deg}(F)}{\dim_K F}.
\]
Its minimal and maximal slopes are
\begin{equation}\label{eq:minimal-slope-quotients}
 \mu_{\min}(F)=\inf_{V\subsetneq F}\mu(F/V),\qquad
 \mu_{\max}(F)=\sup_{0\ne U\subseteq F}
                    \frac{-\log H(U)}{\dim_K U}.
\end{equation}
Here $U,V$ run over $K$-subspaces; $U$ has the induced norms
and $F/V$ the quotient norms. The maximal slope is finite by the
lower height bounds for subspaces of each rank. By
\cite[Proposition~18]{Gaudron}, $\mu_{\min}(F)$ is the last
Harder--Narasimhan slope. Multiplying every archimedean
norm of $F$ by $c>0$ subtracts $\log c$ from every subspace and
quotient slope, and hence from both $\mu_{\max}(F)$ and $\mu_{\min}(F)$.

The tensor product height formula \cite[Proposition~5]{Gaudron} and
\eqref{eq:subspace-quotient-height} give
\[
 H(T_WX)=H(W)^{-e}H(E/W)^d=H(E)^dH(W)^{-n}.
\]
Thus the normalized logarithmic anticanonical height induced by this
metric is
\begin{equation}\label{eq:tangent-anticanonical-height}
 h(W):=\widehat{\deg}(T_WX)
      =n\log H(W)-d\log H(E).
\end{equation}
Peyre's freeness \cite[Definition~6.7]{Peyre2021} is
\begin{equation}\label{eq:freeness-definition}
 \ell(W)=
 \begin{cases}
 \displaystyle\max\left\{0,\frac{m\mu_{\min}(T_WX)}{h(W)}\right\},
       &h(W)>0,\\[2mm]
 0,&h(W)\le0.
 \end{cases}
\end{equation}
Our degrees, slopes, and logarithmic heights differ from
those in \cite{Peyre2021} by the common factor $1/\kappa$.
This factor cancels in \eqref{eq:freeness-definition}. Since
$\mu_{\min}(T_WX)\le\mu(T_WX)=h(W)/m$, one has
$0\le\ell(W)\le1$.

For $(L_1,L_2)\in\cX_{K,d}\times\cX_{K,e}$, the tensor product
$L_1^\vee\otimes L_2$ has height one.  Define
\begin{equation}\label{eq:tangent-slope-defect}
 \Delta(L_1,L_2):=-m\mu_{\min}(L_1^\vee\otimes L_2)\ge0.
\end{equation}
The inequality follows by
taking the zero subspace in \eqref{eq:minimal-slope-quotients}.
By \eqref{eq:height-one-normalization},
\[
 \widehat{T_WX}\simeq(\widehat W)^\vee\otimes\widehat{E/W}.
\]
Indeed, the tensor product on the right multiplies the archimedean
norms of $T_WX$ by
$H(W)^{1/d}H(E/W)^{-1/e}=\exp(h(W)/m)$, and leaves its finite norms
unchanged.  The scaling rule for the minimal slope therefore gives
\begin{equation}\label{eq:freeness-defect-identity}
 \ell(W)=\max\left\{0,
 1-\frac{\Delta([\widehat W],[\widehat{E/W}])}{h(W)}\right\}
 \qquad(h(W)>0).
\end{equation}

For a pair $y=(L_1,L_2)$ whose components have height one, let
\begin{equation}\label{eq:maximal-tangent-shape-slope}
 A(y)=\mu_{\max}(L_1^\vee\otimes L_2)\ge0.
\end{equation}
We also write $A(x)=A(y)$ for $x=(w,y)\in\mathcal X_E$.

\begin{lemma}\label{lem:tangent-slope-lipschitz}
\label{lem:maximal-slope-local-lipschitz}
\label{lem:tangent-defect-bound}
\begin{enumerate}[label=\textup{(\roman*)}]
\item The functions $\Delta$ and $A$ are uniformly locally Lipschitz
on $\cY_{E,d}$.
\item For every $R\ge2$ and every pair
$(L_1,L_2)\in\cX_{K,d}\times\cX_{K,e}$ satisfying
$\lambda_d(L_1)\le R$ and $\lambda_e(L_2)\le R$,
\begin{equation}\label{eq:tangent-defect-bound}
 0\le\Delta(L_1,L_2)\le d^2e\log R.
\end{equation}
\end{enumerate}
\end{lemma}

\begin{proof}
For (i), let $y=(L_1,L_2)$ and $y'=(L_1',L_2')$ with
$\operatorname{dist}(y,y')<\epsilon<\epsilon_0$.
Apply \eqref{eq:application-norm-comparison} to the two factors,
their duals and tensor products. The heights of corresponding rank $q$ subspaces and quotients of the tensor products differ by factors $e^{\pm Cq\epsilon}$. Their slopes therefore differ by at most $C\epsilon$. Taking the supremum over nonzero subspaces and the
infimum over nonzero quotients gives
\[
 |A(y)-A(y')|\le C\epsilon,\qquad
 |\Delta(y)-\Delta(y')|\le mC\epsilon.
\]
Taking the infimum over $\epsilon>\operatorname{dist}(y,y')$
proves (i).

For (ii), choose $K$-bases $x_1,\ldots,x_d$ and $z_1,\ldots,z_e$
in representatives of $L_1$ and $L_2$, respectively, with
$H(x_i),H(z_j)\le R$. Let $x_1^\vee,\ldots,x_d^\vee$ be the dual
basis. The isometry
$L_1^\vee\simeq\bigwedge^{d-1}L_1\otimes(\det L_1)^\vee$
and the exterior power norm inequality give
\[
 H(x_i^\vee)\le H(L_1)^{-1}\prod_{k\ne i}H(x_k)
 \le R^{d-1}.
\]
Thus $x_i^\vee\otimes z_j$ form a basis of $L_1^\vee\otimes L_2$
with
\[
 H(x_i^\vee\otimes z_j)=H(x_i^\vee)H(z_j)\le R^d.
\]
Every nonzero quotient $Q$ of rank $q$ has a basis selected from
the images of these vectors. Quotient norms do not increase
vector heights, and the determinant norm inequality therefore gives
\[
 H(Q)\le R^{dq},\qquad \mu(Q)\ge-d\log R.
\]
Taking the infimum over $Q$ gives
$\mu_{\min}(L_1^\vee\otimes L_2)\ge-d\log R$.
Multiplying by $-m=-de$ proves the upper bound in
\eqref{eq:tangent-defect-bound}; the tensor product has height one,
which gives the lower bound.
\end{proof}

For the anticanonical normalization with local exponents $n_v$, let
\begin{equation}\label{eq:anticanonical-multiplicative-height}
 H_{\mathrm{ac}}(W)=\exp(\kappa h(W))
       =H(W)^{\kappa n}/H(E)^{\kappa d}.
\end{equation}

\begin{corollary}\label{cor:effective-freeness}
\label{cor:freeness-count}\label{prop:low-freeness-count}
Fix $0\le\varepsilon<1$.
If $n\ge3$, let
\begin{equation}\label{eq:freeness-exponent-range}
 0<\eta<\min\left\{\delta_{1,0},
       \frac{n(1-\varepsilon)b_{d,e}}{4d^2e}\right\}.
\end{equation}
If $n=2$, let $0<\eta<\delta_{1,0}$.
\begin{enumerate}[label=\textup{(\roman*)}]
\item For every $B\ge2$,
\begin{equation}\label{eq:low-freeness-count}
 \#\{W\subset E:\dim_KW=d,\ H(W)\le B,\ \ell(W)<\varepsilon\}
 \ll_{E,\varepsilon,\eta}B^{\kappa n-\eta}.
\end{equation}
\item For every bounded Lipschitz function $f$ on $\mathcal X_E$
and every $B\ge2$,
\begin{align}
 \sum_{\substack{W\subset E,\ \dim_KW=d\\H(W)\le B}}
      \mathbf1_{\{\ell(W)\ge\varepsilon\}}f(x_E(W))
 &=\frac{\mathfrak a_{K,n,d}}{H(E)^{\kappa d}}B^{\kappa n}
             \int_{\mathcal X_E}f\,d\mu_E\notag\\*
 &\quad+O_{E,\varepsilon,\eta}
             \left(\|f\|_{\BL}B^{\kappa n-\eta}\right).
 \label{eq:effective-freeness}
\end{align}
\item For the anticanonical height $H_{\mathrm{ac}}$ in
\eqref{eq:anticanonical-multiplicative-height}, as $T\to\infty$,
\begin{align}
 &\#\{W\subset E:\dim_KW=d,\ H_{\mathrm{ac}}(W)\le T,
                                      \ \ell(W)\ge\varepsilon\}
       \notag\\*
 &\quad=\mathfrak a_{K,n,d}T+
       O_{E,\varepsilon,\eta}(T^{1-\eta/(\kappa n)}).
 \label{eq:freeness-anticanonical-count}
\end{align}
\end{enumerate}
\end{corollary}

\begin{proof}
For (i), the set is empty if $\varepsilon=0$.
If $n=2$, then $d=e=m=1$ and $\ell(W)=1$ whenever $h(W)>0$.
The remaining subspaces have $H(W)\le H(E)^{d/n}$ and form a finite
set by Northcott's theorem.

Assume $n\ge3$ and $0<\varepsilon<1$, and let
\[
 r=\frac{n(1-\varepsilon)}{4d^2e},\qquad R=B^r.
\]
For $B$ sufficiently large, $R\ge2$. If
$B^{1/2}\le H(W)\le B$ and $y_W\in\mathcal Y_R$, then
$h(W)\ge(n/2)\log B-d\log H(E)>0$ and
Lemma~\ref{lem:tangent-slope-lipschitz}(ii) gives
\[
 \Delta(y_W)\le d^2er\log B
 =\frac{n(1-\varepsilon)}4\log B<(1-\varepsilon)h(W).
\]
By \eqref{eq:freeness-defect-identity}, $\ell(W)>\varepsilon$.
Thus a subspace counted in (i) has $H(W)<B^{1/2}$ or
$y_W\notin\mathcal Y_R$. Equation~\eqref{eq:total-subspace-count}
and Theorem~\ref{thm:no-escape}(i) yield
\[
 \#\{W:\dim_KW=d,\ H(W)\le B,\ \ell(W)<\varepsilon\}
 \ll_{E,\varepsilon}B^{\kappa n/2}+B^{\kappa n-b_{d,e}r}
 \ll_{E,\varepsilon,\eta}B^{\kappa n-\eta},
\]
since $\eta<b_{d,e}r$ and
$\eta<\delta_{1,0}<\kappa n/2$. Northcott finiteness covers bounded
$B\ge2$ after enlarging the implied constant.

Subtracting the terms with $\ell(W)<\varepsilon$ from
Theorem~\ref{thm:general-equidistribution}(ii), with $\delta=\eta$,
costs $O_{E,\varepsilon,\eta}(\|f\|_\infty B^{\kappa n-\eta})$
by (i), proving (ii). For (iii), take $f=1$ and substitute
$B=H(E)^{d/n}T^{1/(\kappa n)}$, so that
\[
 H(W)\le B\iff H_{\mathrm{ac}}(W)\le T,\qquad
 H(E)^{-\kappa d}B^{\kappa n}=T,\qquad
 B^{\kappa n-\eta}\ll_E T^{1-\eta/(\kappa n)}.
\]
\end{proof}

\begin{remark}\label{rem:freeness-count-comparison}
For standard $E=\Q^n$,
formula~\eqref{eq:freeness-anticanonical-count} recovers the leading
term of Browning, Horesh and Wilsch
\cite[Theorem~1.3]{BrowningHoreshWilsch}, which counts rational
points with bounded anticanonical height and freeness at least a
fixed threshold $0\le\varepsilon<1$. The constants agree under
$H_{\mathrm{ac}}(W)=H(W)^n$ and
$c_{d,n}=\mathfrak a_{\Q,n,d}$.
\end{remark}

To describe freeness near one, define
\begin{equation}\label{eq:rescaled-freeness-defect}
 \mathcal D(W)=\max\{h(W),0\}(1-\ell(W)),
 \qquad
 \vartheta_{E,d}=\Delta_*\nu_{E,d},
\end{equation}
where $\Delta$ is defined in \eqref{eq:tangent-slope-defect}.
For $h(W)>0$, equation~\eqref{eq:freeness-defect-identity} gives
\begin{equation}\label{eq:rescaled-freeness-minimum}
 \mathcal D(W)=\min\{h(W),
                   \Delta([\widehat W],[\widehat{E/W}])\}.
\end{equation}

\begin{corollary}\label{cor:freeness-defect-distribution}
If $n\ge3$, let
\begin{equation}\label{eq:freeness-defect-exponent}
 0<\eta<\min\left\{\delta_{1,0},\frac{nb_{d,e}}{8d^2e}\right\};
\end{equation}
if $n=2$, let $0<\eta<\delta_{1,0}$.
\begin{enumerate}[label=\textup{(\roman*)}]
\item For every bounded Lipschitz function $f$ on $\mathcal X_E$
and every bounded Lipschitz function $\varphi:[0,\infty)\to\C$,
as $B\to\infty$,
\begin{align}
\frac1{N_{E,d}(B)}
   \sum_{\substack{W\subset E,\ \dim_KW=d\\H(W)\le B}}
          f(x_E(W))\varphi(\mathcal D(W)) &=\int_{\mathcal X_E}
       f(x)\varphi(\Delta(L_1,L_2))\,d\mu_E(x)\notag\\*
 &\quad+O_{E,\eta}\left(
           \|f\|_{\BL}\|\varphi\|_{\BL}B^{-\eta}\right),
 \label{eq:effective-freeness-defect}
\end{align}
where $x=(w,(L_1,L_2))$ and $[0,\infty)$ has its usual distance.
\item As $B\to\infty$, the probability measures
\[
 \frac{1}{N_{E,d}(B)}
 \sum_{\substack{W\subset E,\ \dim_KW=d\\H(W)\le B}}
 \delta_{\mathcal D(W)}
\]
converge weakly to $\vartheta_{E,d}$, where $\delta_s$ is the unit
point mass at $s$. For every $t\ge0$ with
$\vartheta_{E,d}(\{t\})=0$,
\begin{equation}\label{eq:critical-freeness-threshold}
 \frac{\#\left\{W\subset E:\begin{array}{l}
       \dim_KW=d,\ H(W)\le B,\ h(W)>t,\\
       \ell(W)\ge1-t/h(W)
       \end{array}\right\}}{N_{E,d}(B)}
 \longrightarrow\vartheta_{E,d}([0,t]).
\end{equation}
\end{enumerate}
\end{corollary}

\begin{proof}
Lemma~\ref{lem:tangent-slope-lipschitz}(i) and
\eqref{eq:local-bl-composition} give
$\|f\,\varphi\circ\Delta\|_{\BL}
\ll_{K,n,d}\|f\|_{\BL}\|\varphi\|_{\BL}$.
Apply \eqref{eq:normalized-equidistribution} to this product with
$\delta=\eta$. If $\ell(W)\ge1/2$, then
\eqref{eq:freeness-defect-identity} gives $h(W)>0$ and
$\Delta(y_W)\le h(W)/2$, hence $\mathcal D(W)=\Delta(y_W)$.
For $n\ge3$, replacing $\varphi(\Delta(y_W))$ by
$\varphi(\mathcal D(W))$ in the normalized sum therefore costs at most
\[
 2\|f\|_\infty\|\varphi\|_\infty
 \frac{\#\{W:\dim_KW=d,\ H(W)\le B,\ \ell(W)<1/2\}}
      {N_{E,d}(B)}
 \ll_{E,\eta}\|f\|_\infty\|\varphi\|_\infty B^{-\eta}
\]
by Corollary~\ref{cor:effective-freeness}(i) and
\eqref{eq:total-subspace-count}. For $n=2$, the normalized tensor
has rank one, so $\Delta=\mathcal D(W)=0$. This proves (i).

Applying the bounded Lipschitz criterion \cite[Theorem~11.3.3]{DudleyProbability} to part (i) with $f=1$ gives the weak convergence in (ii). For $h(W)>t$,
\eqref{eq:rescaled-freeness-minimum} gives
\[
 \mathcal D(W)\le t\iff\Delta(y_W)\le t
 \iff\ell(W)\ge1-t/h(W).
\]
The set $h(W)\le t$ is finite by Northcott's theorem, so the
Portmanteau theorem \cite[Theorem~5.25]{KallenbergProbability}
gives \eqref{eq:critical-freeness-threshold}.
\end{proof}

Define
\[
 \mathcal S_{E,d}(B)=
 \{W\subset E:\dim_KW=d,\ \mu_{\max}(T_WX)\le\log B\}.
\]
\begin{corollary}\label{cor:maximal-slope-count}
Let $0<\delta<\delta_{1,0}$. Then, for every bounded Lipschitz function
$f$ on $\mathcal X_E$, as $B\to\infty$,
\begin{align}
\sum_{W\in\mathcal S_{E,d}(B)}f(x_E(W)) &=\mathfrak a_{K,n,d}B^{\kappa m}
       \int_{\mathcal X_E} f(x)e^{-\kappa mA(x)}\,d\mu_E(x)\notag\\*
 &\quad+O_{E,\delta}\left(
          \|f\|_{\BL}B^{\kappa m-m\delta/n}\right).
 \label{eq:maximal-slope-weighted}
\end{align}
\end{corollary}

In particular,
\begin{equation}\label{eq:maximal-slope-scalar}
 |\mathcal S_{E,d}(B)|
 =\mathfrak C_{E,d}^{\mathrm{slope}}B^{\kappa m}
       +O_{E,\delta}(B^{\kappa m-m\delta/n}),
\end{equation}
where
\begin{equation}\label{eq:maximal-slope-constant}
 \mathfrak C_{E,d}^{\mathrm{slope}}
 =\mathfrak a_{K,n,d}
       \int_{\cY_{E,d}} e^{-\kappa mA(y)}\,d\nu_{E,d}(y),
 \qquad
 0<\mathfrak C_{E,d}^{\mathrm{slope}}\le\mathfrak a_{K,n,d}.
\end{equation}

\begin{proof}
The scaling rule for maximal slope gives
$\mu_{\max}(T_WX)=h(W)/m+A(y_W)$, hence
\begin{equation}\label{eq:maximal-slope-moving-boundary}
 W\in\mathcal S_{E,d}(B)
 \quad\Longleftrightarrow\quad
 H(W)e^{(m/n)A(y_W)}\le H(E)^{d/n}B^{m/n}.
\end{equation}
For $q=e^{(m/n)A}\ge1$,
Lemma~\ref{lem:tangent-slope-lipschitz}(i) gives the required
regularity of $\log q$. We may therefore apply
Corollary~\ref{thm:weighted-height-count} at height bound
$H(E)^{d/n}B^{m/n}$. The main term and error follow from
\[
 \begin{aligned}
 H(E)^{-\kappa d}\bigl(H(E)^{d/n}B^{m/n}\bigr)^{\kappa n}
 q^{-\kappa n}&=B^{\kappa m}e^{-\kappa mA},\\
 \bigl(H(E)^{d/n}B^{m/n}\bigr)^{\kappa n-\delta}
 &=H(E)^{\kappa d-d\delta/n}B^{\kappa m-m\delta/n}.
 \end{aligned}
\]
Taking $f=1$ gives \eqref{eq:maximal-slope-scalar}.
Since $0<e^{-\kappa mA}\le1$,
the bounds in \eqref{eq:maximal-slope-constant} follow.
\end{proof}

\begin{remark}\label{rem:maximal-slope-comparison}
For standard $E=\Q^n$,
Corollary~\ref{cor:maximal-slope-count} recovers the leading term
in Browning, Horesh and Wilsch
\cite[Theorem~1.4 and equation~(6.2)]{BrowningHoreshWilsch}
and improves their remainder $o(B^m)$ to $O(B^{m-m\delta/n})$.
\end{remark}

A height-one rigid adelic space $F$ of rank $r$ is
\emph{semistable} if
\[
 H(U)\ge1\qquad(0\ne U\subseteq F),
\]
where every subspace has the induced norms.  Equivalently,
$\mu_{\max}(F)=0$.  Let $\cX_{K,r}^{\mathrm{ss}}$ be the
semistable locus in $\cX_{K,r}$.

\begin{lemma}\label{lem:semistable-boundary}
Let $r\ge1$.
\begin{enumerate}[label=\textup{(\roman*)}]
\item The set $\cX_{K,r}^{\mathrm{ss}}$ is compact.
\item There are constants $C>0$ and $0<\epsilon_0<1$, depending
only on $K,r$ and the fixed metric, such that
\begin{equation}\label{eq:semistable-boundary-tube}
 \mu_{K,r}\{L:d_r(L,\partial\cX_{K,r}^{\mathrm{ss}})<\epsilon\}
 \le C\epsilon\qquad(0<\epsilon<\epsilon_0).
\end{equation}
\end{enumerate}
\end{lemma}

\begin{proof}
For $r=1$, the semistable locus is all of $\cX_{K,1}$. It is compact and has empty boundary, so the estimate holds.
Assume $r\ge2$. For $[F]\in\cX_{K,r}^{\mathrm{ss}}$, semistability gives $\lambda_1(F)\ge1$.
Since the minima are nondecreasing and their product is bounded
above for height one spaces by Lemma~\ref{lem:block-minima},
it follows that $\lambda_r(F)\ll_{K,r}1$.
The semistable locus is closed: in the local metric models
$\mathcal S_M$, it is the intersection of the closed conditions
$H(U)\ge1$ over all nonzero rational subspaces.  Compactness now
follows from Proposition~\ref{prop:moduli-compactness}.

For (ii), the proper discrete action defining
\eqref{eq:moduli-distance} gives local charts that are quotients
of open subsets of $\mathcal S_M$ by finite stabilizers.
Cover $\cX_{K,r}^{\mathrm{ss}}$ by finitely many such charts
with relatively compact lifts.  For a fixed module
$M$, write $V=M\otimes_{\cO_K}K$, and let $H_q(U)$ denote the
height of a $K$-subspace $U\subset V$ for the metric $q\in\mathcal S_M$
and the fixed finite norms.  On each lift, the archimedean norms
are uniformly comparable with a fixed reference metric $q_0$.
If $H_q(U)\le1$ somewhere on this lift, then
$H_{q_0}(U)\le C_0$, uniformly in $U$ and its rank.
Northcott finiteness for the Pl\"ucker height, applied in each
rank $1\le s<r$, shows that only finitely many such subspaces
occur.  Thus the semistable locus in the chart is defined by
finitely many inequalities
\[
 \varphi_U(q):=\log H_q(U)\ge0.
\]
The boundary of the semistable locus is contained in the union of the corresponding zero sets.

Each $\varphi_U$ is smooth and has nonzero differential on
$\mathcal S_M$.  To see this, write $s=\dim_K U$, choose an
archimedean place $v$, and use the orthogonal decomposition
$V_v=U_v\oplus U_v^\perp$ for the metric at the point in
question.  Multiply norms on $U_v$ by $e^t$ and on its
orthogonal complement by $e^{-st/(r-s)}$, leaving all other
places fixed.  The determinant norm of the whole space is
unchanged, so this is a path in $\mathcal S_M$, and
\[
 \left.\frac{d}{dt}\varphi_U(q_t)\right|_{t=0}
        =\frac{n_v}{\kappa}s\ne0.
\]
On relatively compact subcharts, the finitely many zero sets
therefore have neighborhoods of Riemannian volume $O(\epsilon)$,
by the implicit function theorem.  The Haar description of
$\mu_{K,r}$ gives a smooth density on the metric charts,
locally bounded above; passage to a finite stabilizer quotient
preserves this estimate.  The quotient distance is locally the
minimum of the distances between these finitely many lifts.
Covering the compact boundary by finitely many smaller charts
proves \eqref{eq:semistable-boundary-tube}.
\end{proof}

Let
\[
 \mathcal Y_{E,d}^{\mathrm{ss}}
   =\cY_{E,d}\cap
      (\cX_{K,d}^{\mathrm{ss}}\times\cX_{K,e}^{\mathrm{ss}}).
\]
\begin{corollary}\label{cor:semistable-subspace-count}
Let $0<\delta<\delta_{1,0}$. Then, for every bounded Lipschitz function
$f$ on $\mathcal X_E$, as $B\to\infty$,
\begin{align}
\sum_{\substack{W\subset E,\ \dim_KW=d,\ H(W)\le B\\
                   \widehat W,\,\widehat{E/W}\ \mathrm{semistable}}}
                  f(x_E(W)) &=\frac{\mathfrak a_{K,n,d}}{H(E)^{\kappa d}}B^{\kappa n}
   \int_{\Gr_{d,\infty}(E)\times\mathcal Y_{E,d}^{\mathrm{ss}}}
                  f\,d\mu_E\notag\\*
 &\quad+O_{E,\delta}(\|f\|_{\BL}B^{\kappa n-\delta}).
 \label{eq:semistable-weighted-count}
\end{align}
\end{corollary}

\begin{proof}
Apply Theorem~\ref{thm:boundary-modulus-count}(ii), with
$\alpha=1$, $\gamma=0$, and $\varphi(r,x)=f(x)$, to
\[
 \Omega=(0,1]\times\Gr_{d,\infty}(E)
                    \times\mathcal Y_{E,d}^{\mathrm{ss}}.
\]
For $x=(w,(L_1,L_2))$ and sufficiently small $h>0$,
\[
\begin{aligned}
 \partial_h\Omega\subset{}&
 \{(r,x):e^{-2h}\le r\le e^{2h}\}\\
 &{}\cup\{(r,x):0<r<e^{-2h},\ d_d(L_1,\partial\cX_{K,d}^{\mathrm{ss}})<4h\}\\
 &{}\cup\{(r,x):0<r<e^{-2h},\ d_e(L_2,\partial\cX_{K,e}^{\mathrm{ss}})<4h\}.
\end{aligned}
\]
Indeed, if $r<e^{-2h}$, points within distance $2h$ of $(r,x)$
have radial coordinate less than one. The definition of $\partial_h\Omega$ therefore gives an index $i\in\{1,2\}$ and a shape $L_i'$ at distance less than $2h$ from $L_i$. Exactly one of $L_i,L_i'$ is semistable. The quotient length metric gives a path
of length less than $4h$ from $L_i$ to $L_i'$, which meets the
semistable boundary.
The marginal identities \eqref{eq:marginals} and
Lemma~\ref{lem:semistable-boundary}(ii) give
\[
\begin{aligned}
 \lambda_E(\partial_h\Omega)
 &\ll_E h+\mu_{K,d}\{L:d_d(L,\partial\cX_{K,d}^{\mathrm{ss}})<4h\}\\
 &\qquad+\mu_{K,e}\{L:d_e(L,\partial\cX_{K,e}^{\mathrm{ss}})<4h\}
 \ll_E h.
\end{aligned}
\]
Finally, $\int_0^1\kappa n r^{\kappa n-1}\,dr=1$ gives
\eqref{eq:semistable-weighted-count}.
\end{proof}

\begin{remark}\label{rem:semistable-count-comparison}
Corollary~\ref{cor:semistable-subspace-count} extends to arbitrary
number fields the specialization of Horesh and Karasik
\cite[Theorem~1.1]{HoreshKarasik} in which the position is
unrestricted and both shape sets are the semistable loci.
Over $\Q$ and for $n\ge3$, their theorem gives this scalar count
with a power saving: semistability is preserved by duality, and
Lemma~\ref{lem:semistable-boundary} verifies the required boundary
condition.
\end{remark}

For standard $E=\Q^4$ and $d=e=2$, represent a unimodular rank two
lattice by $y^{-1/2}(\Z+z\Z)$, where $z=x+iy$, $|x|\le1/2$,
$|z|\ge1$, and $y>0$. Its shortest nonzero vector has length
$y^{-1/2}$, so semistability is equivalent to $y\le1$.
Integrating the invariant probability measure $(3/\pi)y^{-2}\,dx\,dy$
gives the same probability for oriented and unoriented lattices,
since this condition is invariant under reflection:
\begin{align*}
 \mu_{\Q,2}(\cX_{\Q,2}^{\mathrm{ss}})
 &=\frac3\pi\int_{-1/2}^{1/2}
           \int_{\sqrt{1-x^2}}^1y^{-2}\,dy\,dx
   =1-\frac3\pi.
\end{align*}
Since $\Pic^0(\Q)$ is trivial, the shape measure is a
product measure. Corollary~\ref{cor:semistable-subspace-count},
together with \eqref{eq:main-term-constant}, therefore gives
\begin{align}
 &\#\left\{W\subset\Q^4:\begin{array}{l}
        \dim_\Q W=2,\ H(W)\le B,\\
        \widehat W,\widehat{\Q^4/W}\ \mathrm{semistable}
        \end{array}\right\}\notag\\*
 &\qquad=\frac{15}{2\zeta(3)}
        \left(1-\frac3\pi\right)^2 B^4
          +O_\delta(B^{4-\delta})
       \qquad(0<\delta<\delta_{1,0}(\Q,4,2)).
 \label{eq:semistable-rank-two-example}
\end{align}
The limiting proportion is $(1-3/\pi)^2$, approximately
$0.2031\%$.

\subsection{Unbounded shape statistics}\label{subsec:positive-moments}

For $d=1$, the minimum $\lambda_1(\widehat W)$ is one.
For $d\ge2$ and $0<s<\kappa d$, let
\begin{equation}\label{eq:positive-minimum-moment}
 \mathfrak M^+_{K,d}(s)
 =\int_{\cX_{K,d}}\lambda_d(L)^s\,d\mu_{K,d}(L).
\end{equation}
\begin{corollary}\label{cor:positive-minimum-moments}
Assume $2\le d<n$, let $0<s<\kappa d$, and let
$0<\delta<\delta_{1,0}$. Then the integral
\eqref{eq:positive-minimum-moment} is finite and positive. Moreover, for
every bounded Lipschitz function $f$ on $\mathcal X_E$, as
$B\to\infty$,
\begin{align}
 &\frac{1}{N_{E,d}(B)}
   \sum_{\substack{W\subset E,\ \dim_KW=d\\H(W)\le B}}
   f(x_E(W))
       \left(\frac{\lambda_d(W)}{H(W)^{1/d}}\right)^s\notag\\*
 &\quad=\int_{\mathcal X_E}f(x)\lambda_d(L_1)^s\,d\mu_E(x)
  +O_{E,s,\delta}\left(
       \|f\|_{\BL}B^{-\delta(1-s/(\kappa d))}\right),
 \label{eq:positive-minimum-weighted}
\end{align}
where $x=(w,(L_1,L_2))$.
\end{corollary}
For $f=1$, the limiting integral in
\eqref{eq:positive-minimum-weighted} equals $\mathfrak M^+_{K,d}(s)$
by \eqref{eq:marginals}.

\begin{proof}
Write $Z(x)=\lambda_d(L_1)$ and $a=\kappa d$.
Equation~\eqref{eq:subspace-tail-explicit}, together with
\eqref{eq:total-subspace-count}, gives, uniformly for $R\ge2$
and all sufficiently large $B$,
\begin{equation}\label{eq:positive-moment-arithmetic-tail}
 \frac{\#\{W:\dim_KW=d,\ H(W)\le B,\ Z(x_E(W))>R\}}
      {N_{E,d}(B)}
 \ll_E R^{-a}.
\end{equation}
By \eqref{eq:normalized-equidistribution}, the empirical
probability measures converge weakly to $\mu_E$. The continuity
of $Z$ makes $\{Z>R\}$ open. The Portmanteau inequality
\cite[Theorem~5.25(ii)]{KallenbergProbability}, together with
\eqref{eq:positive-moment-arithmetic-tail}, then gives
\begin{equation}\label{eq:positive-moment-measure-tail}
 \mu_E\{Z>R\}
 =\mu_{K,d}\{L:\lambda_d(L)>R\}\ll_E R^{-a}.
\end{equation}
For $R\ge2$ define $g_R(x)=\min\{Z(x),R\}^s$.
Equations \eqref{eq:log-minimum-local-lipschitz} and
\eqref{eq:local-bl-composition}, together with the bound
\[
 \left|\min\{e^u,R\}^s-\min\{e^{u'},R\}^s\right|
 \le sR^s|u-u'|
\]
show that $\|g_R\|_{\BL}\ll_{K,d,s}R^s$.
Consequently $\|fg_R\|_{\BL}\ll_{K,d,s}R^s\|f\|_{\BL}$.
Applying \eqref{eq:normalized-equidistribution} to $fg_R$ gives
\[
 \frac{1}{N_{E,d}(B)}
 \sum_{\substack{W\subset E,\ \dim_KW=d\\H(W)\le B}}
       f(x_E(W))g_R(x_E(W)) =\int_{\mathcal X_E}fg_R\,d\mu_E
      +O_{E,s,\delta}\left(\|f\|_{\BL}R^sB^{-\delta}\right).
\]

For $z\ge0$, the layer-cake identity gives
\[
 z^s-\min\{z,R\}^s
 =s\int_R^\infty t^{s-1}\mathbf1_{\{z>t\}}\,dt.
\]
Tonelli's theorem and
\eqref{eq:positive-moment-arithmetic-tail}--\eqref{eq:positive-moment-measure-tail}
give
\[
\begin{aligned}
 &0\le{}\frac{1}{N_{E,d}(B)}
 \sum_{\substack{W\subset E,\ \dim_KW=d\\H(W)\le B}}
       \bigl(Z(x_E(W))^s-g_R(x_E(W))\bigr) +\int_{\mathcal X_E}(Z^s-g_R)\,d\mu_E\\
 &\quad\ll_E{}s\int_R^\infty t^{s-a-1}\,dt
       =\frac{s}{a-s}R^{s-a}.
\end{aligned}
\]
Since $s<a$, the moment \eqref{eq:positive-minimum-moment} is
finite; strict positivity follows from $\lambda_d>0$.
Multiplying the preceding tail estimates by $\|f\|_\infty$ controls the truncation errors in the weighted sum and integral. Combining this with the estimate for $fg_R$
gives the total error
\[
 O_{E,s,\delta}\left(\|f\|_{\BL}
       (R^sB^{-\delta}+R^{s-a})\right).
\]
Taking $R=B^{\delta/a}$ gives
$R^sB^{-\delta}=R^{s-a}=B^{-\delta(1-s/a)}$.
Together with $Z(x_E(W))=\lambda_d(W)/H(W)^{1/d}$, this proves
\eqref{eq:positive-minimum-weighted}.
\end{proof}

The preceding result weights the count by an unbounded shape statistic.
We now allow the height bound itself to grow with the cusp parameter
$\tau_E$ from \eqref{eq:cusp-size}. Let $n\ge3$,
$0\le\gamma<b_{d,e}/(\kappa n)$, and
$0<\delta<\delta_{1,\gamma}$. The identity
\[
 \log\tau_E(w,(L_1,L_2))
 =\max\{\log2,\log\lambda_d(L_1),\log\lambda_e(L_2)\}
\]
and Proposition~\ref{prop:norm-shape-stability}(i) show that
$Q=\tau_E^{-\gamma}$ satisfies the hypotheses of
Theorem~\ref{thm:cusp-decaying-height}, with $c=1$.
For a bounded Lipschitz function $f$ on $\mathcal X_E$, take
$\varphi(r,x)=f(x)$ and use
\[
 \int_0^{\tau_E(x)^\gamma}f(x)\kappa n r^{\kappa n-1}\,dr
 =f(x)\tau_E(x)^{\gamma\kappa n}
\]
to obtain, for every $B\ge2$,
\begin{align}
 \sum_{\substack{W\subset E,\ \dim_KW=d\\
                  H(W)\le B\tau_E(x_E(W))^\gamma}}
       f(x_E(W))
 &=\frac{\mathfrak a_{K,n,d}}{H(E)^{\kappa d}}B^{\kappa n}
       \int_{\mathcal X_E}f(x)\tau_E(x)^{\gamma\kappa n}\,d\mu_E(x)\notag\\*
 &\quad+O_{E,\gamma,\delta}
       \bigl(\|f\|_{\BL}B^{\kappa n-\delta}\bigr).
 \label{eq:cusp-weighted-ordering-count}
\end{align}
Set $f=1$ in \eqref{eq:cusp-weighted-ordering-count} and
normalize by the resulting count. The empirical probability
measures then converge weakly to
\begin{equation}\label{eq:cusp-weighted-ordering-law}
 \frac{\tau_E(x)^{\gamma\kappa n}}
 {\displaystyle\int_{\mathcal X_E}\tau_E(z)^{\gamma\kappa n}\,d\mu_E(z)}
 \,d\mu_E(x),
\end{equation}
with error $O_{E,\gamma,\delta}(\|f\|_{\BL}B^{-\delta})$
against bounded Lipschitz functions.

\subsection{Arakelov determinant classes and module types}
\label{subsec:steinitz-counts}

Identify $\Pic^0(K)$ with the rank-one moduli space
$\cX_{K,1}$ and equip it with the metric $d_1$ of
Section~\ref{subsec:metrics-exponents}. Set
\[
 c_E=\det([\widehat E]),\qquad
 c_W=\det([\widehat W]),\qquad
 c_{E/W}=\det([\widehat{E/W}]).
\]
Recall from \eqref{eq:determinant-line-isometry} that
\[
 c_Wc_{E/W}=c_E.
\]

\begin{corollary}\label{cor:arakelov-determinant-distribution}
Let $0<\delta<\delta_{1,0}$ and let
$\Psi:\Pic^0(K)\times\Pic^0(K)\to\C$ be bounded Lipschitz
for the sum of the two $d_1$ metrics. Then, as $B\to\infty$,
\begin{align}
\frac1{N_{E,d}(B)}
   \sum_{\substack{W\subset E,\ \dim_KW=d\\H(W)\le B}}
           \Psi(c_W,c_{E/W}) &=\int_{\Pic^0(K)}\Psi(c,c_Ec^{-1})\,
                    dm_{\Pic^0(K)}(c)\notag\\*
 &\quad+O_{E,\delta}\left(\|\Psi\|_{\BL}B^{-\delta}\right).
 \label{eq:arakelov-determinant-distribution}
\end{align}
\end{corollary}

\begin{proof}
For $d_r(L,L')<\epsilon<\epsilon_0$, choose representatives $F,F'$
as in \eqref{eq:application-norm-comparison}. For a nonzero vector
$\xi$ in their common determinant line, equality of the finite norms
and $H(F)=H(F')=1$ give
\[
 \sum_{v\mid\infty}n_v
 \log\frac{\|\xi\|_{\det F',v}}{\|\xi\|_{\det F,v}}=0.
\]
The invariant rank one metric from Section~\ref{subsec:metrics-exponents}
and \eqref{eq:application-norm-comparison} therefore give
\[
 d_1(\det L,\det L')
 \ll_K\max_{v\mid\infty}
 \left|\log\frac{\|\xi\|_{\det F',v}}{\|\xi\|_{\det F,v}}\right|
 \ll_{K,r}\epsilon.
\]
Taking the infimum over $\epsilon>d_r(L,L')$ yields
$d_1(\det L,\det L')\ll_{K,r}d_r(L,L')$.
By \eqref{eq:local-bl-composition},
$\|\Psi(\det L_1,\det L_2)\|_{\BL}
\ll_{K,n,d}\|\Psi\|_{\BL}$.
Apply \eqref{eq:normalized-equidistribution} to
$\Psi(\det L_1,\det L_2)$ and use \eqref{eq:joint-moduli-measure}
to identify the limiting integral.
\end{proof}

A fixed nontrivial continuous character $\chi:\Pic^0(K)\to S^1$
is Lipschitz: it is smooth on each of the finitely many torus
components. Take $\Psi(c_1,c_2)=\chi(c_1)$ in
Corollary~\ref{cor:arakelov-determinant-distribution}.
Since this character has zero Haar mean, we obtain, for
$0<\delta<\delta_{1,0}$, as $B\to\infty$,
\begin{equation}\label{eq:arakelov-character-cancellation}
 \frac1{N_{E,d}(B)}
   \sum_{\substack{W\subset E,\ \dim_KW=d\\H(W)\le B}}
                   \chi(c_W)
       =O_{E,\chi,\delta}(B^{-\delta}).
\end{equation}

Projecting the determinant distribution to the ideal class group
gives counts with prescribed module types. For a rigid adelic space $F$ of rank $r$, let $\mathcal M_F$ be
its $\cO_K$-lattice from \eqref{eq:integral-module} and write
\[
 \operatorname{St}(\mathcal M_F)
   =[\bigwedge_{\cO_K}^r\mathcal M_F]\in\operatorname{Cl}(K)
\]
for its Steinitz class.  By the classification recalled in
Section~\ref{subsec:metrics-exponents}, this class is trivial
if and only if $\mathcal M_F$ is free. Let
$\pi:\Pic^0(K)\to\operatorname{Cl}(K)$ be the projection that
forgets the archimedean metric. Archimedean normalization does
not change $\mathcal M_F$, so
\[
 \pi(\det[\widehat F])=\operatorname{St}(\mathcal M_F).
\]
Let $\mathfrak c_E=\operatorname{St}(\mathcal M_E)$.
Applying $\pi$ to the identity $c_Wc_{E/W}=c_E$ from
\eqref{eq:determinant-line-isometry} gives
\begin{equation}\label{eq:steinitz-product-constraint}
 \operatorname{St}(\mathcal M_W)
 \operatorname{St}(\mathcal M_{E/W})
 =\pi(c_Wc_{E/W})=\pi(c_E)=\mathfrak c_E.
\end{equation}

\begin{corollary}\label{cor:steinitz-class-count}
Let $0<\delta<\delta_{1,0}$. For every
$\mathfrak a\in\operatorname{Cl}(K)$ and every $B\ge2$,
\begin{align}
 &\#\{W\subset E:\dim_KW=d,\ H(W)\le B,
               \operatorname{St}(\mathcal M_W)=\mathfrak a\}\notag\\*
 &\qquad=\frac{\mathfrak a_{K,n,d}}
                 {h_KH(E)^{\kappa d}}B^{\kappa n}
       +O_{E,\delta}(B^{\kappa n-\delta}).
 \label{eq:steinitz-class-count}
\end{align}
For $\mathfrak b\in\operatorname{Cl}(K)$, imposing also
$\operatorname{St}(\mathcal M_{E/W})=\mathfrak b$ leaves this count
unchanged if $\mathfrak a\mathfrak b=\mathfrak c_E$; otherwise the
set is empty.
\end{corollary}

\begin{proof}
Apply Corollary~\ref{cor:arakelov-determinant-distribution} to
$\Psi(c_1,c_2)=\mathbf1_{\{\pi(c_1)=\mathfrak a\}}$.
This indicator is Lipschitz, since distinct module components
have distance one. The map $\pi$ pushes Haar probability forward to the uniform measure on $\operatorname{Cl}(K)$. Hence the integral of this indicator is $1/h_K$.
Multiplying by \eqref{eq:total-subspace-count} proves
\eqref{eq:steinitz-class-count}. The constraint
\eqref{eq:steinitz-product-constraint} forces
$\mathfrak b=\mathfrak c_E\mathfrak a^{-1}$, proving the pair count
and the assertion for incompatible pairs.
\end{proof}

Among $d$-dimensional subspaces of height at most $B$, the
proportion for which $\mathcal M_W$ is free equals
$h_K^{-1}+O_{E,\delta}(B^{-\delta})$ as $B\to\infty$.
If $\mathcal M_E$ is free, \eqref{eq:steinitz-product-constraint}
shows that the same proportion have both $\mathcal M_W$ and
$\mathcal M_{E/W}$ free. Otherwise, the two modules cannot
both be free.

\begin{remark}\label{rem:steinitz-count-comparison}
Corollary~\ref{cor:steinitz-class-count} gives, in every rank and
for a fixed rigid adelic space, the analogue of Schanuel's
uniformity among ideal classes for points of projective space
\cite[Theorem~3 and the proof of the corollary on p.~447]{Schanuel1979}.
\end{remark}

\subsection{Local distribution and approximation with prescribed shapes}
\label{subsec:shape-approximation}

Let
\[
 X_\infty=\Gr_{d,\infty}(E),\qquad
 D=\dim_{\R}X_\infty=\kappa d(n-d).
\]
Use the Riemannian metric defining $D_\infty$ in
Section~\ref{subsec:metrics-exponents}.  Choose $r_0>0$ sufficiently
small that, for every $w\in X_\infty$, the inverse exponential map
\[
 p_w=\exp_w^{-1}:U_w\longrightarrow T_wX_\infty,
 \qquad U_w=\{w':D_\infty(w,w')<r_0\},
\]
is a smooth chart.  Write $\|\cdot\|_w$ for the tangent norm and
$du_w$ for its Euclidean volume measure.  In these coordinates let
\[
 (p_w)_*(\sigma_{E,d}|_{U_w})=J_w(u)\,du_w,
 \qquad dm_w(u)=J_w(0)\,du_w.
\]
By compactness, the coordinate maps and the first derivatives
of $J_w$ are uniformly bounded on smaller chart balls.
Moreover, $J_w(0)$ is bounded above and away from zero,
uniformly over the chart centers. On
$T_wX_\infty\times\cY_{E,d}$, use the sum of the tangent distance
and the fixed shape metric.

\begin{corollary}\label{cor:joint-local-distribution}
Let $R>0$ and
\[
 0<\tau<\frac{\delta_{1,0}}{D+1},\qquad
 \tau(D+1)<\delta<\delta_{1,0},\qquad
 \eta=\min\{\tau,\delta-\tau(D+1)\}.
\]
For $w\in X_\infty$, let $\Phi:T_wX_\infty\times\cY_{E,d}\to\C$
be bounded Lipschitz, with support contained in
$\{u:\|u\|_w\le R\}\times\cY_{E,d}$.
\begin{enumerate}[label=\textup{(\roman*)}]
\item As $B\to\infty$,
\begin{equation}
\begin{aligned}
 &B^{-\kappa n+\tau D}
 \sum_{\substack{W\subset E,\ \dim_KW=d\\
                  H(W)\le B,\ W_\infty\in U_w}}
       \Phi\bigl(B^\tau p_w(W_\infty),y_W\bigr)\\
 &\quad=\frac{\mathfrak a_{K,n,d}}{H(E)^{\kappa d}}\int_{T_wX_\infty}\int_{\cY_{E,d}}
       \Phi(u,y)\,d\nu_{E,d}(y)\,dm_w(u) +O_{E,R,\tau,\delta}\bigl(\|\Phi\|_{\BL}B^{-\eta}\bigr).
\end{aligned}
 \label{eq:joint-local-distribution}
\end{equation}
\item Formula~\eqref{eq:joint-local-distribution} also holds with
$H(W)\le B$ replaced by $B/2<H(W)\le B$ in the sum and the
main term multiplied by $1-2^{-\kappa n}$.
\end{enumerate}
\end{corollary}

\begin{proof}
For sufficiently large $B$, depending only on $E,R,\tau$, define
\[
 \varphi_{B,w}(w',y)=
 \begin{cases}
 \Phi(B^\tau p_w(w'),y),&w'\in U_w,\\
 0,&w'\notin U_w
 \end{cases}
\]
Its support lies in the ball
$D_\infty(w,w')\le RB^{-\tau}<r_0/2$ and hence stays a positive
distance from the chart boundary. The uniform coordinate
bounds therefore give, also for the truncated metric $d_\infty$,
\[
 \|\varphi_{B,w}\|_{\BL}\ll_{E,R}B^\tau\|\Phi\|_{\BL}.
\]
Changing variables in the spatial integral gives
\[
\begin{aligned}
 B^{\tau D}\int_{\mathcal X_E}\varphi_{B,w}\,d\mu_E
 &=\int_{\|u\|_w\le R}\int_{\cY_{E,d}}
       \Phi(u,y)J_w(B^{-\tau}u)\,d\nu_{E,d}(y)\,du_w\\
 &=\int_{T_wX_\infty}\int_{\cY_{E,d}}
       \Phi(u,y)\,d\nu_{E,d}(y)\,dm_w(u)
       +O_{E,R}(\|\Phi\|_\infty B^{-\tau}).
\end{aligned}
\]
The error follows from
$|J_w(B^{-\tau}u)-J_w(0)|\ll_E B^{-\tau}\|u\|_w$
on $\|u\|_w\le R$.
Applying Theorem~\ref{thm:general-equidistribution}(ii) and
multiplying by $B^{-\kappa n+\tau D}$ gives the additional error
\[
 O_{E,R,\delta}\bigl(\|\Phi\|_{\BL}
                     B^{-\delta+\tau(D+1)}\bigr).
\]
This proves \eqref{eq:joint-local-distribution}.
For (ii), apply Theorem~\ref{thm:general-equidistribution}(ii) at $B$ and $B/2$ with the same test function $\varphi_{B,w}$, and subtract the resulting formulas.
The main term is multiplied by $1-2^{-\kappa n}$, and the two error
bounds above are unchanged.
\end{proof}

\Needspace{10\baselineskip}
\begin{corollary}\label{cor:approximation-prescribed-shapes}
Let $\Omega\subset\cY_{E,d}$ be open with
$\nu_{E,d}(\Omega)>0$, and let
\[
 0<\tau<\frac{\delta_{1,0}}{D+1}.
\]
Then there are a compact set $C_\Omega\subset\Omega$ and constants
$c>0$, $B_0\ge4$, depending on $E,\Omega,\tau$, such that,
for every $w\in\Gr_{d,\infty}(E)$ and every $B\ge B_0$,
\begin{align}
 &\#\left\{W\subset E:\begin{array}{l}
     \dim_KW=d,\quad B/2<H(W)\le B,\\
     D_\infty(W_\infty,w)<B^{-\tau},\quad y_W\in C_\Omega
     \end{array}\right\}
 \ge cB^{\kappa n-\tau D}.
 \label{eq:approximation-prescribed-shapes}
\end{align}
\end{corollary}

\begin{proof}
The Radon property and local compactness allow us to choose
a Lipschitz function $\psi:\cY_{E,d}\to[0,1]$ with compact
support in $\Omega$ and $\int\psi\,d\nu_{E,d}>0$.
Set $C_\Omega=\supp\psi$. On $T_wX_\infty$, define
$b_w(u)=\max\{0,1-\|u\|_w\}$ and $\Phi_w(u,y)=b_w(u)\psi(y)$.
Uniformly in $w$,
\[
 \|\Phi_w\|_{\BL}\le2+\Lip(\psi),\qquad
 \int\Phi_w\,dm_w\,d\nu_{E,d}
 \asymp_E\int\psi\,d\nu_{E,d}>0.
\]
Choose $\tau(D+1)<\delta<\delta_{1,0}$ and apply
Corollary~\ref{cor:joint-local-distribution}(ii) with $R=1$.
For all sufficiently large $B$, uniformly in $w$, it gives
\[
 \sum_{\substack{W\subset E,\ \dim_KW=d,\ W_\infty\in U_w\\
                  B/2<H(W)\le B}}
       \Phi_w(B^\tau p_w(W_\infty),y_W)
       \ge cB^{\kappa n-\tau D}.
\]
Each summand lies in $[0,1]$ and vanishes unless $y_W\in C_\Omega$ and
$D_\infty(W_\infty,w)=\|p_w(W_\infty)\|_w<B^{-\tau}$.
The lower bound therefore implies \eqref{eq:approximation-prescribed-shapes}.
\end{proof}

\begin{remark}\label{rem:huang-desaxce-local}
For standard $E=\Q^n$, taking $\Phi$ independent of $y$ in
Corollary~\ref{cor:joint-local-distribution} recovers the scalar
local limit in Huang and de Saxc\'e
\cite[Theorem~1.1]{HuangDeSaxce2025} for
$0<\tau<\delta_{1,0}/(D+1)$. Huang and de Saxc\'e prove this limit
for almost every $w\in\Gr_d(\R^n)$ and for every real algebraic
point $w$ outside the proper rational Schubert varieties.
In this range, Corollary~\ref{cor:joint-local-distribution} gives a
power saving and joint distribution with both shapes, uniformly in $w$.

Corollary~\ref{cor:approximation-prescribed-shapes} gives a
quantitative version with prescribed shapes, in a restricted exponent
range, of de Saxc\'e's Dirichlet theorem for approximation by subspaces
of the same dimension \cite[Theorem~1]{deSaxce2026}, with constants
and a height threshold uniform in $w$.
\end{remark}

\appendix
\section{Normalization of the parabolic measure}\label{app:gauge-form}

We prove Proposition~\ref{prop:parabolic-mass}.  Let
\begin{equation*}
 \mathbf P=\mathbf P_d,\qquad
 \mathbf N=\mathbf N_d,\qquad
 \mathbf M=\mathbf M_d,
\end{equation*}
and regard $a_d(t)$ as the adelic element whose archimedean components are
given by \eqref{eq:height-flow} and whose finite components are the identity.  Write
\begin{equation*}
\begin{aligned}
\mathbf P(\A_K)^1
 &=\{p\in\mathbf P(\A_K):|\chi_d(p)|_{\A_K}=1\},
 \\
\mathbf M(\A_K)^1
 &=\{m\in\mathbf M(\A_K):|\chi_d(m)|_{\A_K}=1\}.
\end{aligned}
\end{equation*}
We first compute the volume of the norm-one parabolic quotient,
\[
 \vol\bigl(\mathbf P(\A_K)^1/\mathbf P(K)\bigr)
 =\kappa n\,\mathfrak a_{K,n,d},
\]
and then identify its normalized pushforward with $\nu_{E,d}$.
For \(m\ge2\), let \(r_m(z)=\operatorname{diag}(z,I_{m-1})\), and set
\(r_1(z)=z\).  Define
\begin{equation}\label{eq:determinant-section}
 \iota(z)=\bigl(r_d(z),r_e(z^{-1})\bigr).
\end{equation}

\subsection{Measures and normalizations}

At a finite place, normalize multiplicative Haar measure by
\begin{equation*}
 \vol(\cO_v^\times)=1.
\end{equation*}
At a real place, use the product of the probability measure on \(\{\pm1\}\)
and \(d\log r\) on the positive radius.  At a complex place, use the product
of \(d\theta/(2\pi)\) on \(S^1\) and
\(d\log r^2=2\,d\log r\) on the positive radius.  Write \(d^\times z_v\) for
these local multiplicative Haar measures and
\(d^\times z=\prod_v d^\times z_v\) for the resulting idele Haar measure.
The diagonal archimedean splitting
\begin{equation*}
 t\longmapsto (e^t)_{v\mid\infty}
\end{equation*}
has idelic norm \(e^{\kappa t}\).  Define \(d^\times z^{(1)}\) on
\(\A_K^1\) by disintegrating \(d^\times z\) with respect to \(dt\).

With these choices, the norm one idele class group has volume
\begin{equation}
 \vol\bigl(\A_K^1/K^\times,d^\times z^{(1)}\bigr)
 =\frac{\kappa h_KR_K}{w_K}.
 \label{eq:norm-one-idele-class-volume}
\end{equation}
For $z=(z_v)_v\in\A_K^\times$, let $r_v=|z_v|_v$ and
\(y_v=n_v\log r_v\) at the archimedean places, and set
\[
 |z_f|_{\A_{K,f}}=\prod_{v\nmid\infty}|z_v|_v^{n_v}.
\]
The norm one condition is
\begin{equation*}
 \sum_{v\mid\infty}y_v=-\log|z_f|_{\A_{K,f}}.
\end{equation*}
For each ideal class, choose a representing finite idele $a_f$.
Modulo $K^\times$, the finite component lies in
$a_f\prod_{v\nmid\infty}\cO_v^\times$ and has norm
$|a_f|_{\A_{K,f}}$.  Fix $(c_v)_{v\mid\infty}$ such that
$\sum_{v\mid\infty}c_v=-\log|a_f|_{\A_{K,f}}$.
The translation $x_v=y_v-c_v$ identifies the corresponding affine
hyperplane with the hyperplane on which the coordinate sum is zero.  Recall the compact group
\[
 \mathcal U_K=
 \prod_{v\nmid\infty}\cO_v^\times
 \times\prod_{v\text{ real}}\{\pm1\}
 \times\prod_{v\text{ complex}}S^1.
\]
The finite valuation map and the decomposition by ideal classes give
\(h_K\) components, each represented by
\begin{equation*}
 \left(
 \mathcal U_K\times
 \left\{(x_v)_{v\mid\infty}\in\R^{s_\infty}:
        \sum_{v\mid\infty}x_v=0\right\}
 \right)\big/\cO_K^\times.
\end{equation*}
Choose one archimedean place \(v_0\) and use the coordinates
\((x_v)_{v\ne v_0}\) on this hyperplane, with
\(x_{v_0}=-\sum_{v\ne v_0}x_v\).  The resulting Lebesgue measure is the
normalization entering the regulator $R_K$.  The fixed finite and
archimedean translations preserve Haar measure.  Writing
$y_v=c_v+x_v+n_vt$ gives a change of variables of
determinant $\sum_{v\mid\infty}n_v=\kappa$.  Therefore
\(d^\times z^{(1)}\) induces \(\kappa\) times the product of Haar probability
on \(\mathcal U_K\) and the regulator measure on this hyperplane.
By Dirichlet's unit theorem, the logarithmic image of \(\cO_K^\times\)
has covolume \(R_K\) there and kernel of order \(w_K\).  Each
component therefore has mass \(\kappa R_K/w_K\), and summing over the \(h_K\)
components gives \eqref{eq:norm-one-idele-class-volume}, also for
$s_\infty=1$ with $R_K=1$ on the zero-dimensional hyperplane.

Let \(dx_v\) denote additive Haar measure on \(K_v\), normalized by
\(\vol(\cO_v)=1\) at finite places and taken to be ordinary
Lebesgue measure at archimedean places.  If, in local coordinates,
\begin{equation*}
 \omega=f(x)\,dx_1\wedge\cdots\wedge dx_r,
\end{equation*}
write
\begin{equation*}
 |\omega|_v=|f(x)|_v^{n_v}\,dx_v^{\otimes r}.
\end{equation*}
For \(m\ge2\), let \(\omega_m\) be the left invariant top form on \(\SL_m\) whose value at the identity is one on the Chevalley basis
\begin{equation*}
 (E_{ij})_{i\ne j},\qquad
 H_i=E_{ii}-E_{i+1,i+1}\quad(1\le i<m),
\end{equation*}
in any fixed order.  Let
\begin{equation*}
 ds_{m,v}=|\omega_m|_v,\qquad dg_v=ds_{n,v},
\end{equation*}
and for \(m=1\) let \(ds_{1,v}\) be unit point mass.  Under
\begin{equation}\label{eq:unipotent-matrix-coordinates}
 \mathbf N(K_v)\simeq\operatorname{Mat}_{d\times e}(K_v),\qquad
 X\longmapsto n(X),
\end{equation}
let
\begin{equation*}
 dn_v=
 \left|\bigwedge_{i=1}^d\bigwedge_{j=1}^e dX_{ij}\right|_v.
\end{equation*}
Let \(dn_{\A}\) and \(ds_{m,\A}\) denote the Tamagawa measures
associated with these gauge forms; \(dg_{\A}\) is the Tamagawa measure
fixed above.
With the additive normalization used here, the associated Tamagawa measures satisfy
\begin{equation*}
 \begin{aligned}
 dg_{\A}&=2^{r_2(n^2-1)}|D_K|^{-(n^2-1)/2}\prod_v dg_v,\\
 dn_{\A}&=2^{r_2de}|D_K|^{-de/2}\prod_v dn_v,\\
 ds_{m,\A}&=2^{r_2(m^2-1)}|D_K|^{-(m^2-1)/2}\prod_v ds_{m,v}
 \qquad(m\ge2).
 \end{aligned}
\end{equation*}

\subsection{The standard parabolic measure}

For \(m\ge1\), let
\begin{equation*}
 \mathsf K_{m,v}=
 \begin{cases}
  \SL_m(\cO_v),&v\nmid\infty,\\
  \operatorname{SO}(m),&K_v=\R,\\
  \operatorname{SU}(m),&K_v=\C.
 \end{cases}
\end{equation*}
Let \(E_0=K^n\) with its standard rigid adelic structure, let
\(\Kc_0=\prod_v\mathsf K_{n,v}\), and set
\(H_{d,0,\A}=\Kc_0\cap\mathbf P(\A_K)\).
Recall the ratios in \eqref{eq:archimedean-volume-ratios} and
\eqref{eq:zeta-ratio}:
\[
 \mathcal V_{n,d}=\frac{\prod_{j=n-d+1}^{n}V_j}{\prod_{j=2}^{d}V_j},
 \qquad
\mathcal V^{(2)}_{n,d}
 =\frac{\prod_{j=n-d+1}^{n}V_{2j}}{\prod_{j=2}^{d}V_{2j}},\qquad \mathcal Z_{K,n,d}
 =\frac{\prod_{j=n-d+1}^{n}\zeta_K(j)}{\prod_{j=2}^{d}\zeta_K(j)}.
\]

\begin{proposition}
\label{prop:standard-parabolic-measure}
Define
\begin{equation}\label{eq:standard-parabolic-coordinate-map}
 \Phi_P:\A_K^1\times\operatorname{Mat}_{d\times e}(\A_K)
 \times\SL_d(\A_K)\times\SL_e(\A_K)
 \longrightarrow\mathbf P(\A_K)^1
\end{equation}
by
\begin{equation}\label{eq:standard-parabolic-coordinate-formula}
 \Phi_P(z,X,s_d,s_e)
 =\iota(z)n(X)
  \begin{pmatrix}s_d&0\\0&s_e\end{pmatrix}.
\end{equation}
Then:
\begin{enumerate}[label=\textup{(\roman*)}]
\item \(\Phi_P\) is a homeomorphism.
\item Let $d\sigma_0$ be the $\Kc_0$-invariant probability measure on
$\Kc_0/H_{d,0,\A}$.  There is a unique Haar measure \(dp_0^1\) on
\(\mathbf P(\A_K)^1\) such that, for every
$f\in C_c(\mathbf G(\A_K))$,
\begin{equation}\label{eq:standard-global-Iwasawa}
 \int_{\mathbf G(\A_K)}f(g)\,dg_{\A}
 =\int_{\Kc_0/H_{d,0,\A}}\int_\R\int_{\mathbf P(\A_K)^1}
 f(ka_d(t)u)e^{\kappa nt}\,dp_0^1(u)\,dt\,d\sigma_0(kH_{d,0,\A}).
\end{equation}
Moreover,
\begin{equation}\label{eq:standard-parabolic-haar-measure}
\begin{aligned}
 dp_0^1
 &=2^{r_2(de+1)}|D_K|^{-(de+1)/2}
 \binom nd^{s_\infty}
 \mathcal V_{n,d}^{r_1}
 \bigl(\mathcal V^{(2)}_{n,d}\bigr)^{r_2}
 \mathcal Z_{K,n,d}^{-1}\\
 &\qquad\times(\Phi_P)_*\left(
 d^\times z^{(1)}\otimes dn_{\A}
 \otimes ds_{d,\A}\otimes ds_{e,\A}
 \right).
\end{aligned}
\end{equation}
\end{enumerate}
\end{proposition}

\begin{proof}
If
\(p=\left(\begin{smallmatrix}A&B\\0&D\end{smallmatrix}\right)
\in\mathbf P(\A_K)^1\), then
\begin{equation*}
\begin{gathered}
z=\det A,\qquad
s_d=r_d(z)^{-1}A,\\
s_e=r_e(z^{-1})^{-1}D,\qquad
X=r_d(z)^{-1}Bs_e^{-1},
\end{gathered}
\end{equation*}
which gives the inverse of \(\Phi_P\).

For a place \(v\), let \(dk_{n,v}\) be Haar probability on
\(\mathsf K_{n,v}\).  The coordinates
\begin{equation*}
 (z,X,s_d,s_e)\longmapsto
 \iota(z)n(X)\begin{pmatrix}s_d&0\\0&s_e\end{pmatrix}
\end{equation*}
identify
\(K_v^\times\times\operatorname{Mat}_{d\times e}(K_v)
\times\SL_d(K_v)\times\SL_e(K_v)\) with \(\mathbf P(K_v)\).
Conjugation by \(\iota(z)\) on the unipotent block has determinant
\(z^n\).  The local Iwasawa formula therefore gives a constant \(C_v>0\) such that, for
every $f\in C_c(\SL_n(K_v))$,
\begin{equation*}
\begin{aligned}
&\int_{\SL_n(K_v)}f(g)\,dg_v\\
&\quad=C_v\int_{\mathsf K_{n,v}}\int_{K_v^\times}
 \int_{\operatorname{Mat}_{d\times e}(K_v)}
 \int_{\SL_d(K_v)}\int_{\SL_e(K_v)}
 f\!\left(k\iota(z)n(X)
 \begin{pmatrix}s_d&0\\0&s_e\end{pmatrix}\right)\\
&\qquad\times |z|_v^{nn_v}
 \,ds_{e,v}\,ds_{d,v}\,dn_v\,d^\times z_v\,dk_{n,v},
\end{aligned}
\end{equation*}
cf.\ \cite[\S1.2]{NakamuraWatanabe}.

If \(v\nmid\infty\), let \(\mathfrak p_v\) be the maximal ideal of
\(\cO_v\) and let \(q_v=|\cO_v/\mathfrak p_v|\).  Then
\begin{equation*}
 \vol(\mathsf K_{m,v},ds_{m,v})
 =\prod_{j=2}^m(1-q_v^{-j}).
\end{equation*}
Applying the local formula to \(\mathbf1_{\mathsf K_{n,v}}\) gives
\begin{equation*}
 C_v=
 \frac{\prod_{j=2}^n(1-q_v^{-j})}
 {\bigl(\prod_{j=2}^d(1-q_v^{-j})\bigr)
  \bigl(\prod_{j=2}^e(1-q_v^{-j})\bigr)}.
\end{equation*}
At an archimedean place, we determine $C_v$ from the multiplication
map on the open cell formed by the opposite unipotent, the unipotent
radical, the two special linear blocks, and the determinant section
\(\iota(z)\).  The comparison is with the invariant probability measure
on the compact Grassmannian.  In the ordered tangent coordinates
\begin{equation*}
 \operatorname{Mat}_{e\times d}\oplus
 \operatorname{Mat}_{d\times e}\oplus
 \mathfrak{sl}_d\oplus\mathfrak{sl}_e\oplus K_v,
\end{equation*}
its differential at the identity is block triangular with identity diagonal
blocks.  On the diagonal Lie algebra the extra generator is
\begin{equation*}
 E_{11}-E_{d+1,d+1}=H_1+\cdots+H_d;
\end{equation*}
together with the simple coroot bases of \(\mathfrak{sl}_d\) and
\(\mathfrak{sl}_e\), this is an integral basis of determinant \(\pm1\).
The differential has Jacobian \(1\) in these coordinates.  The
remaining scalar is therefore determined by the compact quotient.

For the rank one real quotient, the affine chart of
\(\mathbf P^{m-1}(\R)\) has invariant probability density
\begin{equation*}
 \frac{\Gamma(m/2)}{\pi^{m/2}}
 (1+\|x\|^2)^{-m/2}\,dx.
\end{equation*}
At the identity of the positive component of \(\R^\times\), the determinant
normalization above gives \(d^\times a=\tfrac12\,da\).  The rank one
scalar is therefore
\begin{equation*}
 2\frac{\pi^{m/2}}{\Gamma(m/2)}=mV_m.
\end{equation*}
Similarly, on \(\mathbf P^{m-1}(\C)\) the invariant probability density is
\begin{equation*}
 \frac{(m-1)!}{\pi^{m-1}}(1+\|x\|^2)^{-m}\,d^{2m-2}x,
\end{equation*}
and at \(a=1\in\C^\times\) our normalization gives
\(d^2a=\pi\,d^\times a\).  The rank one
scalar is therefore
\(mV_{2m}\).

To pass from the rank one quotient to the Grassmannian, use the space of pairs \((L,W)\) with \(L\) a line contained in a \(d\)-dimensional subspace \(W\).  If \(C_{m,r}\) denotes the local scalar for the
quotient parametrizing \(r\)-dimensional subspaces in rank \(m\), then for \(d\ge2\), integrating
first over \(W\) and then over the lines in \(W\), or first over \(L\)
and then over the \((d-1)\)-dimensional subspaces in the quotient, gives
\[
 C_{n,d}\,C_{d,1}=C_{n,1}\,C_{n-1,d-1}.
\]
Together with the rank one case, iteration yields
\[
 C_{n,d}
 =\frac{\displaystyle\prod_{m=n-d+1}^{n}C_{m,1}}
        {\displaystyle\prod_{m=2}^{d}C_{m,1}}.
\]
The determinant measure $d^\times z_v$ accounts for the real sign or
complex phase of the one-dimensional factor. Substituting the
rank one scalars gives
\begin{equation*}
 \frac{\displaystyle\prod_{m=n-d+1}^{n}mV_m}
      {\displaystyle\prod_{m=2}^{d}mV_m}
 =\binom nd\mathcal V_{n,d},
 \qquad
 \frac{\displaystyle\prod_{m=n-d+1}^{n}mV_{2m}}
      {\displaystyle\prod_{m=2}^{d}mV_{2m}}
 =\binom nd\mathcal V^{(2)}_{n,d}.
\end{equation*}
\begin{equation*}
 C_v=
 \begin{cases}
  \displaystyle\binom nd\mathcal V_{n,d},&K_v=\R,\\[2mm]
  \displaystyle\binom nd\mathcal V^{(2)}_{n,d},&K_v=\C.
 \end{cases}
\end{equation*}

Taking the product over all places gives
\begin{equation*}
 \prod_{v\nmid\infty}C_v=\mathcal Z_{K,n,d}^{-1},
 \qquad
 \prod_{v\mid\infty}C_v
 =\binom nd^{s_\infty}\mathcal V_{n,d}^{r_1}
 \bigl(\mathcal V^{(2)}_{n,d}\bigr)^{r_2},
\end{equation*}
and the finite product converges absolutely.

The Jacobian calculation above shows that
\begin{equation*}
 (\Phi_P)_*\left(
 d^\times z^{(1)}\otimes dn_{\A}
 \otimes ds_{d,\A}\otimes ds_{e,\A}\right)
\end{equation*}
is a right Haar measure on \(\mathbf P(\A_K)^1\).  The factor
\(|z|_{\A_K}^{n}\) is one on \(\A_K^1\).  Since the modular character of
\(\mathbf P(\A_K)\) is
\(u\mapsto|\chi_d(u)|_{\A_K}^{n}\), the parabolic subgroup of norm one is
unimodular.

Write
\begin{equation*}
 \epsilon_\infty(t)_v=\begin{cases}e^t,&v\mid\infty,\\1,&v\nmid\infty.
 \end{cases}
\end{equation*}
Every \(z\in\A_K^\times\) is uniquely the product of \(\epsilon_\infty(t)\) and an
element of \(\A_K^1\), where
\(t=\kappa^{-1}\log|z|_{\A_K}\), and
\(d^\times z=d^\times z^{(1)}\,dt\).  Moreover,
\(\iota(\epsilon_\infty(t))a_d(-t)\in
\SL_d(\A_K)\times\SL_e(\A_K)\).
Absorbing this element into the Levi factor and changing variables in
the unipotent factor preserve the respective Haar measures.
Since \(|\epsilon_\infty(t)|_{\A_K}^{n}=e^{\kappa nt}\), taking restricted
products of the local formulas gives
\eqref{eq:standard-global-Iwasawa} with the product of the local
gauge form measures. For the Tamagawa normalization, compare
\cite[Notations and \S1.1]{NakamuraWatanabe}, where the additive
measure has covolume $|D_K|^{1/2}$ on $\A_K/K$ and the complex unit
disc has mass $2\pi$; the Tamagawa normalization for an
$r$-dimensional group uses the factor $|D_K|^{-r/2}$ and the Artin
$L$-factor. Replacing their complex additive measure by ordinary Lebesgue measure (unit disc of
mass $\pi$) multiplies the product of local gauge-form measures by
$2^{-rr_2}$ and changes the additive covolume to
$2^{-r_2}|D_K|^{1/2}$. The normalization factors satisfy
\[
 \bigl(2^{-r_2}|D_K|^{1/2}\bigr)^{-r}2^{-rr_2}
 =|D_K|^{-r/2}.
\]
The groups $\mathbf G,\mathbf N,\SL_d,\SL_e$ have no nontrivial
rational characters, so their Artin factors are trivial. For each of
these groups of dimension $r$, the Tamagawa measure is thus
$\bigl(2^{-r_2}|D_K|^{1/2}\bigr)^{-r}$ times the product of our local
gauge-form measures. The resulting scalar is
\begin{equation*}
\begin{aligned}
 &\left(2^{-r_2}|D_K|^{1/2}\right)^{de+(d^2-1)+(e^2-1)-(n^2-1)}
 \prod_vC_v\\
 &\qquad=2^{r_2(de+1)}|D_K|^{-(de+1)/2}
 \binom nd^{s_\infty}\mathcal V_{n,d}^{r_1}
 \bigl(\mathcal V^{(2)}_{n,d}\bigr)^{r_2}
 \mathcal Z_{K,n,d}^{-1}.
\end{aligned}
\end{equation*}
Equations~\eqref{eq:standard-global-Iwasawa} and
\eqref{eq:standard-parabolic-haar-measure} follow.
\end{proof}

Let \(d\bar p_0\) be the quotient of \(dp_0^1\) by \(\mathbf P(K)\).
Apply Weil's integration formula \cite[Theorem~2.51 and (2.52)]{FollandHarmonic} to
\begin{equation}\label{eq:parabolic-Levi-exact-sequences}
\begin{gathered}
 1\to\mathbf N(\A_K)\to\mathbf P(\A_K)^1\to\mathbf M(\A_K)^1\to1,\\
 1\to\SL_d(\A_K)\times\SL_e(\A_K)
 \to\mathbf M(\A_K)^1\xrightarrow{\chi_d}\A_K^1\to1,
\end{gathered}
\end{equation}
On $\mathbf M(\A_K)^1$ the action on $\mathbf N(\A_K)$ has adelic modulus
$|\chi_d(m)|_{\A_K}^{n}=1$.  Tamagawa number one for
\(\mathbf G_a^{de}\) \cite[Chapter~II]{Weil} and for the special linear
factors \cite{Langlands}, together with \eqref{eq:norm-one-idele-class-volume}, therefore
gives
\begin{equation}\label{eq:standard-parabolic-quotient-volume}
\begin{aligned}
&\vol\bigl(\mathbf P(\A_K)^1/\mathbf P(K),d\bar p_0\bigr)\\
&\quad=\frac{\kappa h_KR_K}{w_K}
 2^{r_2(de+1)}|D_K|^{-(de+1)/2}
 \binom nd^{s_\infty}\mathcal V_{n,d}^{r_1} 
 \bigl(\mathcal V^{(2)}_{n,d}\bigr)^{r_2}
 \mathcal Z_{K,n,d}^{-1}
 =\kappa n\,\mathfrak a_{K,n,d}.
\end{aligned}
\end{equation}

\subsection{The adelic quotient and the measure comparison}

On the quotient $\mathbf G(\A_K)/\mathbf P(K)$, write
$x=g\mathbf P(K)$ and equip $K^n$ with the local norms
$y\mapsto\|g_vy\|_{E,v}$.  Let $W_g$ denote $W_0$ with the induced adelic structure, and endow
$K^n/W_0$ with the quotient norms.  Define
\begin{equation}\label{eq:Xd-invariants}
\begin{gathered}
t_E(x)=\log H(W_g),\qquad
\pi_{E,\infty}(x)=(g_vW_{0,v})_{v\mid\infty},\\
\pi_{E,\mathcal Y}(x)=([\widehat W_g],[\widehat{K^n/W_0}]).
\end{gathered}
\end{equation}
$\pi_{E,\mathcal Y}$ takes values in $\cY_{E,d}$ by the determinant relation.
The three maps are well defined because right multiplication by $\mathbf P(K)$ changes
the induced structures only by $K$-linear changes of basis. We use the same
notation for their pullbacks under the quotient map
$\mathbf G(\A_K)\to\mathbf G(\A_K)/\mathbf P(K)$.

For $k\in\Kc_{E,\infty}$, $t\in\R$, and $p\in P_{d,\infty}^1$,
the norm calculation in Proposition~\ref{prop:sector-decomposition} gives
\begin{equation}\label{eq:global-sector-compatibility}
 \begin{aligned}
 t_E((k a_{d,E}(t-\tau_j)p,\gamma_{j,f})\mathbf P(K))&=t,\\
 \pi_{E,\infty}((k a_{d,E}(t-\tau_j)p,\gamma_{j,f})\mathbf P(K))&=kW_0,\\
 \pi_{E,\mathcal Y}((k a_{d,E}(t-\tau_j)p,\gamma_{j,f})\mathbf P(K))&=\psi_j(p).
 \end{aligned}
\end{equation}

Let $d\bar g$ be the Weil quotient of $dg_{\A}$ by counting measure on
$\mathbf P(K)$.  For every nonnegative left $\Kc_{E,f}$-invariant Borel function $\varphi$ on
$\mathbf G(\A_K)/\mathbf P(K)$, the decomposition into finite adelic double cosets gives
\begin{equation}\label{eq:finite-class-unfolding}
 \int_{\mathbf G(\A_K)/\mathbf P(K)}\varphi(x)\,d\bar g(x)
 =\vol(\Kc_{E,f})\sum_{j=1}^J
 \int_{G/\Gamma_{P,j}}
 \varphi\bigl((g,\gamma_{j,f})\mathbf P(K)\bigr)\,d\bar g_j(g).
\end{equation}
The $j$-th finite adelic class is the image of
$G\times\Kc_{E,f}$ under
\begin{equation*}
 (g,k_f)\longmapsto (g,k_f\gamma_{j,f})\mathbf P(K),
\end{equation*}
and two pairs have the same image precisely when they differ by the right
action
\begin{equation*}
 (g,k_f)\cdot p
 =\bigl(gp_\infty,
        k_f\gamma_{j,f}p_f\gamma_{j,f}^{-1}\bigr)
 \qquad(p\in\Gamma_{P,j}).
\end{equation*}
The fiber over $G/\Gamma_{P,j}$ is $\Kc_{E,f}$, of Haar volume
$\vol(\Kc_{E,f})$. Weil's quotient formula therefore gives
\eqref{eq:finite-class-unfolding}.

\begin{proof}[Proof of Proposition~\ref{prop:parabolic-mass}]
Let $t_E$, $\pi_{E,\infty}$, and
$\pi_{E,\mathcal Y}$ be as in \eqref{eq:Xd-invariants}, and let \(d\bar g\) be the quotient measure
fixed above.  For nonnegative Borel functions $\psi$ on $\R$, $\phi$ on
$\Gr_{d,\infty}(E)$, and $\xi$ on $\cY_{E,d}$, set
\begin{equation*}
 \mathcal I(\psi,\phi,\xi)
 =\int_{\mathbf G(\A_K)/\mathbf P(K)}
 \psi(t_E(x))\phi(\pi_{E,\infty}(x))\xi(\pi_{E,\mathcal Y}(x))\,d\bar g(x).
\end{equation*}
The integrand is left $\Kc_{E,f}$-invariant. Applying
\eqref{eq:finite-class-unfolding}, \eqref{eq:full-sector-quotient},
and \eqref{eq:global-sector-compatibility}, and pushing $dk$
forward to $\sigma_{E,d}$, gives
\begin{equation}\label{eq:mass-sector-factor}
 \begin{split}
 \mathcal I(\psi,\phi,\xi)
 &=\left(\int_\R\psi(t)e^{\kappa nt}\,dt\right)
  \left(\int_{\Gr_{d,\infty}(E)}\phi\,d\sigma_{E,d}\right)\\
 &\qquad\times
 \vol(\Kc_{E,f})\sum_{j=1}^J e^{-\kappa n\tau_j}
 \int_{Q_j}\xi(\psi_j(\bar p))\,d\bar p_j.
 \end{split}
\end{equation}

Let \(\Kc_E=\prod_v\Kc_{E,v}\) and
\(H_{d,E,\A}=\Kc_E\cap\mathbf P(\A_K)\), and let \(d\sigma_E\) be the
\(\Kc_E\)-invariant probability measure on \(\Kc_E/H_{d,E,\A}\).  The adelic Iwasawa integration formula determines a
unique Haar measure \(dp_E^1\) on \(\mathbf P(\A_K)^1\) such that, for every
$f\in C_c(\mathbf G(\A_K))$,
\begin{equation}\label{eq:global-Iwasawa-E}
 \int_{\mathbf G(\A_K)}f(g)\,dg_{\A}
 =\int_{\Kc_E/H_{d,E,\A}}\int_\R\int_{\mathbf P(\A_K)^1}
 f(ka_d(s)u)e^{\kappa ns}\,dp_E^1(u)\,ds\,d\sigma_E(kH_{d,E,\A}).
\end{equation}
The normalization is the one compared with \cite[\S1.2]{NakamuraWatanabe}.
Let \(d\bar p_E\) be
the quotient of \(dp_E^1\) by \(\mathbf P(K)\).  The finiteness of this quotient measure follows from the two exact sequences in \eqref{eq:parabolic-Levi-exact-sequences}, the compactness of \(\A_K^1/K^\times\), and the finite-volume Tamagawa quotients of the unipotent and special linear factors.

Let \(\tau_0=\log H(W_0)\).  If the upper left block of
\(u\in\mathbf P(\A_K)^1\) is \(A\), then
\begin{equation*}
 H(W_u)=H(W_0)|\det A|_{\A_K}^{1/\kappa}=e^{\tau_0}.
\end{equation*}
The factor \(a_d(s)\) multiplies this height by \(e^s\). The element \(k\in\Kc_E\) acts isometrically.  Hence
\(\log H(W_{ka_d(s)u})=s+\tau_0\), while
\(\pi_{E,\mathcal Y}(ka_d(s)u)=\pi_{E,\mathcal Y}(u)\).  Applying Weil's formula to
\eqref{eq:global-Iwasawa-E} and changing variables
\(t=s+\tau_0\) gives
\begin{equation}\label{eq:mass-global-factor}
\begin{aligned}
\mathcal I(\psi,\phi,\xi)
 &=e^{-\kappa n\tau_0}
  \left(\int_\R\psi(t)e^{\kappa nt}\,dt\right)
  \left(\int_{\Gr_{d,\infty}(E)}\phi\,d\sigma_{E,d}\right)\\
&\quad\times \int_{\mathbf P(\A_K)^1/\mathbf P(K)}
 \xi(\pi_{E,\mathcal Y}(u))\,d\bar p_E(u).
\end{aligned}
\end{equation}
Since \(H(E)=1\), comparing \eqref{eq:mass-sector-factor} and
\eqref{eq:mass-global-factor} with \(\phi\equiv1\) and a nonnegative
\(\psi\) satisfying \(\int_\R\psi(t)e^{\kappa nt}\,dt=1\) gives
\begin{equation}\label{eq:mass-bridge}
 \vol(\Kc_{E,f})\sum_{j=1}^J e^{-\kappa n\tau_j}
 (\psi_j)_*d\bar p_j
 =e^{-\kappa n\tau_0}(\pi_{E,\mathcal Y})_*d\bar p_E.
\end{equation}

Recall that $\Pic^0(K)=\mathcal U_K\backslash\A_K^1/K^\times$.
The probability measure $\mu_{K,m}$ on $\cX_{K,m}$
has conditional probability $\mu_{K,m,c}$ on the determinant
fiber $\det^{-1}(c)$, defined for almost every $c\in\Pic^0(K)$ with respect to Haar measure.
The measure to be identified is
\[
 \nu_{E,d}=\int_{\Pic^0(K)}
 \mu_{K,d,c}\otimes\mu_{K,e,c_Ec^{-1}}\,dm_{\Pic^0(K)}(c),
 \qquad c_E=\det([\widehat E]),
\]
as defined in \eqref{eq:joint-moduli-measure}.

First take \(E=E_0=K^n\).  Then \(\tau_0=0\) and \(d\bar p_{E_0}=d\bar p_0\).
For
\begin{equation*}
 \mathbf P(\A_K)^1/\mathbf P(K)
 \longrightarrow\mathbf M(\A_K)^1/\mathbf M(K),
\end{equation*}
the fiber over \(m\mathbf M(K)\) is
\(\mathbf N(\A_K)/(m\mathbf N(K)m^{-1})\).  Since conjugation by
\(m\in\mathbf M(\A_K)^1\) has modulus
\(|\chi_d(m)|_{\A_K}^{n}=1\), Tamagawa number one for
\(\mathbf G_a^{de}\) gives volume one to every fiber.  By
\eqref{eq:standard-parabolic-quotient-volume}, the probability measure
\((\kappa n\,\mathfrak a_{K,n,d})^{-1}d\bar p_0\) pushes forward to
Haar probability on the Levi quotient.

Normalize $d^\times z^{(1)}$ on $\A_K^1/K^\times$ by the total mass in
\eqref{eq:norm-one-idele-class-volume}.  Its pushforward to $\Pic^0(K)$ is the Haar
probability $m_{\Pic^0(K)}$.  Disintegrate with respect to this map.
Recall that $\Kc_m$ is the product of the isometry groups of the standard
local norms in $\GL_m$, so $r_m(\mathcal U_K)\subset\Kc_m$.
Weil's formula for
\begin{equation*}
 1\to\SL_d(\A_K)\times\SL_e(\A_K)
 \to\mathbf M(\A_K)^1\xrightarrow{\chi_d}\A_K^1\to1
\end{equation*}
and Tamagawa number one for the special linear factors show that, after
choosing $z\in\A_K^1$ whose idele class maps to
\(c\in\Pic^0(K)\), the conditional measure on the special linear
factors is the product of Haar probabilities on
\(\SL_d(\A_K)/\SL_d(K)\) and \(\SL_e(\A_K)/\SL_e(K)\).  After quotienting on the left by the compact groups, consider the
pushforwards under
\begin{equation*}
\begin{aligned}
s_d\SL_d(K)&\longmapsto
 \Kc_d r_d(z)s_d\GL_d(K),\\
s_e\SL_e(K)&\longmapsto
 \Kc_e r_e(z^{-1})s_e\GL_e(K).
\end{aligned}
\end{equation*}
By uniqueness of disintegration almost everywhere, these are
\(\mu_{K,d,c}\) and \(\mu_{K,e,c^{-1}}\), respectively, independently
of the choice of $z$: replacing $z$ by $azu$ with
$a\in K^\times$ and $u\in\mathcal U_K$ changes the maps by left
multiplication by the respective compact factors $r_d(u),r_e(u^{-1})$
and conjugation by $r_d(a),r_e(a^{-1})$, which preserve normalized
Haar measure on the special linear quotients. The
conditional measure over \(c\) is therefore
\begin{equation*}
 \mu_{K,d,c}\otimes\mu_{K,e,c^{-1}}.
\end{equation*}
Integrating over \(c\) gives
\begin{equation*}
 (\pi_{E_0,\mathcal Y})_*\left(
 \frac{d\bar p_0}{\kappa n\,\mathfrak a_{K,n,d}}
 \right)
 =\nu_{E_0,d}.
\end{equation*}
Together with \eqref{eq:mass-bridge}, this gives
\eqref{eq:parabolic-mass-W} for $E_0$.

Now let \(E\) be an arbitrary height one rigid adelic space.  Choose \(g_E\in\GL_n(\A_K)\) such that
\(\|x\|_{E,v}=\|g_{E,v}x\|_{n,v}\), and let
\begin{equation*}
 z=\det g_E,\qquad b=r_n(z).
\end{equation*}
Since \(H(E)=1\), \(|z|_{\A_K}=1\), and
$b^{-1}g_E\in\mathbf G(\A_K)$.  Let \(E_b\) be the rigid adelic space
represented by \(b\).  Let $T=b^{-1}g_E$.  Left translation by $T$ preserves $d\bar g$.
It satisfies $t_{E_b}(Tx)=t_E(x)$ and
$\pi_{E_b,\mathcal Y}(Tx)=\pi_{E,\mathcal Y}(x)$, while
$\pi_{E_b,\infty}(Tx)=T_\infty\pi_{E,\infty}(x)$.
The map $T_\infty$ is an isometry from $E_\infty$ to $(E_b)_\infty$.
Applying \eqref{eq:mass-sector-factor} with $\phi=1$ to an arbitrary
nonnegative \(\xi\) shows that the left hand side of
\eqref{eq:parabolic-mass-W} is the same for $E$ and $E_b$.
Moreover, \(c_E=c_{E_b}\), so \(\nu_{E,d}=\nu_{E_b,d}\).

For \(E_b\), \(\tau_0=\kappa^{-1}\log|z|_{\A_K}=0\).  Conjugation by \(b\)
preserves the Tamagawa measure on \(\mathbf G(\A_K)\), sends the compact
factor for \(E_b\) to the standard one, fixes \(a_d(t)\), and preserves
\(\mathbf P(\A_K)^1\).  Comparing \eqref{eq:global-Iwasawa-E} with
\eqref{eq:standard-global-Iwasawa} gives
\begin{equation*}
 (\operatorname{Ad}b)_*dp_{E_b}^1=dp_0^1.
\end{equation*}
With \(b_d=r_d(z)\), the coordinate formula
\eqref{eq:standard-parabolic-coordinate-formula} gives, for every
$z'\in\A_K^1$,
\begin{equation*}
 b\,\Phi_P(z',X,s_d,s_e)\,b^{-1}
 =\Phi_P(z',b_dX,b_ds_db_d^{-1},s_e).
\end{equation*}
The change \(X\mapsto b_dX\) has adelic modulus
\(|z|_{\A_K}^{e}=1\), and conjugation by \(b_d\) preserves
\(ds_{d,\A}\).  Formula \eqref{eq:standard-parabolic-haar-measure} therefore
gives
\begin{equation*}
 (\operatorname{Ad}b)_*dp_0^1=dp_0^1.
\end{equation*}
Comparing with the preceding identity yields \(dp_{E_b}^1=dp_0^1\), and
the quotient measures agree as well.

Here \(b_d\in\GL_d(\A_K)^1\), and the class of \(\det b_d\) in \(\Pic^0(K)\) is \(c_E\).  Let
\(\widetilde\mu_{d,c}\) be the conditional measures obtained by disintegrating
Haar probability on \(\GL_d(\A_K)^1/\GL_d(K)\) through
\begin{equation*}
 g\GL_d(K)\longmapsto\mathcal U_K(\det g)K^\times\in\Pic^0(K).
\end{equation*}
The quotient map
\begin{equation*}
 g\GL_d(K)\longmapsto\Kc_dg\GL_d(K)
\end{equation*}
commutes with the determinant maps, so uniqueness of disintegration gives
\begin{equation*}
 (g\GL_d(K)\mapsto\Kc_dg\GL_d(K))_*\widetilde\mu_{d,c}
 =\mu_{K,d,c}
\end{equation*}
for \(m_{\Pic^0(K)}\)-almost every \(c\).  Left translation by \(b_d\)
preserves Haar probability and sends determinant class \(c\) to \(c_Ec\).
Since \(m_{\Pic^0(K)}\) is translation invariant, uniqueness of
disintegration gives
\begin{equation*}
 (b_d\cdot)_*\widetilde\mu_{d,c}
 =\widetilde\mu_{d,c_Ec}
 \quad\text{for }m_{\Pic^0(K)}\text{-a.e. }c.
\end{equation*}
For \(m_{\Pic^0(K)}\)-almost every \(c\), choose
$z_c\in\A_K^1$ whose idele class lies over \(c\).  The pushforward of
Haar probability on
\(\SL_d(\A_K)/\SL_d(K)\) under
\begin{equation*}
 s_d\SL_d(K)\longmapsto
 \Kc_db_dr_d(z_c)s_d\GL_d(K)
\end{equation*}
is independent of $z_c$ by the same argument as in the standard
case, since $b_d=r_d(z)$ commutes with $r_d(u)$ for $u\in\mathcal U_K$.
Disintegration over the fiber gives \(\mu_{K,d,c_Ec}\), while the second
factor remains \(\mu_{K,e,c^{-1}}\).  Therefore
\begin{equation*}
 (\pi_{E_b,\mathcal Y})_*\left(
 \frac{d\bar p_0}{\kappa n\,\mathfrak a_{K,n,d}}
 \right)
 =\int_{\Pic^0(K)}
 \mu_{K,d,c_Ec}\otimes\mu_{K,e,c^{-1}}\,dm_{\Pic^0(K)}(c)
 =\nu_{E,d},
\end{equation*}
by the change of variable \(c\mapsto c_Ec\).  Since \eqref{eq:mass-bridge} for \(E_b\) has \(\tau_0=0\), we obtain
\eqref{eq:parabolic-mass-W} for $E_b$, and the preceding comparison gives
\eqref{eq:parabolic-mass-W} for $E$.
\end{proof}

\Addresses


\begin{thebibliography}{BHPT01}
\small
\raggedright
\setlength{\itemsep}{2.5pt plus 1pt minus 1pt}

\bibitem[AAW26]{AkaAndritschWieser}
M. Aka, K. Andritsch and A. Wieser, Planes in quadratic 4-space and
associated shapes of lattices. Preprint (2026),
\href{https://arxiv.org/abs/2606.03472v2}{arXiv:2606.03472v2}.

\bibitem[AES16]{AkaEinsiedlerShapira}
M. Aka, M. Einsiedler and U. Shapira, Integer points on spheres and their
orthogonal lattices. \emph{Invent. Math.} \textbf{206} (2016), no.~2,
379--396.

\bibitem[AEW22]{AkaEinsiedlerWieser}
M. Aka, M. Einsiedler and A. Wieser, Planes in four-space and four
associated {CM} points. \emph{Duke Math. J.} \textbf{171} (2022), no.~7,
1469--1529.

\bibitem[AMW24]{AkaMussoWieser}
M. Aka, A. Musso and A. Wieser, Equidistribution of rational subspaces and their shapes. \emph{Ergodic Theory Dynam. Systems} \textbf{44} (2024), no.~8, 2009--2062.

\bibitem[BG06]{BombieriGubler}
E. Bombieri and W. Gubler, \emph{Heights in {D}iophantine geometry}. New Mathematical Monographs, vol.~4. Cambridge University Press, Cambridge, 2006.

\bibitem[BH62]{BorelHC}
A. Borel and Harish-Chandra, Arithmetic subgroups of algebraic groups. \emph{Ann. of Math. (2)} \textbf{75} (1962), no.~3, 485--535.

\bibitem[BHPT01]{BHPT}
I. B{\'a}r{\'a}ny, G. Harcos, J. Pach and G. Tardos, Covering lattice points by subspaces. \emph{Period. Math. Hungar.} \textbf{43} (2001), no.~1--2, 93--103.

\bibitem[BHW22]{BrowningHoreshWilsch}
T. Browning, T. Horesh and F. Wilsch, Equidistribution and freeness
on Grassmannians. \emph{Algebra Number Theory} \textbf{16} (2022), no.~10,
2385--2407.

\bibitem[Bor63]{BorelAdeles}
A. Borel, Some finiteness properties of adele groups over number fields. \emph{Publ. Math. Inst. Hautes {\'E}tudes Sci.} \textbf{16} (1963), 5--30.

\bibitem[CHH88]{CowlingHaagerupHowe}
M. Cowling, U. Haagerup and R. Howe, Almost $L^2$ matrix coefficients.
\emph{J. Reine Angew. Math.} \textbf{387} (1988), 97--110.

\bibitem[Cow23]{CowlingDecay}
M.~G. Cowling, Decay estimates for matrix coefficients of unitary
representations of semisimple {Lie} groups.
\emph{J. Funct. Anal.} \textbf{285} (2023), no.~8, Paper No.~110061.

\bibitem[CP23]{ChintaPereira}
G. Chinta and V. Pereira J\'unior, Planes in $\mathbb Z^4$ and
{Eisenstein} series. \emph{J. Lond. Math. Soc. (2)} \textbf{108} (2023),
no.~6, 2415--2435.

\bibitem[dS26]{deSaxce2026}
N. de Saxc\'e, Rational approximations to linear subspaces.
\emph{Ann. of Math. (2)} \textbf{203} (2026), no.~3, 737--780.
\href{https://doi.org/10.4007/annals.2026.203.3.2}{doi:10.4007/annals.2026.203.3.2}.

\bibitem[Dud02]{DudleyProbability}
R.~M. Dudley, \emph{Real analysis and probability}.
Cambridge Studies in Advanced Mathematics, vol.~74.
Cambridge University Press, Cambridge, 2002.

\bibitem[ERW19]{EinsiedlerRuehrWirth}
M. Einsiedler, R. R\"uhr and P. Wirth, Distribution of shapes of orthogonal
lattices. \emph{Ergodic Theory Dynam. Systems} \textbf{39} (2019), no.~6,
1531--1607.

\bibitem[FS10]{FiekerStehle}
C. Fieker and D. Stehl\'e, Short bases of lattices over number fields.
In \emph{Algorithmic number theory (ANTS-IX)}
(G. Hanrot, F. Morain and E. Thom\'e, eds.), Lecture Notes in Computer
Science, vol.~6197, Springer, Berlin, 2010, pp.~157--173.
\href{https://perso.ens-lyon.fr/damien.stehle/downloads/OKLatRed.pdf}{Revised author version, 16 October 2023}.

\bibitem[FMT89]{FrankeManinTschinkel}
J. Franke, Y.~I. Manin and Y. Tschinkel, Rational points of bounded
height on {Fano} varieties. \emph{Invent. Math.} \textbf{95} (1989),
no.~2, 421--435. Erratum, \emph{Invent. Math.} \textbf{102} (1990),
no.~2, 463.

\bibitem[Fol16]{FollandHarmonic}
G.~B. Folland, \emph{A course in abstract harmonic analysis},
2nd ed., Textbooks in Mathematics, CRC Press, Boca Raton, FL, 2016.

\bibitem[Gau21]{Gaudron}
{\'E}. Gaudron, Chapter {II}: Minima and slopes of rigid adelic spaces. In \emph{Arakelov geometry and {D}iophantine applications} (Emmanuel Peyre and Ga{\"e}l R{\'e}mond, eds.), Lecture Notes in Mathematics, vol.~2276, pp.~37--76. Springer, Cham, 2021.

\bibitem[GGN13]{GhoshGorodnikNevo}
A. Ghosh, A. Gorodnik and A. Nevo, Diophantine approximation and
automorphic spectrum. \emph{Int. Math. Res. Not. IMRN} (2013),
no.~21, 5002--5058.

\bibitem[GN12]{GorodnikNevo}
A. Gorodnik and A. Nevo, Counting lattice points. \emph{J. Reine Angew. Math.} \textbf{663} (2012), 127--176.



\bibitem[Hen02]{Henk2002}
M. Henk, Successive minima and lattice points. \emph{Rend. Circ. Mat. Palermo (2) Suppl.} (2002), no.~70, part I, 377--384.

\bibitem[HK19]{HoreshKarasikVectors}
T. Horesh and Y. Karasik, Equidistribution of primitive vectors, and the
shortest solutions to their {GCD} equations. Preprint (2019),
arXiv:1903.01560.

\bibitem[HK22]{HoreshKarasikFlags}
T. Horesh and Y. Karasik, Counting flags of primitive lattices.
Preprint (2022), arXiv:2202.12724.

\bibitem[HK23]{HoreshKarasik}
T. Horesh and Y. Karasik, Equidistribution of primitive lattices in $\mathbb R^n$. \emph{Q. J. Math.} \textbf{74} (2023), no.~4, 1253--1294.

\bibitem[HP24]{HoreshPaulin}
T. Horesh and F. Paulin, Joint effective equidistribution of partial
lattices in positive characteristic. Preprint (2024),
arXiv:2404.04368.

\bibitem[HdS25]{HuangDeSaxce2025}
Z. Huang and N. de Saxc\'e, Local distribution of rational points in
flag varieties. Preprint (2025),
\href{https://arxiv.org/abs/2510.09954}{arXiv:2510.09954}.

\bibitem[Kal21]{KallenbergProbability}
O. Kallenberg, \emph{Foundations of modern probability}, 3rd ed.
Probability Theory and Stochastic Modelling, vol.~99.
Springer, Cham, 2021.

\bibitem[Kim23]{KimGrassmannian}
S. Kim, Counting rational points of a {Grassmannian}.
\emph{Monatsh. Math.} \textbf{201} (2023), no.~3, 825--864.

\bibitem[Kna86]{KnappRepresentation}
A.~W. Knapp, \emph{Representation theory of semisimple groups:
An overview based on examples}.
Princeton Mathematical Series, vol.~36, Princeton University Press,
Princeton, NJ, 1986.

\bibitem[Lan66]{Langlands}
R.~P. Langlands, The volume of the fundamental domain for some arithmetical
              subgroups of {Chevalley} groups. In \emph{Algebraic groups and discontinuous subgroups} (Armand Borel and George~D. Mostow, eds.), Proc. Sympos. Pure Math., vol.~9, pp.~143--148. Amer. Math. Soc., Providence, RI, 1966.

\bibitem[McF71]{McFeat}
R.~B. McFeat, Geometry of numbers in adele spaces. \emph{Dissertationes Math. (Rozprawy Mat.)} \textbf{88} (1971), 1--49.

\bibitem[Mor15]{MorrisArithmetic}
D. Witte Morris, \emph{Introduction to arithmetic groups}.
Deductive Press, 2015.

\bibitem[NW04]{NakamuraWatanabe}
Y. Nakamura and T. Watanabe, The normalization constant of a certain invariant measure on
              $\mathrm{GL}_n(D_{\A})$. \emph{Manuscripta Math.} \textbf{115} (2004), no.~3, 259--280.

\bibitem[Oh02]{OhMatrix}
H. Oh, Uniform pointwise bounds for matrix coefficients of unitary
 representations and applications to {Kazhdan} constants. \emph{Duke Math. J.} \textbf{113} (2002), no.~1, 133--192.

\bibitem[Pey21]{Peyre2021}
E. Peyre, Chapter {V}: Beyond heights: slopes and distribution of rational
points. In \emph{Arakelov geometry and {D}iophantine applications}
(Emmanuel Peyre and Ga{\"e}l R{\'e}mond, eds.), Lecture Notes in Mathematics,
vol.~2276, pp.~215--279. Springer, Cham, 2021.

\bibitem[Pey95]{Peyre1995}
E. Peyre, Hauteurs et mesures de {Tamagawa} sur les vari\'et\'es de {Fano}.
\emph{Duke Math. J.} \textbf{79} (1995), no.~1, 101--218.

\bibitem[PR94]{PlatonovRapinchuk}
V. Platonov and A. Rapinchuk, \emph{Algebraic groups and number theory}. Pure and Applied Mathematics, vol.~139. Academic Press, Boston, 1994.

\bibitem[RT96]{RoyThunder}
D. Roy and J.~L. Thunder, An absolute {Siegel}'s lemma. \emph{J. Reine Angew. Math.} \textbf{476} (1996), 1--26. Addendum and erratum, J. Reine Angew. Math. 508 (1999), 47--51.

\bibitem[Sch68]{Schmidt1968}
W.~M. Schmidt, Asymptotic formulae for point lattices of bounded determinant and
              subspaces of bounded height. \emph{Duke Math. J.} \textbf{35} (1968), no.~2, 327--339.

\bibitem[Sch79]{Schanuel1979}
S.~H. Schanuel, Heights in number fields. \emph{Bull. Soc. Math. France} \textbf{107} (1979), 433--449.

\bibitem[Sch98]{Schmidt1998}
W.~M. Schmidt, The distribution of sublattices of $\mathbb Z^m$. \emph{Monatsh. Math.} \textbf{125} (1998), no.~1, 37--81.

\bibitem[Thu92]{Thunder1992}
J.~L. Thunder, An asymptotic estimate for heights of algebraic subspaces. \emph{Trans. Amer. Math. Soc.} \textbf{331} (1992), no.~1, 395--424.

\bibitem[Thu93]{Thunder1993}
J.~L. Thunder, Asymptotic estimates for rational points of bounded height on flag varieties. \emph{Compos. Math.} \textbf{88} (1993), no.~2, 155--186.

\bibitem[Wat03]{Watanabe2003}
T. Watanabe, The {Hardy--Littlewood} property of flag varieties.
\emph{Nagoya Math. J.} \textbf{170} (2003), 185--211.

\bibitem[Wei82]{Weil}
A. Weil, \emph{Adeles and algebraic groups}. Progress in Mathematics, vol.~23. Birkh{\"a}user, 1982.

\bibitem[Wid10]{Widmer2010}
M. Widmer, Counting primitive points of bounded height.
\emph{Trans. Amer. Math. Soc.} \textbf{362} (2010), no.~9, 4793--4829.

\end{thebibliography}
\end{document}